\documentclass[10pt,a4paper]{article}

\usepackage{amsmath,amsthm,amssymb}
\usepackage[T1]{fontenc}
\usepackage[utf8]{inputenc}

\usepackage{hyperref} 

\usepackage[svgnames,x11names]{xcolor} 
\hypersetup{
    colorlinks, linkcolor={DarkOrange4},
    citecolor={DarkOliveGreen4}, urlcolor={DarkOliveGreen4}
}

\usepackage{nicefrac}

\usepackage{enumerate}
\usepackage{bm}
\usepackage{multicol}

\usepackage{bbm}

\newtheorem{tw}{Theorem}[section]
\newtheorem{lm}[tw]{Lemma}
\newtheorem{wn}[tw]{Corollary}
\newtheorem{pr}[tw]{Proposition}
\theoremstyle{definition}
\newtheorem{uw}[tw]{Remark}
\newtheorem{df}{Definition}[section]
\newtheorem{qu}{Question}

\newcommand{\R}{\mathbb{R}}
\newcommand{\Z}{\mathbb{Z}}
\newcommand{\N}{\mathbb{N}}
\newcommand{\T}{\mathbb{T}}

\newcommand{\cS}{\mathcal{S}}
\newcommand{\C}{\mathbb{C}}
\newcommand{\Q}{\mathbb{Q}}
\newcommand{\cA}{\mathcal{A}}
\newcommand{\cB}{\mathcal{B}}
\newcommand{\cC}{\mathcal{C}}
\newcommand{\cF}{\mathcal{F}}
\newcommand{\cP}{\mathcal{P}}
\newcommand{\cT}{\mathcal{T}}
\newcommand{\cJ}{\mathcal{J}}
\newcommand{\cU}{\mathcal{U}}
\newcommand{\ca}{\mathcal{A}}

\newcommand{\cf}{\mathcal{F}}
\newcommand{\cj}{\mathcal{J}}
\newcommand{\cV}{\mathcal{V}}
\newcommand{\cR}{\mathcal{R}}
\newcommand{\raz}{\mathbbm{1}}

\newcommand{\xbm}{(X,{\cal B},\mu)}
\newcommand{\ycn}{(Y,{\cal C},\nu)}
\newcommand{\zdr}{(Z,{\cal D},\rho)}
\newcommand{\vep}{\varepsilon}

\newcommand{\gr}[1]{\underline{\bm{#1}}}

\makeatletter
\def\blfootnote{\xdef\@thefnmark{}\@footnotetext}
\makeatother

\title{On embeddability of automorphisms into measurable flows from the point of view of self-joining properties (extended version)}
\author{Joanna Ku\l aga-Przymus}
\begin{document}

\bibliographystyle{abbrv}
\maketitle

\setcounter{tocdepth}{3}

\begin{abstract}
We compare self-joining- and embeddability properties. In particular, we prove that a measure preserving flow $(T_t)_{t\in\R}$ with $T_1$ ergodic is 2-fold quasi-simple (2-fold distally simple) if and only if $T_1$ is 2-fold quasi-simple (2-fold distally simple). We also show that the Furstenberg-Zimmer decomposition for a flow $(T_t)_{t\in\R}$ with $T_1$ ergodic with respect to any flow factor is the same for $(T_t)_{t\in\R}$ and for $T_1$. We give an example of a 2-fold quasi-simple flow disjoint from simple flows and whose time-one map is simple. We describe two classes of flows (flows with minimal self-joining property and flows with the so-called Ratner property) whose time-one maps have unique embeddings into measurable flows. We also give an example of a 2-fold simple flow whose time-one map has more than one embedding.
\end{abstract}

\tableofcontents

\section{Introduction}

\subsection{Embeddability}

The problem of embeddability of automorphisms into measurable flows has been studied in ergodic theory for more than 60 years. Even today,  the basic problem -- a necessary and sufficient condition for embeddability -- remains the most interesting open question in this area. The first articles on this subject dealt with a simpler problem, namely the one of the existence of roots. In~\cite{MR0005802} Halmos gave a necessary condition for the existence of a square root (the absence of $-1$ in the spectrum). In the discrete spectrum case this  is also sufficient. These results were later generalized in various ways -- both in the discrete~\cite{MR0095239,MR0204612,MR0215959,MR660821} and quasi-discrete spectrum~\cite{MR0230877,MR0229798,MR0244460,MR0285692} case. They include necessary and sufficient conditions for the existence of roots and for the existence of embedding into a flow in the discrete and quasi-discrete spectrum case.\footnote{Hahn and Parry~\cite{MR0230877} showed that automorphisms with quasi-discrete, non-discrete spectrum are not embeddable into flows.\label{foot2}} In particular, Lamperti~\cite{MR0095239} gave an example of a discrete spectrum automorphism with roots of all orders, which is however not embeddable into a flow. 

The problem turned out to be much more challenging in the weakly mixing case. Already in 1956, Halmos~\cite{MR0097489} (\cite{MR0111817}) asked about the existence of square roots of weakly mixing automorphisms, the existence of roots for Bernoulli shifts and the embeddability of Bernoulli shifts into flows. The answers came relatively fast. Chacon constructed a non-mixing automorphism with continuous spectrum having no square root~\cite{MR0207954} and a bit later one having no roots at all~\cite{MR0212158}.
Since the centralizer of an automorphism embeddable into a flow contains all time-$t$ automorphisms of this flow, a necessary condition for embeddability is that the centralizer of the considered automorphism is uncountable. Del Junco~\cite{MR531323}, continuing the study of the Chacon's type constructions, showed, in particular, that the centralizer of the classical Chacon's automorphism~\cite{MR0247028} consists only of its powers, whence it is too small for the automorphism even to admit roots. Further results were provided by Ornstein -- he answered positively the remaining two questions of Halmos on the Bernoulli embedding~\cite{MR0257322,MR0272985}. While all previous constructions were not mixing, and Bernoulli shifts have strong mixing properties, it seemed natural that sufficiently strong mixing properties could imply embeddability (in 1966 this problem was still open~\cite{MR0204612}). This, however, turned out to be false. Ornstein consecutively constructed a mixing automorphism~\cite{MR0399415} and a K-automorphism~\cite{MR0330416} without square roots. 

After that there seems to have been no substantial progress until 2000 when King~\cite{MR1784212} showed that a generic automorphism admits roots of all orders. Three years later de la Rue and de Sam Lazaro proved even more -- that a generic automorphism can be embedded into a flow~\cite{MR1959844}. This result was later generalized (in various ways) to the case of $\Z^d$-actions~\cite{MR2041884,MR2230134,MR2261246}. 
In 2004 Stepin and Eremenko~\cite{MR2138483} showed that typically there are uncountably many pairwise spectrally non-isomorphic embeddings, thus strengthening the result from~\cite{MR1959844}.

\subsection{Joining properties}\label{selfj1}
Joinings were introduced by Furstenberg~\cite{MR0213508} in 1967 and have been a very fruitful tool for studying the dynamical systems since then.  A joining of two (or more) systems (in our case automorphisms or flows) is a measure invariant under the product action, whose marginals are the invariant measures of the original systems. We deal with self-joinings, i.e.\ joinings of copies of a fixed system. Depending on the ``number of ways'' a system can be joined with itself in an ergodic way, one defines several classes of dynamical systems. We will be mostly interested in 2-fold joinings. The self-joining properties of higher order are defined using the so-called PID property~\cite{MR922364}.\footnote{Danilenko showed~\cite{MR2346554} that $\cT$ has the PID property if and only if $T_1$ has this property.} We distinguish systems with \emph{minimal self-joining property} (MSJ) (Rudolph~\cite{MR555301}), \emph{2-fold simple} systems (Veech~\cite{MR685378}, del Junco, Rudolph \cite{MR922364}), \emph{2-fold quasi-simple} (2-QS) systems (Ryzhikov, Thouvenot~\cite{MR2265691}), \emph{2-fold distally simple} (2-DS) systems (del Junco, Lema\'{n}czyk \cite{delJ-Lem}). A particular case of the 2-QS systems are systems whose 2-fold ergodic self-joinings other than the product measure are finite extensions of the marginal factors. We will denote them by ``$n:1$''  for $n\geq 1$. For $n=1$ this notion is equivalent to 2-fold simplicity. We also consider \emph{joining primeness property} (JP) (Lema\'{n}czyk, Parreau, Roy~\cite{MR2729082}). Contrary to the previous classes, systems with the JP property were originally not defined in terms of self-joinings -- instead, some restriction was imposed on the type of ergodic joinings with Cartesian products of weakly mixing systems. Precise definitions are given in Section~\ref{selfj2} and there is the following relation between the mentioned classes: $\text{MSJ}\subset \text{2-fold simple}\subset \text{2-QS}\subset \text{2-DS}\subset \text{JP}$.

One more property we will consider is the so-called \emph{R-property}, which was first observed by Ratner in~\cite{MR717825} for horocycle flows. It describes the behavior of orbits of nearby distinct points. Roughly speaking, sufficiently long pieces of their orbits are close to each other, up to a time shift. An important ingredient of the definition of the R-property is the restriction on how large this time shift can be. Precise conditions (see~\cite{MR717825,MR2237466,MR2753947,FKa}) form a variety of combinatorial properties, all yielding the same dynamical consequences in case of weakly mixing flow: each of them implies that the flow is ``$n:1$'' for some $n\in\N$.\footnote{Another version of R-property can be found in~\cite{MR796188} -- here however the notion was investigated from other point of view.} In fact it implies more: ergodic joinings with other flows, other than the product measure, are also finite extensions of the coordinate factor of the ``additional'' flow~\cite{MR717825} (see also~\cite{MR1325699}). Apart from the horocycle flows, several classes of flows satisfying the R-property are known, see~\cite{MR2237466,MR2753947,Kanigowski:2013fk,FKa}.

\subsection{Motivation and results}
\subsubsection{Between simplicity and JP: 2-QS and 2-DS}
A starting point for our investigations is a result of del Junco and Rudolph \cite{MR896795}  who showed that whenever $\cT=(T_t)_{t\in\R}$ (with $T_1$ ergodic) is 2-fold simple then $T_1$ is also 2-fold simple. On the other hand, it is not difficult to prove that given a flow $\cT=(T_t)_{t\in\R}$ (with $T_1$ ergodic),  $\cT$ has the JP property whenever $T_1$ has the JP property (see Section~\ref{JPsekcja}). Analogous results hold also for automorphisms and their powers. 

Since the 2-QS and 2-DS properties are ``intermediate'' between simplicity and the JP property, it is reasonable to expect that the 2-QS and 2-DS property will share either the features of the simple systems or the features of the JP systems, or both. Moreover, it is quite easy to see that if $T_1$ is 2-fold simple then $\cT$ is 2-QS. This makes it even less surprising that indeed either both $\cT$ and $T_1$ are 2-QS (2-DS) or none of them is:

\begin{tw}\label{ainthm1}
Let $\cT=(T_t)_{t\in\R}$ be an ergodic flow. The following conditions are equivalent:
\begin{itemize}
\item[(i)]
$\cT$ is 2-DS (2-QS).
\item[(ii)]
$T_t$ is 2-DS (2-QS) for $t\in\R$ such that $T_t$ is ergodic.
\item[(iii)]
There exists $t_0\in\R$ such that $T_{t_0}$ is ergodic and $T_{t_0}$ is 2-DS (2-QS).
\end{itemize}
\end{tw}
Therefore, even though rather few explicit examples of 2-QS and 2-DS systems are known, these notions seem very natural. 
We first provide the proof of Theorem~\ref{ainthm1} for the 2-QS property. Then, building on the obtained results, we pass to the 2-DS property (see Section~\ref{pierwszydowod} and~\ref{drugidowod}). An essential role is played by lemmas which describe the relation between an extension of a flow and the corresponding extension of its time-one map. Rougly speaking, these extensions are of the same character (e.g.\ either both relatively distal or both relatively weakly mixing, for the definitions see Section~\ref{se:extensions}). Our main tool is the so-called T-compactness property~\cite{MR1930996}. In Appendix~\ref{apB}, we provide another, more direct proof of Theorem~\ref{ainthm1} in case of 2-DS property. It uses the notion of so-called \emph{separating sieves}~\cite{MR0234443} which helps to avoid some technical difficulties.

\subsubsection{Mixing properties}
Recall the following classical result:
\begin{pr}
Let $\cT$ be a measurable flow. Then $\cT$ is weakly mixing (mildly mixing, lightly mixing, partially mixing, mixing, mixing of order $n$, rigid, partially rigid) if and only if $T_1$ has this property.
\end{pr}
In course of proving Theorem~\ref{ainthm1}, as a by-product, we show that a ``relative version'' of the above proposition is true in case of weak mixing.
In fact, we show more than this. In order to state the result, we need to recall first a classical theorem of Furstenberg and Zimmer (valid also for flows and other group actions):
\begin{tw}[Furstenberg-Zimmer]\label{fz}
Let $T$ be an ergodic automorphism on $\xbm$ with a factor $\cA\subset \cB$. Then there exists a unique intermediate factor $\cA\subset \cC\subset \cB$ such that the extension $ \cB\to \cC$ is relatively weakly mixing and $\cC\to\cA$ is relatively distal.
\end{tw}
The factorization from the above theorem is called \emph{Furstenberg-Zimmer decomposition}. We show that for a flow and for its ergodic time-$t$ automorphism this decomposition  is the same with respect to any flow factor:\footnote{Again, an analogous result holds for automorphisms and their powers.}
\begin{tw}\label{ainthm2}
Let $\cT=(T_t)_{t\in\R}\colon (X,\mathcal{B},\mu)\to (X,\mathcal{B},\mu)$ be a flow with $T_1$ ergodic and let $\cA\subset \cB$ be $\cT$-invariant.
Then the Furstenberg-Zimmer decomposition for $\cT$ and for $T_1$ with respect to $\cA$ is the same. In particular, $\cT$ is relatively weakly mixing (relatively distal) over $\cA$ if and only if $T_1$ is relatively weakly mixing (relatively distal) over $\cA$.
\end{tw}

\subsubsection{(Non-)uniqueness of the embedding}
Rudolph \cite{MR0414826} gave an example of two non-isomorphic K-automorphisms with isomorphic squares (see also \cite{MR0444907}). Also, Stepin and Eremenko~\cite{MR2138483} showed that for $T$ with simple spectrum, which is embeddable into a measurable flow, the embedding is either unique or $T$ has infinitely many spectrally non-isomorphic embeddings. In a typical situation the set of possible embeddings (or roots of a given order) has cardinality continuum and any two different members are not spectrally isomorphic.

In Section~\ref{uniq}, we discuss the problem of uniqueness of the embedding from the point of view of self-joining properties. We show that for time-one map of a flow with MSJ property the embedding is unique. A similar phenomenon can be observed for time-one maps of flows having the R-property, in particular for the horocycle flows. For flows with the MSJ property we provide two independent proofs. The first of them can be easily adapted to yield a similar result for roots. The second proof can be seen as a special case of a more general argument. Namely, we show that given a weakly mixing flow $\cT$, a sufficient condition for the uniqueness of the embedding for $T_1$ is that any ergodic joining of $\cT$ with a weakly mixing flow $\cS$ remains weakly mixing.  This is also the core of the argument we apply to the horocycle flows (using the theory of Ratner~\cite{MR721735,MR657240,MR717825}).

Except for the Bernoulli shifts (both in finite and infinite entropy case) \cite{MR0257322,MR0274716,MR0272985,MR0318452}, ergodic automorphisms being time-$t$ maps of MSJ flows or flows with the R-property\footnote{Recall that both flows with the MSJ and R-property are of zero entropy.} seem to be the first concrete examples of automorphisms with a unique embedding into a measurable flow.

Finally, we pass to ``negative'' examples. They are given in terms of compact group extensions of 2-fold simple systems. To construct a weakly mixing 2-fold simple automorphism which has two non-isomorphic square roots, we use the dihedral group of order $6$ (one of these roots is 2-fold simple, the other one is only 2-QS). To give an example of a weakly mixing automorphism which is 2-fold simple and embeddable into two non-isomorphic flows we use the group $SU(2)$ (again, one of the obtained flows is 2-fold simple, the other one is only 2-QS).\footnote{Contrary to what was claimed in~\cite{MR0229798}, it is possible to have non-uniqueness of embedding and roots also in the quasi-discrete spectrum case, see also footnote~\ref{foot2}. To see this, consider an irrational rotation on the circle $Tx=x+\alpha$. It has countably many non-isomorphic embeddings: $T_{n,t}x=x+t(\alpha+n)$. This yields immediately also countably many non-isomorphic roots for $T$. In a similar way, we can obtain non-isomorphic roots in the case of quasi-discrete, non-discrete spectrum. For $T(x,y)=(x+\alpha,x+y)$ we have $R\in C(T)$ for $R(x,y)=(x,y+\nicefrac12)$, whence $T^2=(TR)^2$.} 
Another natural class of automorphisms, which is of a different nature, consists of Gaussian automorphisms.\footnote{It was shown in~\cite{MR2729082} that JP systems are disjoint from the Gaussian systems. Other related results can be found in~\cite{MR1325699,MR1687223,1112.5545}. We refer the reader e.g. to~\cite{MR832433} or~\cite{MR0272042} for an introduction to Gaussian systems.}

\subsubsection{Counter-examples}
In Section~\ref{se5}, we focus on various types of counter-examples related to Theorem~\ref{ainthm1}. The examples mentioned at the end of the previous section show immediately that 2-fold simplicity of an (ergodic) time-one map of a flow does not imply that also the whole flow is 2-fold simple. Similarly, 2-fold simplicity of $T^2$ does not imply 2-fold simplicity of $T$. In other words, Theorem~\ref{ainthm1} is ``the most'' we can hope for, and the aforementioned result of del Junco and Rudolph~\cite{MR922364} cannot be reversed. In fact, already Danilenko~\cite{MR2346554} constructed an automorphism $T$ such that $T^2$ is 2-fold simple and $T$ is only ``$2:1$''. The key tool is the so-called (C, F)-construction procedure, which is an algebraic counterpart of the well-known cutting and stacking technique.\footnote{This method is a useful tool for producing examples and counter-examples of different types of behavior. For a detailed survey, see~\cite{MR2402408}.} A general idea behind this type of constructions is to use a larger (usually non-abelian) group with some special algebraic properties, which are reflected in the dynamics of a properly chosen sub-action. The example given in~\cite{MR2346554} actually has even stronger properties than we need: the constructed automorphism is 2-QS (it is ``2:1''), it is disjoint from simple systems and its square is 2-fold simple. Notice each of the actions constructed in the previous section has a factor which is 2-fold simple. In particular, they are not disjoint from 2-fold simple systems. On the other hand, these constructions are much simpler than the example from~\cite{MR2346554}.

The next example we consider has the same basic properties: $T^2$ is 2-fold simple whereas $T$ is only 2-QS. It is given in terms of a double group extension of an automorphism with the MSJ property (so it is again much simpler than the example from~\cite{MR2346554}, but not disjoint from 2-fold simple automorphisms). The second group extension in our construction is very explicit: we use the affine cocycle. This is however not necessarily a strong point: it is unclear how to adjust this example to obtain a 2-QS flow with a 2-fold simple time-one map.

In Section~\ref{se52}, using the construction from~\cite{MR2346554} as the starting point, we provide an example of a flow $\cT$ with $T_2$ ergodic and 2-fold simple, such that $\cT$ is only 2-QS (and not 2-fold simple). Moreover, the obtained flow has a factor with the same self-joining properties as the original flow and disjoint from 2-fold simple flows. To prove this, we show that given a 2-fold simple flow $\cT$ on $(X,\cB,\mu)$  with $T_t$ ergodic and a flow factor $\cA\subset \cB$, either both $\cT|_{\cA}$ and $T_t|_{\cA}$ are 2-fold simple or none of them is. In particular, the horocycle flows and their factors cannot serve as an example here. In our (C,F)-construction we deal with uncountable non-abelian groups. This results is additional technical difficulties related to equidistribution of sequences, cf.\ \cite{MR2402408,MR3227151} -- since this part of the paper is of a different flavor than the rest of it, we include necessary tools in Appendix~\ref{se:ud}.

\section{Preliminaries}\label{def+not}
We deal with measure-preserving automorphisms and flows, i.e. $\Z$- and $\R$-actions on standard Borel spaces, preserving a probability measure. In particular, we assume that all flows are measurable, i.e. the map $X\times\R\ni (x,t)\mapsto T_tx\in X$ is measurable. Sometimes the invariant measure is not fixed and we work with a standard Borel space $(X,\cB)$ (the existence of an invariant measure will be always guaranteed). The properties which are of our interest are invariant under measure-theoretical isomorphisms. Therefore,  without loss of generality, we can make a tacit assumption that the flows we consider are continuous flows on compact metric spaces.\footnote{In the case of flows one uses the special flow representation~\cite{MR0004730,MR0005800} and a result from~\cite{MR0430211} on the homology class of integrable functions.}

We formulate most of the definitions in terms of $\Z$-actions. However, they can be transferred directly (or almost directly) to other group actions.

\subsection{Topology, invariant measures, ergodic decomposition}\label{miary}
Let $X$ be a compact metric space. Denote by $M_1(X)$ the set of probability measures on $X$. $M_1(X)$ with the usual weak topology is compact and metrizable. 
For a flow $\cT=(T_t)_{t\in\R}$ and automorphism $T$ on $X$ we define:
\begin{align*}
& \cP_\R(\cT):=\{\mu\in M_1(X)\colon \mu\circ T_t=\mu \text{ for all }t\in\R\}, \\
& \cP_\R^e(\cT):=\{\mu\in \cP_\R(\cT)\colon \mu \text{ is  ergodic for }\cT\},\\
& \cP_\Z(T):=\{\mu\in M_1(X)\colon \mu\circ T=\mu\}, \\
& \cP_\Z^e(T):=\{\mu\in \cP_\Z(T)\colon \mu \text{ is  ergodic for } T\}.
\end{align*}

\begin{uw}\label{uw:ciaglosc}
Since $X$ is a compact metric space and $\cT$ is assumed to be continuous, it follows that the map $\R\ni t\mapsto \mu\circ T_t\in M_1(X)$ is continuous whenever $\cP_\R(\cT)\neq\varnothing$.
\end{uw}

We will identify the following spaces: $\R/\Z$, $[0,1)$, $\T=\{z\in\mathbb{C}\colon |z|=1\}$ ($[0,1)$ will be understood additively and $\T$ -- multiplicatively) and we will equip them with the Lebesgue measure. Finally, given a measurable flow $\cT=(T_t)_{t\in\R}$ and a measure $\nu\in\cP_\Z^e(T_1)$ we set 
\begin{equation}\label{dec}
\gr{\nu}=\int_0^1 \nu \circ T_t\ dt.
\end{equation}

Recall the following folklore result (see, e.g.,~\cite{MR2346554}):
\begin{lm}\label{lm:dec}
The measure $\gr{\nu}$ is $\cT$-invariant and ergodic. Moreover, the map
\begin{equation}\label{2}
\cP^e_\Z(T_1)\ni \nu \mapsto \gr{\nu}\in \cP^e_\R(\cT)
\end{equation}
is onto and~\eqref{dec} is the ergodic decomposition of $\gr{\nu}$ for $T_1$.
\end{lm}

Note that the measures $\nu\circ T_t$ in decomposition~\eqref{dec} need not be pairwise orthogonal. However, we have the following:
\begin{lm}\label{lm:rozkladdokl}
In the decomposition~\eqref{dec} either all measures $\nu\circ T_t$ are equal to $\nu$ or there exists $k\geq 1$ such that
\begin{equation}\label{rozkl}
	\gr{\nu}=\int_{\T} \nu \circ T_t\ dt=k\int_0^{1/k} \nu\circ T_t\ dt
\end{equation}
and for $0\leq t<1/k$ all measures $\nu\circ T_t$ are mutually singular.
\end{lm}
\begin{proof}
Notice that any two measures in decomposition~\eqref{dec} are either equal or mutually singular. Let $A:=\{t\in \T\colon \eta=\eta\circ T_t\}$. It is clearly a subgroup of $\T$. Moreover, $A$ is closed by Remark~\ref{uw:ciaglosc}. Therefore either $A=\T$ or $A$ is finite. In the latter case we obtain formula~\eqref{rozkl} with $k=\#A$.
\end{proof}

\begin{uw}\label{widmodys}
If $\cT$ is a suspension flow with the first return time to the base equal to $a>0$ and a weakly mixing first return time map, then the discrete part of the spectrum is concentrated on $\frac{1}{a}\Z$.
\end{uw}

\begin{uw}\label{u:zawie}
Recall that a flow is a suspension flow (with a constant first return time to the base) if and only it has a rational point in its discrete spectrum~\cite{MR521638}. Therefore, whenever~\eqref{rozkl} holds, the flow $\cT$ can be represented as a suspension flow over its ${1/k}$-time map with common return time to the base equal to $1/k$. In other words, the whole space can be identified with $\left(X\times [0,1/k],\nu\otimes\lambda_{[0,1/k]}\right)$, where we glue together each pair of points $(x,1/k)$ and $(T_{1/k}x,0)$, 
with the action of $(T_t)_{t\in\R}$ given by $T_t(x,s)=(x,s+t)$. Moreover, by Remark~\ref{widmodys}, the discrete part of its spectrum is concentrated on $k\Z$.
\end{uw}

\subsection{Joinings I}\label{se:j1}
\subsubsection{Basic definitions and properties}
Let $T$ and $S$ be automorphisms of $(X,\mathcal{B},\mu)$ and $(Y,\cC,\nu)$ respectively. By $\mathcal{J}(T,S)$ we denote the set of all \emph{joinings} between $T$ and $S$, i.e. the set of all $T\times S$-invariant probability measures on $(X\times Y,\mathcal{B}\otimes \mathcal{C})$, whose projections on $X$ and $Y$ are equal to $\mu$ and $\nu$ respectively. For $\mathcal{J}(T,T)$ we write $\mathcal{J}(T)$ and we speak about \emph{self-joinings}. If we assume additionally that both $T$ and $S$ are ergodic, then $\mathcal{J}(T,S)$ is a simplex whose extreme points are joinings ergodic for $T\times S$. We write $\mathcal{J}^e(T,S)$ (or $\mathcal{J}^e(T)$ if $S=T$). If $\mathcal{J}(T,S)=\{\mu\otimes\nu\}$ we say that $T$ and $S$ are disjoint~\cite{MR0213508} and write $T\perp S$.

In a similar way one defines joinings of higher order, i.e. joinings of more than two automorphisms. We denote them by $\mathcal{J}(T_1,\dots,T_k)$, $k\geq 2$ and $\mathcal{J}(T_1,T_2,\dots)$ in case of a finite and an infinite number of automorphisms respectively. Again, we may be interested in the situation where all $T_i$ are isomorphic and speak about self-joinings. To distinguish self-joinings of different orders we write e.g. $\mathcal{J}_{\infty}^e(T)$ for infinite ergodic self-joinings of $T$ or $\mathcal{J}_3(T)$ for three-fold self-joinings of $T$. By $\mathcal{J}_2(T;\mathcal{A})$, where $\cA$ is a factor of $T$, we denote all 2-fold self-joinings of $T$ which project onto $\cA\otimes \cA$ as the diagonal measure. Sometimes, to simplify the notation, we will denote a joining between $T$ and $S$ by $T\vee S$.

Recall that 2-fold joinings are in one-to-one correspondence with Markov operators $\Phi\colon L^2(X,\mathcal{B},\mu)\to L^2(Y,\mathcal{C},\nu)$
such that $\Phi\circ T=S \circ \Phi$. This identification allows us to view $\mathcal{J}_2(T)$ as a metrizable compact semitopological semigroup endowed with the weak operator topology. A metric compatible with the weak topology on $\mathcal{J}_2(T)$ can be defined in the following way. Fix a dense subset $\{f_i'\colon i\geq 1\}\subset L^2(X,\cB,\mu)$, put $f_i:=f'_i/\|f'_i\|_{L^2(X,\mu)}$ for $i\geq 1$ and let $\overline{d}_{\{f_n\}_{n\in\N}}$ be given by
\begin{multline}\label{metryka}
\overline{d}_{\{f_n\}_{n\in\N}}(\xi,\nu)\\
=\sum_{i,j=1}^{\infty}\frac{1}{2^{i+j}}\left| \int_{X\times X} f_i(x)\overline{f_j(y)}\ d\xi(x,y)-\int_{X\times X} f_i(x)\overline{f_j(y)}\ d\nu(x,y) \right|.
\end{multline}
Notice that for $\xi,\nu\in \cJ_2(T)$ we always have 
\begin{equation}\label{eq:niezm}
\overline{d}_{\{f_n\}_{n\in\N}}(\xi\circ (T\times T),\nu\circ (T\times T))=\overline{d}_{(f_n)_{n\in\N}}(\xi,\nu).
\end{equation}

\begin{uw}\label{ro:jo}
We will often use Lemma~\ref{lm:dec} and Lemma~\ref{lm:rozkladdokl} for joinings. Notice that given a flow $\cT$ with $T_1$ ergodic, the map $\eta\mapsto \gr{\eta}$ has values in $\cJ_2^e(\cT)$ whenever $\eta\in \cJ^e_2(T_1)$. Moreover, it is surjective. 
\end{uw}

Recall that $S$ is a \emph{factor} of $T$ ($T$ is an \emph{extension} of $S$) if there exists $\pi\colon X\to Y$ such that $\pi\circ T=S\circ \pi$ and $\nu=\pi_\ast(\mu)$, i.e. for any $A\in\cC$ we have $\mu(\pi^{-1}A)=\nu(A)$. We write $T\to S$. Notice that when $S$ is a factor of $T$ then $\pi^{-1}(\cC)\subset \cB$ is $T$-invariant. On the other hand, every $T$-invariant sub-$\sigma$-algebra $\cA\subset \cB$ corresponds to some factor of $T$. If no confusion arises we will write $\cB\to \cA$.

We will use the following types of joinings:
\begin{itemize}
\item
For $S_2,\dots, S_N\in C(T)$\footnote{$C(T)$ stands for the centralizer of $T$, i.e. the group of all measure-preserving automorphisms which commute with $T$.} the measure $\nu=\mu_{S_2,\dots,S_N}$ defined by
$$
\mu_{S_2,\dots,S_N}(A_1\times\dots \times A_N)=\mu(A_1\cap S_2^{-1}A_2\cap \dots \cap S_N^{-1}(A_N))
$$
(with an obvious modification for $N=\infty$) is an $N$-self-joining of $T$ called a \emph{graph self-joining} (or \emph{off-diagonal self-joining}). If $S_2=\dots=S_N=Id$, we speak about the \emph{diagonal self-joining} and we denote it by $\Delta_N$. For $N=2$ we often write $\Delta$ instead of $\Delta_2$. When confusion might arise, we denote the marginals of $\Delta$ in parenthesis: $\Delta(\mu)$.
\item
Each pair of factors $\cA_1,\cA_2\subset \cB$, together with $\lambda\in \cJ(T|_{\cA_1},T|_{\cA_2})$, yields a self-joining of $T$ defined by
$$
\widehat{\lambda}(A_1\times A_2)=\int_{X/{\cA_1}\times X/{\cA_2}} E(A_1|\cA_1)(\overline{x})E(A_2|\mathcal{A}_2)(\overline{y})\ d\lambda(\overline{x},\overline{y}).
$$
It is called a \emph{relatively independent extension of $\lambda$}. In particular, when $\cA_1=\cA_2=\cA$ and $\lambda=\Delta$ we say that $\widehat{\Delta}$ is a relatively independent extension over $\cA$. We denote it by $\mu\otimes_{\cA}\mu$.
\item
Let $\lambda\in \cJ(T_1,\dots, T_n)$. Then $\lambda|_{X_i,X_j}=\lambda|_{X_i\times X_j}$ defined by
$$
\nu|_{X_i\times X_j}(C\times D)=\nu(\underbrace{X\times\dots\times X}_{i-1}\times C \times \underbrace{X\times \dots\times X}_{j-i-1}\times D\times \underbrace{X\times\dots\times X}_{n-j})
$$
is an element of $\cJ(T_i,T_j)$. Moreover, if $\lambda\in \cJ^e(T_1,\dots,T_n)$ then $\lambda|_{X_i,X_j}\in \cJ^e(T_i,T_j)$.
\item 
Let $\lambda\in \cJ(T_1,T_2)$ and let $\cA_1,\cA_2$ be factors of $T_1,T_2$, respectively. Then $\lambda|_{\cA_1\otimes\cA_2}\in \cJ((T_1)|_{\cA_1},(T_2)|_{\cA_2})$.
\end{itemize}

\subsection{Factors and extensions II}\label{se:extensions}

Any ergodic extension $\widetilde{T}\to T$ is isomorphic to a \emph{skew product} over $T$, i.e. to an automorphism $\overline{T}$ on $(X\times Y,\cB\otimes\cC,\mu\otimes \nu)$ of the form
\begin{equation}\label{skosny}
\overline{T}(x,y)=(Tx,S_xy),
\end{equation}
where $(S_x)_{x\in X}$ is a measurable family of automorphisms of some standard probability space $(Y,\cC,\nu)$~\cite{MR0140660}. A particular case of skew products are \emph{Rokhlin extensions}, i.e. automorphisms of the form $T_{S,\varphi}(x,y)=(Tx,S_{\varphi(x)}y)$,
where $(S_g)_{g\in G}$ is a measurable representation of a locally compact abelian group $G$ in the group of automorphisms of $\ycn$ and $\varphi\colon X \to G$ is a measurable function.

\subsubsection{Relatively weakly mixing extensions}\label{relsekcja1}
Following~\cite{MR603625}, we say that $T\colon \xbm\to \xbm$ is \emph{relatively weakly mixing} with respect to factor $\cA\subset\cB$ 
if the relatively independent extension of the diagonal self-joining of $T|_\cA$ is ergodic. In terms of the skew product representation~\eqref{skosny}, the property of relative weak mixing is equivalent to ergodicity of automorphism $\overline{\overline{T}}$ acting on $ (X\times Y\times Y,\cB\otimes \cC\otimes \cC,\mu\otimes\nu\otimes\nu)$, given by the formula
$\overline{\overline{T}}(x,y_1,y_2)=(Tx,S_xy_1,S_xy_2)$.

\subsubsection{Relatively isometric and relatively distal extensions}\label{relsekcja}
Relative distality is a concept which is of ``opposite nature'' to relative weak mixing (cf. Theorem~\ref{fz}).
\paragraph{Compact group extensions}
Let $G$ be a compact metrizable group with the normalized Haar measure $\lambda_G$ and let $\varphi\colon X\to G$ be measurable. Then $T_\varphi\colon (x,g)\mapsto (Tx,\varphi(x)g)$ is an automorphism of $(X\times G,\cB\otimes\cB(G),\mu\otimes \lambda_G)$. It is called a (compact) \emph{group extension} of $T$. We have the following:
\begin{pr}[see e.g.\ \cite{MR922364}]\label{pr:veech1}
If $T\colon (X,\mathcal{B},\mu)\to (X,\mathcal{B},\mu)$ is an ergodic compact group extension of its factor $\cA$ then $J^e_2(T;\cA)$ consists of graphs joinings.
\end{pr}
\begin{pr}[$\Z$-actions:~\cite{MR1325699,MR685378}, actions of other groups:~\cite{MR922364}]\label{pr:veech2}
If $T\colon (X,\mathcal{B},\mu)\to (X,\mathcal{B},\mu)$ is ergodic, $\cA$ is a factor of $T$ and $J^e_2(T;\cA)$ consists of graphs joinings then $\cB\to\cA$ is a compact group extension.
\end{pr}
\begin{lm}[\cite{MR1091425}]
Let $T\colon (X,\mathcal{B},\mu)\to (X,\mathcal{B},\mu)$ be a (not necessarily ergodic) compact group extension of an ergodic factor $\cA$. Then every ergodic component of $T$ is also isomorphic to some compact group extension of $\cA$.
\end{lm}

\paragraph{Isometric extensions}
Let $G$ be a compact group with a closed subgroup $H$ and let $\varphi\colon X\to G$ be measurable. The action of $T_\varphi$ restricted to $X\times G/H$ is called an~\emph{isometric extension} of $T$.

In~\cite{MR0409770} Zimmer introduced the notion of relatively discrete spectrum. It is based on the classical notion of discrete spectrum of a unitary representation. As we will not need the precise definition, let us remark here only that
\begin{multline}\label{szesc}
\mbox{an ergodic extension has relatively discrete spectrum}\\
\mbox{if and only if it is isometric.}
\end{multline}

\paragraph{Distal extensions}
We say that the extension $\widetilde{T}\colon (\widetilde{X},\widetilde{\cB},\widetilde{\mu})\to (\widetilde{X},\widetilde{\cB},\widetilde{\mu})$ of $T\colon \xbm\to\xbm$ is \emph{relatively distal}~\cite{MR0414832} if there exists a transfinite sequence $(\cB_\alpha)_{\alpha\leq \beta}$ of factors of $\widetilde{T}$ such that $\cB_0=\cB$, $\cB=\widetilde{\cB}$, the extensions $\cB_{\alpha+1}\to \cB_\alpha$ are isometric and, for $\alpha$ being limit ordinals, $\cB_\alpha$ is the inverse limit of the preceding factors.

\subsection{Joinings II}
\subsubsection{Self-joining properties}\label{selfj2}
Let $T\colon \xbm\to\xbm$ be an ergodic automorphism. It is said that $T$
\begin{itemize}
\item
has the \emph{pairwise independence property} (PID) if every pairwise independent self-joining of $T$ is equal to the product measure;
\item is \emph{2-fold simple} if $\mathcal{J}_2^e(T)\subset \{\mu_S : S\in C(T)\}\cup \{\mu\otimes \mu\}$; is \emph{simple} if it is 2-fold simple and PID;
\item has \emph{minimal self-joinings} (MSJ) if it is \emph{2-fold simple} and $C(T)=\{T^k\colon k\in\Z\}$;
\item is \emph{2-fold quasi-simple} (2-QS)  if $\mathcal{J}^e(T)\setminus\{\mu\otimes\mu\}$ consists of isometric extensions of each of the coordinate factors; is \emph{quasi-simple} if it is 2-QS and PID;
\item is \emph{``n:1''} for some $n\in\N$ if $\mathcal{J}^e(T)\setminus\{\mu\otimes\mu\}$ consists of ``$n:1$''-extensions of each of the coordinate factors (it is a special case of 2-QS);
\item is \emph{2-fold distally-simple} (2-DS) if $\mathcal{J}^e(T)\setminus\{\mu\otimes\mu\}$ consists of distal extensions of each of the coordinate factors; is \emph{distally simple} if is it 2-DS and PID.
\end{itemize}
We say that $T$ has \emph{joining primeness property} (JP) if for every pair of weakly mixing systems $S_1,S_2$ and every $\lambda\in \mathcal{J}^e(T,S_1\times S_2)$ we have
$\lambda=\lambda|_{X,Y_1}\otimes \nu_2\text{ or }\lambda=\lambda|_{X,Y_2}\otimes \nu_1$.
Notice that assuming $S_1$ and $S_2$ are isomorphic would give an equivalent definition. We will show later that JP is, in fact, an intrinsic notion and we will provide an equivalent definition in terms of self-joinings.

Finally, we recall the \emph{Ratner's property} (R-property). It is a property which can be enjoyed by flows and it has a different flavor than all the properties mentioned above. Since the definition is very technical and we will not use it explicitly, we will not quote it here. Instead, we refer the reader to~\cite{MR2753947,MR717825,MR832433,MR2237466} for the details. Note that the definition of R-property itself does not refer to joinings at all, however it imposes restrictions on the joinings of the flow under consideration. The following proposition is a little bit imprecise, but it will be sufficient for our needs.
\begin{pr}[\cite{MR2753947}]\label{rat}
Let $\cT=(T_t)_{t\in\R}$ be a weakly mixing flow on $\xbm$ which has the R-property. Let $\cS=(S_t)_{t\in\R}$ be an ergodic flow on $\ycn$ and let $\lambda\in \cJ^e(\cT,\cS)$. Then either $\lambda=\mu\otimes \nu$ or $\lambda$ is a finite extension of $\nu$.
\end{pr}

\subsection{Distal and isometric extensions -- tools}\label{relsekcja2}
\subsubsection{T-compactness}
Let $\cA$ be a factor of $T$. Following~\cite{MR1930996}, we say that the extension $\cB\to \cA$ is \emph{$T$-compact} if for any $\vep>0$ there exists $N\geq 2$ such that for each $\nu\in J^e_N(T;\cA)$ we can find $1\leq i<j\leq N$ such that $\overline{d}(\nu|_{X_i\times X_j},\Delta_N|_{X_i\times X_j})<\vep$.\footnote{See the definitions in Section~\ref{se:j1}.}

\begin{pr}[\cite{MR1930996}]\label{Tzwarte}
Assume that $T$ is an ergodic automorphism on $(X,\mathcal{B},\mu)$ and let $\mathcal{A}\subset\mathcal{B}$ be its factor. Then the following are equivalent: 
\begin{itemize}
\item[(i)]
$\mathcal{B}\to\mathcal{A}$ has relatively discrete spectrum,
\item[(ii)]
$\mathcal{B}\to\mathcal{A}$ is isometric,
\item[(iii)]
$\mathcal{B}\to\mathcal{A}$ is $T$-compact.
\end{itemize}
\end{pr}

\section{Self-joining properties of $\cT$ and $T_1$ -- part I}

\subsection{Simplicity}

Our starting point is a result of del Junco and Rudolph. We rephrase it slightly and also include the proof, using language compatible with the one used later on.

\begin{pr}[\cite{MR922364}]\label{pr:simple}
Let $\cT=(T_t)_{t\in \R}$ be a weakly mixing flow. The following are equivalent:
\begin{enumerate}[(i)]
\item $\cT$ is 2-fold simple,
\item $T_1$ is 2-fold simple and $C(\cT)=C(T_1)$.
\end{enumerate}
\end{pr}

\paragraph{Remarks}
\begin{enumerate}
\item
In particular, $C(\cT)=C(T_1)$ whenever $C(T_1)$ is abelian.
\item
An analogous result to Proposition~\ref{pr:simple} holds for automorphisms and their powers.
\item
A flow $\cT$ which is not weakly mixing and which has an ergodic time $T_t$ which is 2-fold simple, is also 2-fold simple. Indeed, as $\cT$ is not weakly mixing, $T_t$ is also not weakly mixing, whence it has purely discrete spectrum~\cite{MR922364}. Therefore also $\cT$ has purely discrete spectrum and in particular, it is 2-fold simple.

\end{enumerate}

\begin{proof}[Proof of Proposition~\ref{pr:simple}]
Suppose that $\cT$ is 2-fold simple. It suffices to show that $\cJ^e_2(\cT)=\cJ^e_2(T_1)$. We have $\cJ^e_2(\cT)\subset \cJ^e_2(T_1)$. Suppose that we can find $\eta\in \cJ^e_2(T_1)\setminus\cJ^e_2(\cT)$. Define
\begin{equation}\label{eta}
\gr{\eta}=\int_0^1 \eta\circ (T_t\times T_t)\ dt.
\end{equation}
By Remark~\ref{ro:jo}, we have $\gr{\eta}\in\cJ^e_2(\cT)\subset \cJ^e_2(T_1)$. It follows by Lemma~\ref{lm:rozkladdokl} and by the uniqueness of ergodic decomposition that~\eqref{eta} takes the form $\gr\eta=\eta$, which yields a contradiction with the choice of $\eta$.

The other implication follows easily: we obtain $\cJ^e_2(T_1)\subset \cJ^e_2(\cT)$, whence $\cJ_2(T_1)\subset \cJ_2(\cT)$. $\cJ_2(\cT)\subset \cJ(T_1)$ is obvious.
\end{proof}

\begin{uw}\label{uw:delru}
Let $T$ and $S$ be 2-fold simple, such that $T^k=S^l$ for some $k,l\in\Z\setminus\{0\}$. It follows from Proposition~\ref{pr:simple} that $T^k=S^l$ is 2-fold simple and $C(T)=C(T^k)=C(S^l)=C(S)$. In particular, $\cJ^e_2(T)=\cJ^e_2(S)$. 
\end{uw}

For a measure-preserving flow $\cT=(T_t)_{t\in\R}$ let
$$
I_s(\cT)=\{t\in\R \colon T_t\text{ is 2-fold simple}\}\cup\{0\}.
$$

\begin{lm}\label{denseset1}
Let $\cT=(T_t)_{t\in\R}$ be a weakly mixing flow. Suppose that $T_{\alpha}$ and $T_\beta$ are 2-fold simple for some $\alpha,\beta\neq0$ such that $\alpha/\beta\not\in\Q$. Then $\cT$ is 2-fold simple. In particular, $I_s(\cT)=\R$.
\end{lm}
\begin{proof}
We may assume without loss of generality that $\beta=1$ and $\alpha\not\in\Q$. There are two possibilities:
\begin{multicols}{2}
\begin{enumerate}[(i)]
\item
$C(T_1)=C(T_\alpha)$,
\item
$C(T_1)\neq C(T_\alpha)$.
\end{enumerate}
\end{multicols}
We first cover case (i). Fix $t_0\in\R$. For $n\in \N$ choose $t_n\in \Z+\alpha\Z$ such that $t_n\to t_0$ (such a choice can always be made as $\alpha\not\in\Q$). Take $S\in C(T_1)=C(T_\alpha)$. We have
$$
ST_{t_n}\to ST_{t_0}\text{ and }T_{t_n}S\to T_{t_0}S.
$$
Since $ST_{t_n}=T_{t_n}S$, it follows that $ST_{t_0}=T_{t_0}S$, i.e. $C(T_1)=C(\cT)$. By Proposition~\ref{pr:simple}, $\cT$ is 2-fold simple.

We pass now to case (ii). We may assume without loss of generality that there exists $S\in C(T_1)\setminus C(T_\alpha)$. Consider
\begin{equation}\label{31b}
\gr{{\mu}_S}:=\int_0^1 \mu_S\circ (T_t\times T_t)\ dt.
\end{equation}
By Remark~\ref{ro:jo}, $\gr{\mu_S}\in \cJ_2^e(\cT)$. Clearly $\gr{\mu_S}\not\in J_2^e(T_1)$ (otherwise the measures in the decomposition~\eqref{31b} would be all equal, which would imply that $S\in C(\cT)\subset C(T_\alpha)$). By Lemma~\ref{lm:rozkladdokl}, the ergodic decomposition of $\gr{{\mu}_S}$ for $T_1\times T_1$ is of the form 
\begin{equation}\label{er1}
\gr{\mu_S}=k\int_0^{1/k}\mu_S\circ (T_t\times T_t)\ dt
\end{equation}
for some $k\geq 1$. By Remark~\ref{u:zawie} 
\begin{equation}\label{31a}
\mbox{the set of eigenvalues of $(\cT\times \cT,\gr{{\mu}_S})$ is equal to $k\Z$. }
\end{equation}

Suppose that $\gr{\mu_S}$ is not ergodic for $T_\alpha\times T_\alpha$. Then, by Lemma~\ref{lm:dec},
$
\gr{{\mu}_S}=\frac{1}{\alpha}\int_0^\alpha \mu_W\circ(T_t\times T_t)\ dt
$
for some $W\in C(T_\alpha)$. By Lemma~\ref{lm:rozkladdokl} the ergodic decomposition of $\gr{{\mu}_S}$ for $T_\alpha\times T_\alpha$ is therefore of the form 
$
\gr{\mu_S}=\frac{l}{\alpha}\int_0^{\alpha/l} \mu_W\circ (T_t\times T_t)\ dt
$
for some $l\geq 1$. Hence, by Remark~\ref{u:zawie}, the set of eigenvalues $(\cT\times \cT,\gr{{\mu}_S})$ is equal to $(l/\alpha)\Z$. This yields a contradiction with~\eqref{31a}.

Suppose now that $\gr{\mu_S}$ is ergodic for $T_\alpha\times T_\alpha$. Then $\gr{\mu_S}=\mu_W$ for some $W\in C(T_\alpha)$. It follows that $(T_\alpha\times T_\alpha,\gr{\mu_S})$ is weakly mixing, whence also $(\cT\times \cT,\gr{\mu_S})$ is weakly mixing. This yields a contradiction with~\eqref{31a} and ends the proof.
\end{proof}

\begin{pr}
Let $\cT=(T_t)_{t\in\R}$ be a weakly mixing flow. Then 
$I_s(\cT)=\R$ or $I_s(\cT)=\beta\Z$ for some $\beta\in\R$. Moreover, if $I_s(\cT)=\R$ then $\cT$ is 2-fold simple.
\end{pr}
\begin{proof}
Suppose that $I_s(\cT)\neq \beta\Z$ for all $\beta\in\R$. Consider the following two possibilities:
\begin{enumerate}[(i)]
\item
there exist $t_1,t_2\in I_s(\cT)\setminus \{0\}$ which are rationally independent,
\item
all numbers in $I_s(\cT)$ are rationally dependent.
\end{enumerate}
In case (i) it follows from Lemma~\ref{denseset1} that $\cT$ is 2-fold simple and $I_s(\R)=\R$. In case (ii) there exist $\alpha>0$ such that $I_s(\cT)\subset \alpha\Q$. Take $t_1,t_2\in I_s(\cT)$. We claim that the whole additive subgroup generated by $t_1$ and $t_2$ is in $I_s(\cT)$. Indeed, we have $t_1=p_1/q$, $t_2=p_2/q$ for some $p_1,p_2,q\in \Z$. Let $d=\gcd(p_1,p_2)$. Then for some $k,l\in\Z$ we have $kp_1+lp_2=d$. Hence $d/q=kp_1/q+lp_2/q$. It follows by Remark~\ref{uw:delru} that $C(T_{p_1/q})=C(T_{p_2/q})$. Therefore $C(T_{p_1/q})\subset C(T_{d/q})$. On the other hand, $C(T_{d/q})\subset C(T_{p_1/q})$ since $T_{d/q}$ is a root of $T_{p_1/q}$. Hence $C(T_{d/q})=C(T_{p_1/q})$ and it follows that $T_{d/q}$ is 2-fold simple. Therefore, for all $k\in\Z$ also $T_{kd/q}$ is 2-fold simple. It follows that $I_s(\cT)=I_s(\cT)\cap \alpha\Q$ is an additive subgroup of $\alpha\Q$. Notice that $I_s(\cT)$ cannot be dense in $\R$. Indeed, if it was dense, an argument similar to the one we used in case (i) in the proof of Lemma~\ref{denseset1} would imply $I_s(\cT)=\R$.
\end{proof}

\subsection{Joining primeness property}\label{JPsekcja}

We will show now that the converse of Proposition~\ref{pr:simple} is true if we replace ``2-fold simple'' with ``JP'':
\begin{pr}\label{pr:jp}
Let $\cT=(T_t)_{t\in \R}$ be a measure-preserving flow with $T_1$ ergodic. If $T_1$ has the JP property then $\cT$ has the JP property.
\end{pr}
\begin{proof}
Let $\gr{\eta}\in J^e(\mathcal{T},\mathcal{R}\times \mathcal{S})\subset J(T_1,R_1\times S_1)$ for some weakly mixing flows $\mathcal{R}$ and $\mathcal{S}$. If $\gr{\eta}\in J^e(T_1,R_1\times S_1)$ then $\gr{\eta}=\gr{\eta}|_{X\times Y}\otimes \nu_2$ or $\gr{\eta}=\gr{\eta}|_{X\times Z}\otimes \nu_1$. If $\gr{\eta} \not\in J^e(T_1,R_1\times S_1)$ then by Lemma~\ref{lm:dec}, Lemma~\ref{lm:rozkladdokl} and Remark~\ref{ro:jo}, we may assume without loss of generality that for some $\eta\in J^e(T_1,R_1\times S_1)$ we have
$$
\gr{\eta}=\int_0^1 \eta\circ (T_t\times R_t\times S_t)\ dt.
$$
Notice that $\eta$ cannot be the product measure. Indeed, it would be invariant under the product flow $\mathcal{R}\times \mathcal{S}=(R_t\times S_t)_{t\in\R}$ and equal to $\gr{\eta}$, which is impossible since $\gr{\eta}\not\in\cJ^e(T_1,R_1\times S_1)$. Therefore, up to a permutation of the coordinates, we have $\eta=\eta|_{X\times Y}\otimes \nu_2$.
Hence
$$
\gr{\eta}=\int_0^1 \eta|_{X\times Y}\circ (T_t\times R_t)\otimes \nu_2\ dt= \left(\int_0^1\eta|_{X\times Y}\circ (T_t\times R_t)\ dt\right) \otimes \nu_2
$$
which ends the proof.
\end{proof}

\begin{uw}\label{uwp1}
If $\cT^{(1)}$ and $\cT^{(2)}$ are weakly mixing and all ergodic joinings of $\cT^{(1)}$ and $\cT^{(2)}$ are weakly mixing then $\cJ(\cT^{(1)},\cT^{(2)})=\cJ(T^{(1)}_1,T^{(2)}_1)$ and $\cJ^e(\cT^{(1)},\cT^{(2)})=\cJ^e(T^{(1)}_1,T^{(2)}_1)$. Indeed, by Lemma~\ref{lm:dec} the only possibility to obtain a new ergodic joining is via suspension, every suspension however has a discrete spectrum factor.
\end{uw}
\begin{uw}\label{uwp2}
If $\cj_2(\cT)=\cj_2(T_1)$ then $\cT$ and $T_1$ have the same invariant $\sigma$-algebras and $C(\cT)=C(T_1)$.
\end{uw}

Given an ergodic automorphism $R$ on a standard probability Borel space $\zdr$, we denote by $\cj\cf(R)$ the family of all factors of all infinite ergodic self-joinings of $R$. Notice that 
\begin{align}\label{9p}
\begin{split}
&\mbox{the class $\cj\cf(R)$ is closed under}\\
&\mbox{(infinite) ergodic self-joinings and factors.}
\end{split}
\end{align}
Therefore, for each ergodic automorphism $S$ acting on $\ycn$ there exists a largest factor $\cA_R(S)\subset \cC$ of $S$ such that $\cA_R(S)\in \cj\cf(R)$.\footnote{By separability, the smallest $\sigma$-algebra including all factors of $S$ from $\cj\cf(R)$ is countably generated, whence it corresponds to a joining of at most countable number of factors.} Since $S|_{W\ca_R(S)}$ and $S|_{\cA_R(S)}$ are isomorphic for each $W\in C(S)$, we obtain that 
\begin{equation}\label{10p}
\mbox{$\cA_R(S)$ is invariant under all $W\in C(S)$.}
\end{equation}
Now, fix a weakly mixing automorphism $R\colon \zdr\to\zdr$ and an ergodic automorphism $S\colon \ycn\to\ycn$. Then, by~\eqref{10p}, $\cA_{R\times R}(S\times S)$ is a factor of the product action $(S^i\times S^j)_{(i,j)\in \Z^2}$. Recall a result on product $\Z^2$-actions:
\begin{pr}[\cite{MR2237524}]
Let $\tilde{T}\colon (\tilde{X},\tilde{\cB},\tilde{\mu})\to (\tilde{X},\tilde{\cB},\tilde{\mu})$ and $\tilde{S}\colon (\tilde{Y},\tilde{\cC},\tilde{\nu})\to (\tilde{Y},\tilde{\cC},\tilde{\nu})$ be ergodic automorphisms of standard probability Borel spaces. Consider the corresponding $\Z^2$-action, i.e. the action of $(\tilde{T}^n\times \tilde{S}^m)_{n,m\in\Z^2}$ on $(\tilde{X}\times \tilde{Y},\tilde{\cB}\otimes\tilde{\cC},\tilde{\mu}\otimes \tilde{\nu})$. Suppose that $\tilde{\cA}$ is a factor of the $\Z^2$-action. Let $\overline{B}\otimes \overline{C}$ be the smallest product factor containing $\tilde{A}$.\footnote{The existence of such factor was also proved in~\cite{MR2237524}.} Then there exists a compact metric group $G$ and two continuous 1-1 homomorphisms
$$
g\mapsto \tilde{R}_g\in C(\tilde{T}),\ \ g\mapsto \tilde{R}'_g\in C(\tilde{S}),
$$
such that $\tilde{A}=\{A\in \overline{B}\otimes \overline{C}\colon (\tilde{R}_g\times \tilde{R}'_g)A=A \}$.
\end{pr}
It follows from the above proposition and from~\eqref{10p} applied to $\cA_{R\times R}(S\times S)$ that $\cA_{R\times R}(S\times S)$ is a product-$\sigma$-algebra. Therefore 
\begin{equation}\label{12p}
\cA_{R\times R}(S\times S)=\cA_R(S)\otimes \cA_R(S)
\end{equation}
as this is the largest product $\sigma$-algebra in $\cj\cf(R\times R)$.

Moreover, using~\eqref{9p} and a result from~\cite{Austin:2009fk} (see also~\cite{Rue:2009kx}), we conclude that for any $\lambda \in \cJ^e(R,S)$ we have
$\lambda=\left(\lambda|_{\mathcal{D}\otimes\cA_R(S)}\right){\widehat{}}\ $,
i.e. $\lambda$ is the relatively independent extension of $\lambda|_{\mathcal{D}\otimes \cA_R(S)}$. In particular, using~\eqref{12p}, for $\lambda\in \cJ^e(R,S\times S)$, we obtain
\begin{equation}\label{11p}
\lambda=\left(\lambda|_{\mathcal{D}\otimes\cA_{R}(S)\otimes \cA_R(S)}\right){\widehat{}}.
\end{equation}

We will show now that JP is an intrinsic property.
\begin{pr}\label{j5}
Let $R$ be an ergodic automorphism such that in the class $\cj\cf(R)$ there is some weakly mixing transformation. An automorphism $R$ has the JP property if and only if for each weakly mixing $S\in \cj\cf(R)$ and each $\lambda\in J^e(R,S\times S)$ we have either $\lambda=(\lambda|_{Z\times Y_1})\otimes\nu$ or $\lambda=(\lambda|_{Z\times Y_2})\otimes\nu$.
\end{pr}
\begin{proof}
Fix an arbitrary weakly mixing automorphism $S$  and let $\lambda\in J^e(T,S\times S)$. Then~\eqref{11p} holds,
equivalently, $\mathcal{D}$ and $\cC\otimes \cC$ are relatively independent over $\cA_R(S)\otimes \cA_R(S)$. Take $f\in L^\infty(Z,\rho)$, $g_i\in L^\infty(Y_i,\nu)$ for $i=1,2$. Assume that $\int g_2=0$. We have
\begin{multline}\label{13p}
\int f\otimes g_1\otimes g_2\ d\lambda=\int E^\lambda(f\otimes g_1\otimes g_2 | \cA_R(S)\otimes \cA_R(S))\ d\lambda\\
\stackrel{\eqref{11p}}{=} \int E^\lambda(f)| \cA_R(S)\otimes \cA_R(S)) E^\lambda( g_1\otimes g_2|\cA_R(S)\otimes \cA_R(S))\ d\lambda.
\end{multline}
To distinguish both $\sigma$-algebras $\cA_R(S)$ we will write $\cA^1_R(S)$ for $\cA_R(S)\otimes \{\varnothing,Y\}$ and $\cA^2_R(S)$ for $\{\varnothing ,Y\}\otimes \cA_R(S)$. By the assumption $\mathcal{D}\otimes \cA^1_R(S) \perp \cA^2_R(S)$ (up to a permutation of coordinates) with respect to $\lambda$, whence the last integral in~\eqref{13p} equals
\begin{multline*}
\int E^\lambda(f|\cA^1_R(S))E^\nu(g_1|A^1_R(S))E^\nu(g_2|A^2_R(S))\ d\lambda\\
\stackrel{(\ast)}{=}\left(\int E^\lambda(f|\cA^1_R(S))E^\nu(g_1|\cA^1_R(s))\ d\lambda \right)\cdot \int E^\nu(g_2|\cA^2_R(S))\ d\lambda=0,
\end{multline*}
where $(\ast)$ follows by the independence of $\cA^1_R(S)$ and $\cA^2_R(S)$. This completes the proof.
\end{proof}

\begin{uw}
Notice that if $R$ is such that in the class $\cj\cf(R)$ there are no weakly mixing automorphisms then $R$ is disjoint from all weakly mixing systems (see Theorem~\ref{parreau}), which implies the JP property.
\end{uw}
\begin{uw}
It was shown in~\cite{MR2018611}, using Rokhlin extensions, that there exists a non-weakly mixing automorphism, such that the class $\cj\cf(R)$ includes some weakly mixing systems.\footnote{Recall that the class of JP systems is closed under taking Rokhlin extensions.}
\end{uw}

\begin{wn}
Let $\cT=(T_t)_{t\in\R}$ be such that its all ergodic self-joinings are weakly mixing. Then $T_1$ has the JP property if and only if $\cT$ has the JP property. 
\end{wn}
\begin{proof}
This follows immediately by Remarks~\ref{uwp1},~\ref{uwp2} and Proposition~\ref{j5}.
\end{proof}

The following problem remains open:
\begin{qu}\label{pytanie}
Does the JP property of a weakly mixing flow $\cT=(T_t)_{t\in\R}$ imply the JP property for $T_t$, $t\neq 0$?
\end{qu}

\subsection{Isometric, distal and relatively weakly mixing extensions}
\subsubsection{Isometric extensions}\label{pierwszydowod}
Suppose that $T\colon \xbm\to \xbm$ is embeddable into some measurable flow $\cT=(T_t)_{t\in\R}$, i.e. $T=T_1$. Let $\{f'_n\colon {n\in\N}\}\subset L^2(X,\mu)$ be dense and let $f_n:=f'_n/\|f'_n\|_{L^2(X,\mu)}$ for $n\in\N$. Let $d_{\{f_n\}_{n\in\N}}$ be the metric defined as in~\eqref{metryka}, using $\{f_n\}_{n\in\N}$.
\begin{lm}\label{vep4}
For any $\vep>0$ there exists $\delta>0$ such that for all $\xi\in\cJ_2(T)$ and $t\in\R$ such that $|t|<\delta$ we have  $\overline{d}_{\{f_n\}_{n\in\N}}(\xi\circ(T_t\times T_t),\xi)<\vep/4$.
\end{lm}
\begin{proof}
Fix $\vep>0$. Let $N\in\N$ be such that $\sum_{i> N}1/2^i<\vep/32$. Let $\delta>0$ be such that for $1\leq i\leq N$  and $t\in\R$ such that $|t|<\delta$ we have
$\left\|f_i\circ T_t-f_i \right\|_{L^2(X,\mu)}<\vep/8$. Then for $t\in\R$ and $\xi\in\cJ_2(T)$ by Cauchy-Schwarz inequality we obtain
\begin{align*}
d_{\{f_n\}_{n\in\N}}&(\xi\circ (T_t\times T_t),\xi)=\sum_{i,j\geq 1}\frac{1}{2^{i+j}}\left|\int_{X\times X}\left(f_i\circ T_t\otimes f_j\circ T_t-f_i\otimes f_j \right)\ d\xi \right|\\
\leq& \sum_{i,j\geq 1}\frac{1}{2^{i+j}}\left|\int_{X\times X}\left(f_i\circ T_t\otimes f_j\circ T_t-f_i\circ T_t\otimes f_j \right)\ d\xi \right|\\
&+ \sum_{i,j\geq 1}\frac{1}{2^{i+j}}\left|\int_{X\times X}\left(f_i\circ T_t\otimes f_j-f_i\otimes f_j \right)\ d\xi \right|\\
\leq&\sum_{i,j\geq 1}\frac{1}{2^{i+j}}\left\|f_j\circ T_t-f_j \right\|_{L^2(X,\mu)}+\sum_{i,j\geq 1}\frac{1}{2^{i+j}}\left\|f_i\circ T_t-f_i \right\|_{L^2(X,\mu)}\\
=&2\left(\sum_{i=1}^{N}\frac{1}{2^i}\left\|f_i\circ T_t-f_i \right\|_{L^2(X,\mu)}+\sum_{i>N}\frac{1}{2^i}\left\|f_i\circ T_t-f_i \right\|_{L^2(X,\mu)} \right)\\
\leq& 2\left(\sum_{i=1}^{N}\frac{1}{2^i}\cdot\vep/8+\vep/16\right)<\vep/4,
\end{align*}
which ends the proof.
\end{proof}
Let $A_\sigma:=\{(m,p,q)\in\N^3\colon m\geq 1,q\geq 0, 0\leq p\leq 2^q-1,\ p\text{ is odd}\}$ and
$
\sigma\colon A_\sigma \to \N
$
be a bijection. For $i\geq 1$ let 
$$
g_i:=f_m\circ T_{p/2^q} \text{ where }(m,p,q)=\sigma(i).
$$
Then $\{g_n\colon n\in\N\}\subset L^2(X,\cB,\mu)$ is clearly a dense subset as it includes $\{f_n\colon n\in\N\}$. Let $\overline{d}$ be the metric on $\cJ_2(T)$ defined as in~\eqref{metryka}, but using $\{g_n\}_{n\in\N}$ instead of $\{f_n\}_{n\in\N}$. 

\begin{lm}\label{vep5}
For any $\vep>0$ and any $q\in\N$ there exists $C=C(\vep,q)$ such that for all $0\leq p\leq 2^q-1$ and all $\rho_1,\rho_2\in\cJ_2(T)$ we have
\begin{equation*}
\overline{d}\left(\rho_1\circ(T_{p/2^q}\times T_{p/2^q}),\rho_2\circ(T_{p/2^q}\times T_{p/2^q}) \right)\leq C(\vep,q)\cdot \overline{d}(\rho_1,\rho_2)+\vep/2.
\end{equation*}
\end{lm}
\begin{proof}
Fix $\vep>0$ and $q\in\N$. Let $0\leq p\leq 2^q-1$ and let $\rho_1,\rho_2\in\cJ_2(T)$. Take $N\in\N$ such that
\begin{equation}\label{wyborN}
\sum_{i=1}^{N}\sum_{j>N}\frac{1}{2^{i+j}}+\sum_{i>N}\sum_{j=1}^{N}\frac{1}{2^{i+j}}<\vep/2.
\end{equation}
Let
$
a_{i,j}:=\left|\int_{X\times X}g_i\otimes g_j\ d\rho_1-\int_{X\times X}g_i\otimes g_j\ d\rho_2 \right|
$
for $i,j\geq 1$. Then
$$
\overline{d}(\rho_1,\rho_2)=\sum_{i,j\geq 1}\frac{1}{2^{i+j}}a_{i,j}.
$$
Let $\pi\colon \N\to \N$ be a bijection defined in the following way:
\begin{multline*}
\pi(i)=j \iff 
	\begin{cases}
		\sigma^{-1}(i)=(m,p_i,q_i),\\
		\sigma^{-1}(j)=(m,p_j,q_j),
	\end{cases}\\
	\text{for some }m\in\N \text{ and }q_i,q_j\in\N,\ 0\leq p_i\leq 2^{q_i}-1,\ 0\leq p_j\leq 2^{q_j}-1\\
	\text{such that }\frac{p_i}{2^{q_i}}+\frac{p}{2^q}-\frac{p_j}{2^{q_j}}\in\Z.
\end{multline*}
Permutation $\pi$ is well-defined as the addition of $p/2^q$ mod 1 is a bijection of the set $\{p_i/2^{q_i}\colon (1,p_i,q_i)\in A_\sigma\}$.  Recall also that $\rho_1,\rho_2$ are $T\times T$-invariant as elements of $\cJ_2(T)$. Therefore
$$
\overline{d}(\rho_1\circ(T_{p/2^{q}}\times T_{p/2^{q}}),\rho_2\circ(T_{p/2^{q}}\times T_{p/2^{q}}))=\sum_{i,j\geq 1}\frac{1}{2^{i+j}}a_{\pi(i),\pi(j)},
$$
By~\eqref{wyborN} we obtain
\begin{align*}
\sum_{i,j\geq 1}&\frac{1}{2^{i+j}}a_{\pi(i),\pi(j)}\leq \sum_{i,j=1}^{N}\frac{1}{2^{i+j}}a_{\pi(i),\pi(j)}+\vep/2\\
&=\sum_{i,j=1}^{N}2^{\pi(i)-i+\pi(j)-j}\frac{1}{2^{\pi(i)+\pi(j)}}a_{\pi(i),\pi(j)}+\vep/2\\
&\leq \sum_{i,j=1}^{\max\{\pi(i)\colon 1\leq i\leq N\}} 2^{2\max\{|\pi(i)-i|\colon 1\leq i\leq N\}}\frac{1}{2^{i+j}}a_{i,j}+\vep/2\\
&\leq 2^{2\max \{|\pi(i)-i|\colon 1\leq i \leq N\}}\sum_{i,j\geq 1}\frac{1}{2^{i+j}}a_{i,j}+\vep/2
=C(\vep,p,q)\cdot \overline{d}(\rho_1,\rho_2)+\vep/2,
\end{align*}
where $C(\vep,p,q)=2^{2\max \{|\pi(i)-i|\colon 1\leq i \leq N\}}$. To end the proof it suffices to put
$C(\vep,q):=\max\{C(\vep,p,q)\colon 0\leq p\leq 2^{q}-1\}$.
\end{proof}

\begin{lm}\label{vep44}
For any $\vep>0$ there exists $\delta>0$ such that for any $\xi_1,\xi_2\in\cJ_2(T)$, whenever $\overline{d}(\xi_1,\xi_2)<\delta$ then for all $t\in\R$, $\overline{d}(\xi_1 \circ(T_t\times T_t),\xi_2\circ(T_t\times T_t))<\vep$. In particular, $\overline{d}\left(\int_0^1 \xi_1 \circ (T_t\times T_t)\ dt,\int_0^1\xi_2\circ (T_t\times T_t)\ dt\right)<\vep$.
\end{lm}
\begin{proof}
Fix $\vep>0$. By Lemma~\ref{vep4}, we can find $\delta_0>0$ such that for $|t|<\delta_0$ and $\xi\in\cJ_2(T)$
\begin{equation}\label{f1a}
\overline{d}(\xi\circ(T_t\times T_t),\xi)<\vep/8.
\end{equation}
Let $q\in \N$ be such that $1/2^{q}<\delta_0$ and let $C=C(\vep,q_0)$ be such as in Lemma~\ref{vep5}. For any $t\in \R$ find $0\leq p\leq 2^{q}-1$ such that $|t-p/2^{q}|<\delta_0$. Using~\eqref{f1a} we obtain
\begin{align*}
\overline{d}(\xi_1\circ(T_t\times T_t),&\xi_2\circ(T_t\times T_t))\\
&\leq \overline{d}(\xi_1\circ(T_{p/2^{q}}\times T_{p/2^{q}}),\xi_2\circ(T_{p/2^{q}}\times T_{p/2^{q}}))+\vep/4\\
&\leq C\cdot\overline{d}(\xi_1,\xi_2)+\vep/4+\vep/4= C\cdot\overline{d}(\xi_1,\xi_2)+\vep/2,
\end{align*}
Let $0<\delta<\frac{\vep}{2\cdot C}$. Then the condition $\overline{d}(\xi_1,\xi_2)<\delta$ implies 
$
\overline{d}(\xi_1\circ (T_t\times T_t),\xi_2\circ (T_t\times T_t))<\vep.
$
It follows that
\begin{multline*}
\overline{d}\left(\int_0^1 \xi_1 \circ (T_t\times T_t)\ dt,\int_0^1 \xi_2\circ(T_t\times T_t)\ dt \right)\\
\leq \int_0^1 \overline{d}(\xi_1\circ(T_t\times T_t),\xi_2\circ (T_t\times T_t))\ dt<\vep,
\end{multline*}
which ends the proof.
\end{proof}

\begin{pr}\label{pr12}
Let $\cT=(T_t)_{t\in\R}$ be an ergodic flow on $(X,\cB,\widetilde{\gr{\nu}})$ with a factor $\cA$. Let $\gr{\nu}=\widetilde{\gr{\nu}}|_\cA$. Let $\widetilde{\gr{\nu}}=\int_0^1 \widetilde{\nu}\circ T_t\ dt$ and $\gr{\nu}=\int_0^1 \nu\circ {T_t}|_\cA\ dt$ for some $\widetilde{\nu}\in\cP_\Z^e(T_1), \nu\in \cP_\Z^e({T_1}|_\cA).$\footnote{It may happen that $\gr{\nu}=\nu$ or $\widetilde{\gr{\nu}}=\widetilde{\nu}$. The existence of $\nu$ and $\widetilde{\nu}$ follows from Lemma~\ref{lm:dec}.\label{stopka1}} Additionally assume that $\widetilde{\nu}|_\cA=\nu$. Then the extension $(T_1,\widetilde{\nu})\to({T_1}|_\cA,\nu)$ is isometric if and only if the extension $(\cT,\widetilde{\gr{\nu}})\to (\cT|_\cA,\gr{\nu})$ is isometric.
\end{pr}
\begin{proof}
Suppose that the extension $(\cT,\widetilde{\gr{\nu}})\to (\cT|_\cA,\gr{\nu})$ is isometric. Let $\overline{\cT}=(\overline{T}_t)_{t\in\R}\colon (\overline{X},\overline{\cB},\overline{\gr{\nu}})\to (\overline{X},\overline{\cB},\overline{\gr{\nu}})$ be an ergodic extension of $\cT$ such that $(\overline{\cT},\overline{\gr{\nu}})\to (\cT|_\cA,\gr{\nu})$ is a compact group extension. By Lemma~\ref{lm:dec} there exists $\overline{{\nu}}\in \cP_\Z^e(\overline{T}_1)$ such that $\overline{\gr{\nu}}=\int_0^1\overline{\nu}\circ \overline{T}_t\ dt$ (it may happen that $\overline{\gr{\nu}}=\overline{\nu}$). By the uniqueness of the ergodic decomposition and by Lemma~\ref{lm:rozkladdokl}, there exist $k,l,m\in\N\cup\{\infty\}$ such that
$$
\overline{\gr{\nu}}=k\int_0^{1/k}\overline{\nu}\circ \overline{T}_t\ dt,\ \ \ \widetilde{\gr{\nu}}=kl\int_0^{1/kl}\widetilde{\nu}\circ T_t\ dt,\ \ \ {\gr{\nu}}=klm\int_0^{1/klm}{\nu}\circ {T_t}|_\cA\ dt.\footnote{We interpret $\int_0^{1/\infty}\gr{\lambda}\circ S_t\ dt$ as $\gr{\lambda}$ for any measure $\gr{\lambda}$ and any flow $(S_t)_{t\in\R}$ such that this formula makes sense.\label{stopka}}
$$
Without loss of generality we may assume that $\overline{\nu}|_\cB=\widetilde{\nu}$. Since $(\overline{\cT},\overline{\gr{\nu}})\to (\cT|_\cA,\gr{\nu})$ is a compact group extension, $(\overline{T}_1,\overline{\gr{\nu}})\to (T_1|_\cA,\gr{\nu})$ is also a compact group extension. Notice that in the decomposition $\overline{\gr{\nu}}=k\int_0^{1/k}\overline{\nu}\circ \overline{T}_t\ dt$ the only measures which project down to $\nu$ are measures of the form $\overline{\nu}\circ T_{i/lm}$ for $0\leq i\leq lm-1$. Therefore also
$$
\left(\overline{T}_1,\frac{1}{lm}\sum_{i=0}^{lm-1}\overline{\nu}\circ T_{i/lm}\right)\to ({T_1}|_\cA,\nu)
$$
is a compact group extension (if $lm=\infty$, we interpret $\frac{1}{lm}\sum_{i=0}^{lm-1}\overline{\nu}\circ T_{i/lm}$ as $\overline{\gr{\nu}}$). It follows that also
$(\overline{T}_1,\overline{\nu})\to (T_1,\nu)$ is a compact group extension.
Thus
$(T_1,\widetilde{\nu})\to (T_1,\nu)$ is isometric which ends the first part of the proof.

Assume now that the extension $(T_1,\widetilde{\nu})\to (T_1|_\cA,\nu)$ is isometric. As in the first part of the proof, there exist $k,l\in\N\cup\{\infty\}$ such that
\begin{equation}\label{lam2}
\widetilde{\gr{\nu}}=k\int_0^{1/k}\widetilde{\nu}\circ T_t\ dt,\ {\gr{\nu}}=kl\int_0^{1/kl}{\nu}\circ {T_t}|_{\cA}\ dt
\end{equation}
are the corresponding ergodic decompositions (see footnotes~\ref{stopka1} and~\ref{stopka}). By Proposition~\ref{Tzwarte} the extension $(T_1,\widetilde{\nu})\to ({T_1}|_\cA,\nu)$ is T-compact. Fix $\vep>0$. Let $0<\delta<\vep$ be as in Lemma~\ref{vep4}. Let $N\in\N$ come from the definition of T-compactness for $(T_1,\widetilde{\nu})\to ({T_1}|_{\cA},\nu)$ with $\vep$ replaced with $\delta$. Take $\gr{\lambda}\in J^e_N((T,\widetilde{\gr{\nu}});\cA)$. Then $\gr{\lambda}\in \cP_\Z(T_1^{\times N})$ and there are two possibilities:
\begin{multicols}{2}
\begin{enumerate}[(i)]
\item
$\gr{\lambda}\in\cP_\Z^e(T_1^{\times N})$,
\item
$\gr{\lambda}\not\in\cP_\Z^e(T_1^{\times N})$.
\end{enumerate}
\end{multicols}
Case (i) may occur only if $k=\infty$, i.e. $\widetilde{\gr{\nu}}=\widetilde{\nu}$, $\gr{\nu}=\nu$ (otherwise the marginals of $\gr{\lambda}$ are not ergodic for $T_1$, so $\gr{\lambda}$ is not ergodic for $T_1^{\times N}$). Then by Proposition~\ref{Tzwarte} there exist $1\leq i<j\leq N$ such that
\begin{equation*}
\overline{d}(\gr{\lambda}|_{X_i\times X_j},\Delta_{X_i\times X_j}(\widetilde{\gr{\nu}}))<\delta<\vep.
\end{equation*}
We will cover now case (ii). To fix our attention, we will assume that $k,l\in\N$, i.e. $\widetilde{\gr{\nu}}\neq\widetilde{\nu}$ and $\gr{\nu}\neq\nu$ (in other cases the proof is similar). By Lemma~\ref{lm:dec} there exists $\lambda\in\cP_\Z^e(T_1^{\times N})$ such that 
\begin{equation}\label{lam}
\gr{\lambda}=\int_0^1\lambda\circ T_t^{\times N}\ dt. 
\end{equation}
We will show that 
\begin{equation}\label{AAA}
\mbox{such a measure $\lambda$ can be chosen from $\cJ_N^e((T_1,\widetilde{\nu});\cA)$.}
\end{equation}
In what follows we will use the fact that ergodic decompositions are unique. We have the following:
\begin{align}
\begin{split}
&\gr{\lambda}|_{\cA^{\otimes N}}=\int_0^1 \lambda|_{\cA^{\otimes N}}\circ \left({T_t}|_\cA \right)^{\times N}\ dt, \\ 
&\gr{\lambda}|_{\cA^{\otimes N}}=\Delta_{\cA^{\otimes N}}(\gr{\nu})=kl\int_0^{1/kl}\Delta_{\cA^{\otimes N}}(\nu)\circ \left({T_t}|_\cA \right)^{\times N}\ dt\label{lam1}
\end{split}
\end{align}
(the latter equality follows from the second formula in~\eqref{lam2}). Therefore we may assume without loss of generality (changing $\lambda\in\cP_\Z^e(T_1^{\times N})$ if necessary) that 
\begin{equation}\label{wesola}
\lambda|_{\cA^{\otimes N}}=\Delta_{\cA^{\otimes N}}(\nu)
\end{equation}
and $\gr{\lambda}|_{\cA^{\otimes N}}=kl\int_0^{1/kl}\lambda|_{\cA^{\otimes N}}\circ \left({T_t}|_{\cA} \right)^{\times N}\ dt$.
Moreover, for $1\leq i\leq N$ we have
\begin{align*}
&\gr{\lambda}|_{X_i}=k\int_0^{1/k}\lambda|_{X_i}\circ T_t\ dt,\\
&\gr{\lambda}|_{X_i}=\widetilde{\gr{\nu}}=k\int_0^{1/k}\widetilde{\nu}\circ T_t\ dt,
\end{align*}
whence there exists $t_i\in[0,1/k)$ such that 
\begin{equation}\label{smutna}
\lambda|_{X_i}=\widetilde{\nu}\circ T_{t_i}.
\end{equation}
It follows that
$$
\left(\lambda|_{X_i} \right)|_\cA=\widetilde{\nu}|_\cA\circ {T_{t_i}}|_\cA=\nu\circ \left(T_{t_i}|_\cA \right).
$$
On the other hand, by~\eqref{wesola}, we have $\left(\lambda|_{X_i} \right)|_\cA=\nu$.
Hence $t_i=0$ and~\eqref{AAA} follows from~\eqref{wesola} and~\eqref{smutna}. By Proposition~\ref{Tzwarte} there exist $1\leq i<j\leq N$ such that
\begin{equation}\label{lam3}
\overline{d}(\lambda|_{X_i\times X_j},\Delta_{X_i\times X_j}(\widetilde{\nu}))<\delta.
\end{equation}
By~\eqref{lam} we have
\begin{equation}\label{lam4}
\gr{\lambda}|_{X_i\times X_j}=\int_0^1 \lambda|_{X_i\times X_j} \circ (T_t\times T_t)\ dt.
\end{equation}
Moreover,
\begin{equation}\label{lam5}
\Delta_{X_i\times X_j}(\widetilde{\gr{\nu}})=\int_0^1 \Delta_{X_i\times X_j}(\widetilde{\nu})\circ (T_t\times T_t)\ dt
\end{equation}
It follows from the choice of $\delta$, from ~\eqref{lam3},~\eqref{lam4} and~\eqref{lam5} and from Lemma~\ref{vep44}  that
$
\overline{d}(\gr{\lambda}|_{X_i\times X_j},\Delta_{X_i\times X_j}(\widetilde{\gr{\nu}}))<\vep.
$
The claim follows from Proposition~\ref{Tzwarte}.
\end{proof}

\begin{pr}\label{pr12-1}
Let $l\in\N$ and let $T$ be an ergodic automorphism of $(X,\cB,\widetilde{\gr{\nu}})$ with a factor $\cA$. Let $\gr{\nu}=\widetilde{\gr{\nu}}|_\cA$. Let $\widetilde{\gr{\nu}}=\frac{1}{l}\sum_{i=1}^{l-1}\widetilde{\nu}\circ T^i$ and $\gr{\nu}=\frac{1}{l}\sum_{i=0}^{l-1}\nu\circ {T^i}|_\cA$ for some $\widetilde{\nu}\in\cP_\Z^e(T^l), \nu\in \cP_\Z^e({T^l}|_\cA)$. Additionally assume that $\widetilde{\nu}|_\cA=\nu$. Then the extension $(T^l,\widetilde{\nu})\to({T^l}|_\cA,\nu)$ is isometric if and only if the extension $(T,\widetilde{\gr{\nu}})\to (T|_\cA,\gr{\nu})$ is isometric.
\end{pr}
\begin{proof}
The proof goes along the same lines as the proof of Proposition~\ref{pr12}.
\end{proof}

\subsubsection{Distal extensions and Furstenberg-Zimmer decomposition}\label{drugidowod}
\paragraph{Distal extensions via isometric extensions}
\begin{pr}\label{pr12a}
Let $\cT=(T_t)_{t\in\R}$ be an ergodic flow on $(X,\cB,\widetilde{\gr{\nu}})$ with a factor $\cA$. Let $\gr{\nu}=\widetilde{\gr{\nu}}|_\cA$. Let $\widetilde{\gr{\nu}}=\int_0^1 \widetilde{\nu}\circ T_t\ dt$ and $\gr{\nu}=\int_0^1 \nu\circ T_t\ dt$ for some $\widetilde{\nu}\in\cP_\Z^e(T_1), \nu\in \cP_\Z^e({T_1}|_\cA).$\footnote{See footnote~\ref{stopka}.} Additionally assume that $\widetilde{\nu}|_\cA=\nu$. Then the extension $(T_1,\widetilde{\nu})\to({T_1}|_\cA,\nu)$ is distal if and only if the extension $(\cT,\widetilde{\gr{\nu}})\to (\cT|_\cA,\gr{\nu})$ is distal.
\end{pr}
\begin{proof}
Suppose that the extension $(\cT,\widetilde{\gr{\nu}})\to (\cT|_\cA,\gr{\nu})$ is distal. Since any factor of $\cT$ is a factor of $T_1$, it clearly follows from the definition of a distal extension and Proposition~\ref{pr12} that also the extension $(T_1,\widetilde{\nu})\to({T_1}|_\cA,\nu)$ is distal.

Suppose now that the extension $(T_1,\widetilde{\nu})\to({T_1}|_\cA,\nu)$ is distal. Let $(\cC_\alpha)_{\alpha\leq \beta}$ be a transfinite sequence of factors of $T_1$ such that $\cC_0=\cA$, $\cC_\beta=\cB$, the extension $\cC_{\alpha+1}\to \cC_\alpha$ for $T_1$ is isometric and for $\alpha$ being limit ordinals, $\cC_\alpha$ is the inverse limit of the preceding factors. We may assume that for each $\alpha$ the factor $\cC_{\alpha+1}$ is the largest extension of $\cC_\alpha$ inside $\cB$ such that $\cC_{\alpha+1}\to\cC_\alpha$ is isometric. In particular, for each $\alpha$, $\cC_\alpha$ is $\cT$-invariant, i.e. it is a factor of $\cT$. Indeed, for $\alpha$ which is not a limit ordinal, ${T_1}|_{\cC_\alpha}$ is isomorphic to ${T_1}|_{T_t\cC_\alpha}$, both $\cC_\alpha\to \cC_{\alpha-1}$, $T_t\cC_\alpha\to \cC_{\alpha-1}$ are isometric, whence $\cC_\alpha\vee T_t\cC_\alpha\to\cC_{\alpha-1}$ is also isometric, which follows e.g. from~\eqref{szesc}.\footnote{$\cC_\alpha\vee T_t\cC_\alpha$ stands for the $\sigma$-algebra generated by $\cC_\alpha$ and $T_t\cC_\alpha$.} For $\alpha$ which is a limit ordinal it suffices to notice that the property of being $\cT$-invariant is preserved under taking inverse limits. This ends the proof by Proposition~\ref{pr12}.
\end{proof}
The same proof works in the co-finite case:
\begin{pr}\label{pr12a1}
Let $T$ be an ergodic automorphism acting on $(X,\cB,\widetilde{\gr{\nu}})$ with a factor $\cA$. Fix $l\in\N$. Let $\gr{\nu}=\widetilde{\gr{\nu}}|_\cA$. Let $\widetilde{\gr{\nu}}=\frac{1}{l} \sum_{i=0}^{l-1}\widetilde{\nu}\circ T_{i}$ and $\gr{\nu}=\frac{1}{l}\sum_{i=0}^{l-1}\nu\circ {T_i}|_\cA$ for some $\widetilde{\nu}\in \cP_\Z^e(T^l), \nu\in \cP_\Z^e(T^l|_\cA)$. Additionally assume that $\widetilde{\nu}|_\cA=\nu$. Then the extension $(T^l,\widetilde{\nu})\to (T^l|_{\cA},\nu)$ is distal if and only if the extension $(T,\widetilde{\gr{\nu}})\to (T|_\cA,\gr{\nu})$ is distal.
\end{pr}

\paragraph{Furstenberg-Zimmer decomposition}
Using the techniques from the proof of Proposition~\ref{pr12a} and arguing by contradiction one can show more:
\begin{wn}
Under the assumptions of Proposition~\ref{pr12a} the following statements are true and equivalent:
\begin{enumerate}[(i)]
\item
The extension $(T_1,\widetilde{\nu})\to({T_1}|_\cA,\nu)$ is distal if and only if the extension $(\cT,\widetilde{\gr{\nu}})\to (\cT|_\cA,\gr{\nu})$ is distal.
\item
The extension $(T_1,\widetilde{\nu})\to({T_1}|_\cA,\nu)$ is relatively weakly mixing if and only if the extension $(\cT,\widetilde{\gr{\nu}})\to (\cT|_\cA,\gr{\nu})$ is relatively weakly mixing.
\item
The Furstenberg-Zimmer decomposition for $(T_1,\widetilde{\nu})\to({T_1}|_\cA,\nu)$ is the same as for $(\cT,\widetilde{\gr{\nu}})\to (\cT|_\cA,\gr{\nu})$, i.e. there exists $\cA\subset\cC\subset\cB$ which is a factor of $\cT$ such that the extension $\cB\to \cC$ is relatively weakly mixing both for $(\cT,\widetilde{\gr{\nu}})$ and for $(T_1,\widetilde{\nu})$ and the extension $\cC\to \cA$ is distal both for $(\cT,\widetilde{\gr{\nu}}|_\cC)$ and for $(T_1,\widetilde{\nu}|_\cC)$.
\end{enumerate}
\end{wn}
In the co-finite case we have the following:
\begin{wn}
Under the assumptions of Proposition~\ref{pr12a1} the following statements are true and equivalent:
\begin{enumerate}[(i)]
\item
The extension $(T^l,\widetilde{\nu})\to (T^l|_{\cA},\nu)$ is distal if and only if the extension $(T,\widetilde{\gr{\nu}})\to (T|_\cA,\gr{\nu})$ is distal.
\item
The extension $(T^l,\widetilde{\nu})\to (T^l|_{\cA},\nu)$ is relatively weakly mixing if and only if the extension $(T,\widetilde{\gr{\nu}})\to (T|_\cA,\gr{\nu})$ is relatively weakly mixing.
\item
Then the Furstenberg-Zimmer decomposition for $(T^l,\widetilde{\nu})\to ({T^l}|_\cA,\nu)$ is the same as for $(T,\widetilde{\gr{\nu}})\to (T|_\cA,\gr{\nu})$, i.e. there exists $\cA\subset \cC\subset \cB$ which is a factor of $T_1$ such that the extension $\cB\to \cC$ is relatively weakly mixing both for $(T,\widetilde{\gr{\nu}})$ and for $(T^l,\widetilde{\nu})$ and the extension $\cC\to \cA$ is distal both for $(T,\widetilde{\gr{\nu}}|_\cC)$ and for $(T^l,\widetilde{\nu}|_\cC)$. 
\end{enumerate}
\end{wn}

\subsection{Quasi- and distal simplicity}
In this section we prove counterparts of Proposition~\ref{pr:simple} and~\ref{pr:jp} for the 2-QS and 2-DS properties.

\subsubsection{Quasi-simplicity}
\begin{pr}\label{QS-cpt}
Let $\cT=(T_t)_{t\in\R}$ be a measure-preserving flow on $(X,\cB,\mu)$ with $T_1$ ergodic. Then $\cT$ is 2-QS if and only if $T_1$ is 2-QS.
\end{pr}
\begin{proof}
Suppose that $\cT$ is 2-QS and let $\lambda\in \cJ^e_2(T_1)$, $\lambda\neq \mu\otimes \mu$. There are two possibilities:
\begin{enumerate}[(i)]
\item
$\lambda=\lambda\circ(T_t\times T_t)\text{ for all }t\in\R,$
\item
$\lambda\perp\lambda\circ(T_t\times T_t)\text{ for some }t\in(0,1).$
\end{enumerate}
In case (i) apply Proposition~\ref{pr12} to $(\cT\times \cT,\lambda)$, $(T_1\times T_1,\lambda)$ and the coordinate factors. In case (ii) consider
$\gr{\lambda}:=\int_0^1\lambda\circ (T_t\times T_t)\ dt$
and apply Proposition~\ref{pr12} to $(\cT\times \cT,\gr{\lambda})$, $(T_1\times T_1,\lambda)$ and the coordinate factors.

Suppose now that $T_1$ is 2-QS. Let $\gr{\lambda}\in \cJ_2^e(\cT)$, $\lambda\neq\mu\otimes \mu$. There are two possibilities:
\begin{multicols}{2}
\begin{enumerate}[(i)]
\item
$\gr{\lambda}\in \cJ_2^e(\cT),$
\item
$\gr{\lambda}\not\in \cJ^e(\cT).$
\end{enumerate}
\end{multicols}
In case (i) apply Proposition~\ref{pr12} to $(\cT\times \cT,\gr{\lambda})$, $(T_1\times T_1,\lambda)$ and the coordinate factors. In case (ii) take $\lambda\in \cJ^e_2(T_1)$ such that $\gr{\lambda}=\int_0^1 \lambda\circ (T_t\times T_t)\ dt$ and apply Proposition~\ref{pr12} to $(\cT\times \cT,\gr{\lambda})$, $(T_1\times T_1,\lambda)$ and the coordinate factors. 
\end{proof}

\begin{pr}\label{QS-cft}
Let $l\geq 1$ and let $T$ be an automorphism of $\xbm$, such that $T^l$ is ergodic. Then $T$ is 2-QS if and only if $T^l$ is 2-QS.
\end{pr}
\begin{proof}
The proof goes along the same lines as the proof of Proposition~\ref{QS-cpt}. It suffices to replace integrals by finite averages and apply Proposition~\ref{pr12-1} instead of Proposition~\ref{pr12}.
\end{proof}

\begin{uw}
In Appendix~\ref{apB} we also give a short proof that, given $\cT$ such that $T_1$ is ergodic, $\cT$ is 2-QS whenever $T_1$ is 2-fold simple.
\end{uw}

\subsubsection{Distal simplicity}
\begin{pr}\label{DS-cpt}
Let $\cT=(T_t)_{t\in\R}$ be a measure-preserving flow on $(X,\cB,\mu)$ with $T_1$ ergodic. Then $\cT$ is 2-DS if and only if $T_1$ is 2-DS.
\end{pr}
\begin{proof}
In the proof of Proposition~\ref{QS-cpt} replace ``Proposition~\ref{pr12}'' with ``Proposition~\ref{pr12a}''.
\end{proof}
\begin{pr}\label{DS-cft}
Let $l\geq 1$ and let $T$ be an automorphism of $\xbm$, such that $T^l$ is ergodic. Then $T$ is 2-DS if and only if $T^l$ is 2-DS.
\end{pr}
\begin{proof}
In the proof of Proposition~\ref{QS-cft} replace ``Proposition~\ref{pr12-1}'' with ``Proposition~\ref{pr12a1}''.
\end{proof}

\section{(Non)-uniqueness of the embedding}\label{uniq}
\subsection{Uniqueness}
We will discuss now the problem of the uniqueness of embedding of automorphisms into measurable flows (an the uniqueness of roots of automorphisms) from the point of view of self-joining properties. 
\begin{uw}
Given a flow $\cT$ the following conditions are equivalent:
\begin{enumerate}[(i)]
\item
for any flow $\cS$ the condition $T_1=S_1$ implies $\cT=\cS$,
\item
for any flow $\cS$ such that $T_1\simeq S_1$, the isomorphism between $T_1\simeq S_1$ is an isomorphism between $\cT$ and $\cS$.
\end{enumerate}
\end{uw}
\begin{uw}\label{nieroz}
For weakly mixing flows $\cT$, $\cS$ such that $T_1\not\perp S_1$, we always have $\cT\not\perp\cS$. Indeed, every non-trivial joining between $T_1$ and $S_1$ lifts to a non-trivial joining between $\cT$ and $\cS$. This holds in particular when $T_1=S_1$.
\end{uw}

\begin{pr}\label{pr:msj}
Let $\cT$ be a weakly mixing, measure-preserving flow on $\xbm$ with the MSJ property. Then $T_1$ has a unique embedding.
\end{pr}

\begin{proof}
Let $\cS=(S_t)_{t\in\R}$ on $\xbm$ be such that $T_1=S_1$. Then $T_1=S_1$ is 2-fold simple by Proposition~\ref{pr:simple}. It follows from Proposition~\ref{pr2qs} that $\cS$ is 2-QS. Since, by Remark~\ref{nieroz}, $\cT\not\perp\cS$, $\cT$ and $\cS$ have a non-trivial common factor (see~\cite{delJ-Lem}). Since $\cT$ has the MSJ property this factor is equal to $\cB$ and there exists $\cC\subset \cB$ invariant under $\cS$, such that $\cS|_\cC=\cT$. On the other hand, $T_1=S_1$ which implies $\cC=\cB$ and the claim follows.
\end{proof}

\begin{pr}
If $T$ is a weakly mixing automorphism with the MSJ property then  for $k\in\N$ the automorphism $T^k$ has a unique root of order $k$.
\end{pr}
\begin{proof}
The proof of Proposition~\ref{pr:msj} can be easily adjusted to the case of automorphisms.
\end{proof}

We will show now a more universal way of proving uniqueness of embedding. It will be the main tool for proving the analogue of Proposition~\ref{pr:msj} for flows with R-property. 

\begin{pr}\label{283a}
Let $\cT$ be a weakly mixing flow on $\xbm$ such that for any weakly mixing flow $\cS$ on $\ycn$ and any $\lambda\in\cJ^e(\cT,\cS)$ either $\lambda=\mu\otimes \nu$ or the fibers in the extension $(\cT\times \cS,\lambda)\to (\cS,\nu)$ are finite. Then for any $\lambda\in\cJ^e(\cT,\cS)$, the flow $(\cT\times \cS,\lambda)$ is weakly mixing.
\end{pr}
\begin{proof}
Let $\lambda\in \cJ^e(\cT,\cS)$. If $\lambda=\mu\otimes\nu$ then $(\cT\times \cS,\lambda)$ is clearly weakly mixing. Suppose that $\lambda\neq\mu\otimes \nu$ and $(\cT\times \cS,\lambda)$ is not weakly mixing. Denote by $\mathcal{K}$ the Kronecker factor of $(\cT\times \cS,\lambda)$. Then, since $(\cT\times \cS,\lambda)$  is not weakly mixing, it has $\cB\otimes \mathcal{K}$ as a factor.\footnote{The symbol $\otimes$ used here does not mean that $\mathcal{K}\subset\cC$. By $\cB\otimes \mathcal{K}$ we denote the sub-$\sigma$-algebra of $\cB\otimes \cC$ generated by $\cB$ and by the Kronecker factor $\mathcal{K}$. The action of $\cT\times \cS$ restricted to $\cB\otimes \mathcal{K}$ is isomorphic to the Cartesian product of $\cT$ and $(\cT\times \cS)|_{\mathcal{K}}$.} Consider the following disintegrations:
\begin{equation*}
\lambda=\int \mu_x\ d\mu, \ \ \ \lambda=\int \lambda_\omega\ d\lambda|_{\cB\otimes \mathcal{K}}, \ \ \ \lambda|_{\cB\otimes \mathcal{K}}=\int\widetilde{\lambda}_z\ d\mu.
\end{equation*}
For $f\in L^1(X\times Y,\cB\otimes \cC,\lambda)$ we have 
$\int f\ d\lambda=\iint f \ d\mu_x \ d\mu.$
On the other hand,
$
\int f\ d\lambda= \iint f\ d \lambda_\omega\ d\lambda|_{\cB\otimes \mathcal{K}}=\iiint f\ d\lambda_\omega\ d\widetilde{\lambda}_z\ d\mu=\iiint f\ d\widetilde{\lambda}_z\ d\lambda_\omega\  d\mu.
$
Since $(\cT\times \cS,\lambda)$ is not weakly mixing, the discrete part of the spectrum forms a countable (infinite) subgroup of $\R$ and the corresponding factor acts by rotations on the dual group to this group. The measures $\widetilde{\lambda}_z$ are equal to the Haar measure on this group, in particular they have infinite supports. It follows by the uniqueness of the disintegration, that also the supports of the measures $\mu_x$ have infinite number of points. This is however impossible by our assumption $\lambda$ is a finite extension of $\mu$, i.e. the measures $\mu_x$ are discrete with a finite number of atoms. The claim follows.
\end{proof}

\begin{uw}\label{finitegroup}
The claim of the above proposition can be strenghtened, using the same arguments: If $\cT$ is a weakly mixing flow such that for any weakly mixing flow $\cS$ and any $\lambda\in \cJ^e(\cT,\cS)$ the flow $(\cT\times \cS,\lambda)$ is also weakly mixing then any finite group extension of $\cT$ has the same property.
\end{uw}

\begin{pr}\label{283b}
Let $\cT$ be a weakly mixing flow with the following property: for any weakly mixing flow $\cS$ and any $\lambda\in \cJ^e(\cT,\cS)$ the flow $(\cT\times \cS,\lambda)$ is weakly mixing. Then $T_1$ has a unique embedding.
\end{pr}
\begin{proof}
Suppose that there exists $\cS$ such that $T_1=S_1$ and consider the diagonal joining of $\Delta\in \mathcal{J}^e(T_1,S_1)$. If $\Delta\not\in \cJ(\cT,\cS)$ then, for some $k\in\N$,
$
\gr{\Delta}:=\int_0^{1/k}\Delta\circ (T_t\times S_t)\ dt\in \cJ^e(\cT,\cS),
$
where the measures $\Delta\circ (T_t\times S_t)$ are pairwise orthogonal for $t\in [0,1/k)$. It follows by Remark~\ref{u:zawie} that $(\cT\times \cS,\gr{\Delta})$ is not weakly mixing which contradicts our assumptions. Hence $\Delta\in \cJ(\cT,\cS)$ which completes the proof.
\end{proof}

By combining Proposition~\ref{283a} with Proposition~\ref{283b} and~\ref{rat}, we obtain the following:
\begin{wn}\label{ratnerpr}
Let $\cT$ be a weakly mixing flow with the R-property. Then $T_1$ has a unique embedding.
\end{wn}

\begin{pr}\label{283c}
Let $\cT$ be a weakly mixing flow with the MSJ property. Then for any weakly mixing flow $\cS$ and any $\lambda\in\cJ^e(\cT,\cS)$ other than the product measure the extension $(\cT\times\cS,\lambda)\to (\cS,\nu)$ has finite fibers.
\end{pr}
Before we begin the proof, let us recall a result on the absence of disjointness and state a necessary lemma.
\begin{tw}[\cite{MR1784644}]\label{parreau}
Let $\cT$ and $\cS$ be ergodic flows. If $\cT$ and $\cS$ are not disjoint then there exists $\lambda\in\cJ_{\infty}^e(\cT)$ such that $(\cT^{\times \infty},\lambda)$ and $\cS$ have a common factor.
\end{tw}
\begin{lm}\label{tensam}
Let $\cT,\widetilde{\cT},\cR$ be ergodic flows. Suppose that the extension $\widetilde{\cT}\to \cT$ is isometric (``k:1''). Let $\widetilde{\rho}\in \cJ^e(\widetilde{\cT},\mathcal{R})$ and let $\rho$ be the restriction of $\widetilde{\rho}$ to a joining between $\cT$ and $\cR$. Then also the extension $(\widetilde{\cT}\times \mathcal{R},\widetilde{\rho})\to (\cT\times \cR,\rho)$ is isometric (``k:1'').
\end{lm}
\begin{proof}
It suffices to show that the claim remains true if we replace the words ``isometric extension'' with ``compact group extension''. This is however clear, as for any compact group extension $\cT_\varphi$, we can write $\cT_\varphi \vee \cR$ as $(\cT\vee\cR)_{\widetilde{\varphi}}$, where $\widetilde{\varphi}(x,y)=\varphi(x)$. In this way we obtain the first part of the claim (about isometric extensions). The more detailed part (about the extensions being ``k:1'') follows from the fact that the group $G$ where $\varphi$ and $\widetilde{\varphi}$ take their values is the same. Moreover, any isometric extension with finite fibers is an intermediate extension for some finite group extension.
\end{proof}

\begin{proof}[Proof of Proposition~\ref{283c}]
Take $\mu\otimes\nu\neq\lambda\in\cJ^e(\cT,\cS)$. It follows from~\cite{MR922364} that there exists $n\in\N$ such that $\cT^{\odot n}$ is a factor of $\cS$,\footnote{By $\cT^{\odot n}$ we denote the symmetric factor of $\cT^{\times n}$, i.e. the sub-$\sigma$-algebra of sets invariant under all permutations of coordinates, see, e.g.,~\cite{MR922364}.} Moreover, the restriction of $\lambda$ to a joining $\cT\vee\cT^{\odot n}$ between $\cT$ and $\cT^{\odot n}$ is such that $\cT\vee\cT^{\odot n}$ is a factor of $\cT^{\times n}$. Since the extension $\cT^{\times n}\to \cT^{\odot n}$ has finite fibers, it follows that also the intermediate extension $\cT\vee \cT^{\odot n}\to \cT^{\odot n}$ has finite fibers (see the proof of Proposition~\ref{283a}). Now, applying Lemma~\ref{tensam} to $\widetilde{\cT}=\cT\vee\cT^{\odot n}$, $\cR=\cS$ and $\rho$ being the joining associated with the factoring map, we obtain that the fibers in the extension $\cT\vee\cS\to \cS$ are finite as well.
\end{proof}

\begin{uw}
Notice that Proposition~\ref{pr:msj} also follows by combining Proposition~\ref{283a}, Proposition~\ref{283b} and~\ref{283c}. In Appendix~\ref{apB} we give yet another proof of Proposition~\ref{pr:msj}.
\end{uw}

\subsection{Non-uniqueness}
Since the R-property implies the 2-QS property, one could ask if a result analogous to Corollary~\ref{ratnerpr} holds also for 2-QS flows. It turns out this is not true. In fact, we ``loose'' the uniqueness of the embedding of the time-one map already in case of 2-fold simple flows.

Let $D_6$ stand for the smallest non-abelian group, i.e. dihedral group of order $6$. Recall that the group operations in $D_6$ can be summarized using Cayley table as follows:
\begin{equation}\label{tabelka}
\begin{tabular}{c|c|c|c|c|c|c}
  $\ast$ & \bf{e} & \bf{a} & \bf{b} & \bf{c} & \bf{d} & \bf{f}\\  
  \hline
  \bf{e} & e &a&	b&	c&	d&	f\\
  \hline
  \bf{a}&	a&	e&	d&	f&	b&	c\\
  \hline
  \bf{b}&	b&	f&	e&	d&	c&	a\\
  \hline
  \bf{c}&	c&	d&	f&	e&	a&	b\\
  \hline
  \bf{d}&	d&	c&	a&	b&	f&	e\\
  \hline
  \bf{f}&	f&	b&	c&	a&	e&	d
\end{tabular}
\end{equation}

\begin{pr}\label{2wlaut}
There exists an automorphism which is 2-fold simple and has two non-isomorphic square roots (one of these roots is 2-fold simple, the other one is only 2-QS).
\end{pr}
\begin{proof}
Let $T$ be a weakly mixing automorphism with the MSJ property and consider $\varphi\colon X\to D_6$ such that $T_\varphi\colon X\times D_6\to X\times D_6$ is also 
weakly mixing. Then $T_\varphi$ is 2-fold simple. For $g\in D_6$ let $\sigma_g\colon X\times D_6 \to X\times D_6$ be given by $\sigma_g(x,h)=(x,hg)$. Notice that for all $g\in D_6$ we have $\sigma_g\in C(T_\varphi)$. From now on we will be using the notation from~\eqref{tabelka} for elements of $D_6$. Since $a^2=e$, we have $(\sigma_a)^2=Id$, whence 
\begin{equation}\label{1.02b}
(T_\varphi\circ \sigma_a)^2=(T_\varphi)^2.
\end{equation}
Moreover, since $a\ast b\neq b\ast a$, $\sigma_b\not\in C(\sigma_a)$, whence $\sigma_b\not\in C(T_\varphi\circ \sigma_a)$. Therefore 
\begin{equation}\label{1.02c}
C(T_\varphi\circ \sigma_a)\neq C(T_\varphi).
\end{equation}
By Proposition~\ref{pr:simple}, we have
\begin{equation}\label{1.02d}
C(T_\varphi)=C((T_\varphi)^2).
\end{equation}
combining~\eqref{1.02b},~\eqref{1.02c} and~\eqref{1.02d} we conclude that $C((T_\varphi\circ \sigma_a)^2)\neq C(T_\varphi\circ \sigma_a)$. It follows by Lemma~\ref{pr:simple} that $T_\varphi$ is 2-fold simple, whereas $T_\varphi\circ \sigma_a$ is not 2-fold simple (it is 2-QS by Proposition~\ref{QS-cft}).
\end{proof}

\begin{pr}\label{rrr}
There exists a 2-fold simple automorphism which is embeddable into two non-isomorphic flows (one of these flows is 2-fold simple and the other one is only 2-QS).
\end{pr}
\begin{proof}
Let $\cT=(T_t)_{t\in\R}$ be a weakly mixing flow on $\xbm$ with the MSJ property. Let $\varphi\colon X\times \R \to SU(2)$ be an $\R$-cocycle such that $\cT_\varphi=({(T_t)}_\varphi)_{t\in\R}$ is weakly mixing. For $g\in SU(2)$ let $\sigma_g(x,h)=(x,hg)$ and let $\mathcal{R}=(R_t)_{t\in\R}$ be given by $R_t=\sigma_{g_t}$, where
$g_t=\left(\begin{matrix}e^{2\pi i t}& 0\\ 0&e^{-2\pi i t} \end{matrix}\right).$
Let $\cT\circ \mathcal{R}:=((T_t)_\varphi \circ R_t)_{t\in\mathbb{R}}$. Then $(T_1)_\varphi\circ R_1=(T_1)_\varphi$. Moreover, for all $g\in SU(2)$ we have $\sigma_g\in C(\cT_\varphi)$. In particular, $\sigma_{h_0}\in C(\cT_\varphi)$, where 
$h_0=\left(\begin{matrix}0& -1\\ 1&0 \end{matrix}\right)$.
On the other hand, easy calculation shows that ${h_0} g_t\neq g_t h_0$ for $t\not\in \nicefrac12\Z$, i.e.\ $\sigma_{h_0}\not\in  C(\mathcal{R})$, whence
$C(\mathcal{T}_\varphi\circ \mathcal{R})\neq C(\mathcal{T}_\varphi)$. Thus, arguing as in the proof of the previous proposition, we obtain $C(\mathcal{T}_\varphi\circ \mathcal{R})\neq C((T_1)_\varphi\circ R_1)$ and $\mathcal{T}_\varphi\circ \mathcal{R}$ is not 2-fold simple (it is 2-QS by Proposition~\ref{QS-cpt}).
\end{proof}

\begin{uw}\label{uw:4.6}
It follows by the above example that the answer to the following question is negative:
\begin{align*}
&\text{Suppose that the extension $\widetilde{\cT}\to\cT$ is such that}\\
&\text{$\widetilde{T}_1\to T_1$ is a compact group extension. Is is true that also}\\
&\text{$\widetilde{\cT}\to\cT$ is a compact group extension?}
\end{align*}
Indeed, if the considered extensions were compact group extensions, then $T_\varphi\circ \sigma_a$ and $\cT\circ \cR$ would be 2-fold simple, which is not the case.

Also the answer to the following question is negative:
\begin{align*}
&\text{Suppose that the extension $\widetilde{T}\to T$ is such that}\\
&\text{$\widetilde{T}^2\to T^2$ is a compact group extension. Is is true that also}\\
&\text{$\widetilde{T}\to T$ is a compact group extension?}
\end{align*}
\end{uw}

\section{Self-joining properties of $\cT$ and $T_1$ -- part II}\label{se5}
A direct consequence of Proposition~\ref{2wlaut} and~\ref{rrr} is that the condition $C(\cT)=C(T_1)$ in condition (ii) in Proposition~\ref{pr:simple} or in its counterpart for automorphisms and their roots cannot be omitted. We will now discuss this problem in more details, providing more counter-examples. Let us however see first, such counter-examples cannot be obtained by taking factors of products of simple systems. In particular, products of horocycle flows, their factors and their weakly mixing distal extensions do not have the desired properties.

\begin{pr}\label{rr}
Let $\cT$ be a 2-fold simple flow with $T_1$ ergodic. Then any factor $\cA$ has the following property: either $\cT|_\cA$ and $T_1|_{\cA}$ are both 2-fold simple or $\cT|_\cA$ and $T_1|_{\cA}$ are both 2-QS, but not 2-fold simple.
\end{pr}
\begin{proof}
Notice that $\cT$ and $T_1$ have the same factors. Moreover, every factor corresponds to a compact subgroup of $C(\cT)=C(T_1)$. If this subgroup is normal, then $\cT$ and $T_1$ restricted to the corresponding factor are 2-fold simple. Otherwise they are both only 2-QS but not $2$-fold simple.
\end{proof}
\begin{wn}\label{273a}
Let $\cT^{(i)}$
 be 2-fold simple, weakly mixing flows. Let $\cA$ be a factor of $\cT^{(1)}\times \cT^{(2)}\times \dots$ which is 2-QS. Then either $(\cT^{(1)}\times \cT^{(2)}\times\dots)|_\cA$ and $(T^{(1)}_1\times T^{(2)}_1\times\dots)|_{\cA}$ are both 2-fold simple or $(\cT^{(1)}\times \cT^{(2)}\times\dots)|_\cA$ and $(T^{(1)}_1\times T^{(2)}_1\times\dots)|_{\cA}$ are both 2-QS, but not 2-fold simple.
\end{wn}
\begin{proof}
This follows immediately from Proposition~\ref{rr} and from the fact that whenever a 2-fold QS system is a factor of a product of two weakly mixing systems, then it is a factor of one of the coordinate factors~\cite{delJ-Lem}.
\end{proof}
Consider the following property of $\cT$: 
\begin{equation}
\text{$T_1$ is simple and $\cT$ is only 2-QS.}\tag{P}
\end{equation}
\begin{wn}\label{273aa}
Let $\cT$ be a weakly mixing flow with the JP property, satisfying the following condition: for any factor $\cA$ of $\cT$ there exists a factor $\cC\subset \cA$ such that $\cT|_\cC$ satisfies property (P). Then $\cT$ is disjoint from weakly mixing extensions of simple flows.
\end{wn}
\begin{proof}
Recall~\cite{MR2729082} that a JP system is disjoint from a weakly mixing simple system if and only if they do not have a common non-trivial factor. Now it suffices to apply Corollary~\ref{273a} and a result from~\cite{MR0213508} on lifting disjointness.
\end{proof}
\begin{uw}
The assumptions of the above corollary are satisfied when $\cT$ satisfies property (P) and has no non-trivial factors.
\end{uw}

\subsection{Easy counter-example}

Del Junco and Rudolph~\cite{MR922364} recall an example by Glasner~\cite{MR719119} of a weakly mixing group extension of a 2-fold simple $\Z$-action, which is itself not 2-fold simple. It is of the form
$
R(x,z_1,z_2)=(Tx,\varphi(x)z_1,z_1z_2),
$
where $T$ is an ergodic automorphism with minimal self-joinings and $\varphi\colon X\to \T$ yields a weakly mixing group extension $T_\varphi$ of $T$. An example of an ergodic self-joining which is ``2:1'' is given in~\cite{MR922364}, showing that this automorphism is not 2-fold simple. We will follow a similar scheme and give an example of a $\Z_2$-extension\footnote{We will identify $\Z_2$ with $\{0,1\}$ with addition modulo $2$.} of a 2-fold simple map such that the obtained automorphism is ``2:1'' and its square is 2-fold simple.

Let us first recall some necessary lemmas and tools. Since we will use them for abelian groups only, we state them in such setting.

\begin{lm}[\cite{MR922364}]\label{l1}
Let $T_\varphi$ be an ergodic compact group extension of a 2-fold simple automorphism $T$. Then $T_\varphi$ is 2-fold simple if and only if for every $S\in C(T)$ there exists $\widetilde{S}\in C(T_\varphi)$ which is an extension of $S$.
\end{lm}

\begin{lm}[$\Z$-actions~\cite{MR527744}, abelian groups~\cite{MR1141362}, general case~\cite{MR1279374}]\label{l2}
Suppose that $T$ is 2-fold simple and let $T_\varphi$ be its compact group extension, i.e. $T_\varphi\colon X\times G\to X\times G$, $T_\varphi(x,g)=(Tx,\varphi(x)+g)$. Let $\widetilde{S}\in C(T_\varphi)$. Then there are a continuous group automorphism $v\colon G\to G$, a measurable map $f\colon X\to G$ and $S\in C(T)$ such that $\widetilde{S}(x,g)=(Sx,f(x)+v(g))$. Equivalently, $\varphi(x)- v(\varphi(x))=f(Tx)- f(x)$.
\end{lm}
\begin{uw}\label{u1}
If $G=\Z_2$ then $v(i)=i$ and $f=const$, i.e. any $\widetilde{S}\in C(T_\varphi)$ is of the form $\widetilde{S}(x,g)=(Sx,g)$ or $\widetilde{S}(x,g)=(Sx,g+1)$ for some $S\in C(T)$.
\end{uw}

\begin{lm}[see e.g.~\cite{katokbook2003}]\label{l3}
Let $T_\varphi$ be an ergodic compact group extension of $T$. Then $c\in\T$ is an eigenvalue of $T_\varphi$ if and only if there exists a measurable function $f\colon X\to \mathbb{T}$ and a character $\chi\in \hat{G}$ such that $\chi\circ\varphi =c\cdot{f\circ T}/{f}$.
\end{lm}

Let $T\colon (X,\cB,\mu)\to (X,\cB,\mu)$ be weakly mixing and have the MSJ property, and let $\varphi\colon X\to \Z_2$ yield a weakly mixing group extension $T_\varphi$. Consider
$\psi\colon X\times \Z_2\to \Z_2$ given by the formula $\psi(x,s)=s$. Let
$$
\overline{T}(x,s,r):=(T_\varphi)_\psi(x,s,r)=(Tx,\varphi(x)+s,s+r).
$$
\begin{lm}\label{lm:Twm}
$\overline{T}$ is weakly mixing.
\end{lm}
\begin{proof}
Suppose that $\overline{T}$ is not weakly mixing. By Lemma~\ref{l3} there exists $c\in\T\setminus\{1\}$, $\xi\colon X\times \Z_2\to \T$ and $\chi\in\widehat{\Z}_2$ satisfying the equation
\begin{equation}\label{rnie}
\chi\circ \psi=c\cdot\xi\circ T_\varphi/\xi.
\end{equation}
There are two possibilities: (a) $\chi\equiv 1$, (b) $\chi(s)=(-1)^s$.

In case (a) equation~\eqref{rnie} takes the form $\xi\circ T_\varphi=\overline{c}\cdot \xi$, i.e. $\overline{c}$ is an eigenvalue of $T_\varphi$. This is however not possible, as $T_\varphi$ is weakly mixing. 

In case (b) equation~\eqref{rnie} takes the form $(-1)^s=c\cdot \xi\circ T_\varphi/\xi$. Hence $\overline{c}^2\cdot\xi^2=\xi^2\circ T_\varphi$ which implies $c^2=1$ and $\xi^2=1$ almost everywhere as $T_\varphi$ is weakly mixing (i.e. we have $|\xi|=1$ almost everywhere). Now, there are again two possibilities: (b1) $c=1$, (b2) $c=-1$. Equation~\eqref{rnie} yields
$$
(-1)^s=\xi\circ T_\varphi / \xi\text{ in case (b1)}\text{ and }(-1)^{s+1}=\xi\circ T_\varphi / \xi\text{ in case (b2)},
$$
i.e. we have
\begin{align*}
&(-1)^s\xi(x,s)=\xi\circ T_\varphi(x,s)=\xi(Tx,\varphi(x)+s) \text{ in case (b1)},\\ 
&(-1)^{s+1}\xi(x,s)=\xi\circ T_\varphi(x,s)=\xi(Tx,\varphi(x)+s) \text{ in case (b2)}. 
\end{align*}
Using the theory of characters ($\xi$ is a square-integrable function)
$$
\xi(x,s)=\xi_1(x)+\xi_2(x)(-1)^s
$$
for some $\xi_j\colon X\to \C$, $j=1,2$. Therefore
\begin{align*}
&(-1)^s(\xi_1(x)+\xi_2(x)\cdot (-1)^s)=\xi_1(Tx)+\xi_2(Tx)\cdot (-1)^{\varphi(x)+s}\text{ in case (b1)},\\ 
&(-1)^{s+1}(\xi_1(x)+\xi_2(x)\cdot (-1)^s)=\xi_1(Tx)+\xi_2(Tx)\cdot (-1)^{\varphi(x)+s}\text{ in case (b2)}, 
\end{align*}
whence
\begin{align*}
&\xi_2(x)-\xi_1(Tx)=(-1)^s(-\xi_1(x)+\xi_2(Tx)\cdot (-1)^{\varphi(x)}) \text{ in case (b1)},\\ 
&-\xi_2(x)-\xi_1(Tx)=(-1)^s(\xi_1(x)+\xi_2(Tx)\cdot (-1)^{\varphi(x)}) \text{ in case (b2)}. 
\end{align*}
Therefore (the above equations hold for all $s$)
\begin{align*}
\xi_2(x)=\xi_1(Tx) \text{ a for almost every }x \text{ in case (b1)},\\ 
\xi_2(x)=-\xi_1(Tx) \text{ a for almost every }x \text{ in case (b2)}. 
\end{align*}
This implies
\begin{align*}
&\xi_1(x)=\xi_2(Tx)\cdot (-1)^{\varphi(x)}=\xi_1(T^2x)\cdot (-1)^{\varphi(x)}\text{ in case (b1)},\\ 
&\xi_1(x)=-\xi_2(Tx)\cdot (-1)^{\varphi(x)}=\xi_1(T^2x)\cdot (-1)^{\varphi(x)}\text{ in case (b2)}, 
\end{align*}
whence $|\xi_1|$ is constant. In either case we have
$$
\frac{\xi_1(x)\cdot \xi_1(Tx)}{\xi_1(Tx)\cdot \xi_1(T^2x)}=\frac{\xi_1(x)}{\xi_1(T^2x)}=(-1)^{\varphi(x)}, 
$$
which means that $\chi\circ \varphi$ is a coboundary. By Lemma~\ref{l3}, this yields a contradiction as $T_\varphi$ is ergodic.
\end{proof}

\begin{pr}\label{niepr}
$\overline{T}$ is not 2-fold simple, whereas $\overline{T}^2$ is 2-fold simple.
\end{pr}
\begin{proof}
To show that $\overline{T}$ is not 2-fold simple, by Lemma~\ref{l1}, it suffices to find an element of $C(T_\varphi)$ which cannot be lifted to an element of $C(\overline{T})$. We claim that $\sigma(x,s)=(x,s+1)$ is such an automorphism. Suppose that we can find $\overline{\sigma}\in C(\overline{T})$ which projects down to $\sigma$. By Lemma~\ref{l2}, we can solve the following equation:
$$
F\circ T_\varphi - F=\psi\circ \sigma-v\circ \psi,
$$
where $v\in Aut(\Z_2)$ and $F\colon X\times \Z_2\to \Z_2$ is measurable. Since the group operations are in $\Z_2$, we have
\begin{multline*}
F\circ T_\varphi(x,s)-F(x,s)=\psi\circ \sigma (x,s) + \psi(x,s)\\
=\psi(x,s+1)+\psi(x,s)=s+1+s=1.
\end{multline*}
Applying to this formula $\chi(r)=(-1)^r$, we obtain
$$
\frac{(-1)^F\circ T_\varphi}{(-1)^F}=-1,
$$
which is not possible as $T_\varphi$ is weakly mixing.

We will now show that $\overline{T}^2$ is 2-fold simple. We have $\overline{T}^2=((T_\varphi)^2)_{\psi+\psi\circ T_\varphi}$. Since $T_\varphi$ is 2-fold simple, Proposition~\ref{pr:simple} implies that $(T_\varphi)^2$ is also 2-fold simple and, by Lemma~\ref{l1}, it suffices to show that all elements of $C((T_\varphi)^2)$ can be lifted to $C(\overline{T}^2)$. By Proposition~\ref{pr:simple}, we have $C((T_\varphi)^2)=C(T_\varphi)$. Therefore (see Remark~\ref{u1}) it suffices to show that $\sigma(x,s)=(x,s+1)$ lifts to an element of $C(\overline{T}^2)$, i.e. we need to find a measurable solution $F\colon X\times\Z_2\to\Z_2$ for the following cocycle equation:
\begin{equation}\label{f1}
\psi^{(2)}\circ \sigma(x,s)+\psi^{(2)}(x,s)=F\circ (T_\varphi)^2 (x,s)+F(x,s),
\end{equation}
where $\psi^{(2)}=\psi+\psi\circ T_\varphi$. However,
\begin{align*}
\psi^{(2)}(x,s+1)&+\psi^{(2)}(x,s)\\
&=\psi(x,s+1)+\psi(T_\varphi(x,s+1))+\psi(x,s)+\psi(T_\varphi(x,s))\\
&=s+1+\varphi(x)+s+1+s+\varphi(x)+s=0,
\end{align*}
so it suffices to take $F\equiv 0$ to complete the proof.
\end{proof}

\subsection{Advanced counter-example}\label{se52}
We will construct a 2-QS flow which is disjoint from simple flows and whose time-one map is 2-fold simple. 

\subsubsection{(C, F)-constructions}
(C,F)-constructions for amenable, unimodular, locally compact second countable (l.c.s.c.) groups were introduced in~\cite{MR2402408} (this is the most general setting appearing in the literature). We recall here this type of constructions, correcting a small error appearing in~\cite{MR2402408}.\footnote{The function $K_m\times D_m^{(n)}\times D_m^{(n)}\in(g,x) \mapsto T_{m,g}^{(n)}x \in R_m^{(n)}$ defined in~\cite{MR2402408}  seems to have a wrong target set. We bypass this problem in~\eqref{eq:16} and~\eqref{bypass}.}

Let $G$ be a unimodular l.c.s.c. group. Given two subsets $E, F\subset G$, by $EF$ we mean their product: $EF = \{ef \colon e \in E, f \in F\}$. If $E=\{e\}$, we write $eF$ for $EF$.  The set $\{e^{-1}\colon e\in E\}$ is denoted by $E^{-1}$. To define a (C, F)-action of $G$ we need two sequences $(F_n)_{n\geq 0}$ and $(C_n)_{n>0}$ of subsets of $G$, satisfying additional conditions:
\begin{align}
&(F_n)_{n\geq 0}\text{ is a F\o lner sequence in }G,\label{w:fo}\\
&C_n \text{ is finite and }\# C_n>1,\label{w:2}\\
&F_nC_{n+1}\subset F_{n+1},\label{w:3}\\
&F_nc\cap F_n c'=\varnothing\text{ for all }c\neq c'\in C_{n+1}.\label{w:4}
\end{align}
We put $X_n:=F_n\times \prod_{k>n}C_k$, endow $X_n$ with the standard Borel product $\sigma$-algebra and define a Borel embedding $X_n\to X_{n+1}$ by setting
\begin{equation}\label{eq:8}
(f_n,c_{n+1},c_{n+2},\dots)\mapsto (f_nc_{n+1},c_{n+2},\dots).
\end{equation}
It is well-defined due to~\eqref{w:3}. Then we have $X_1\subset X_2\subset \dots$. Hence $X:=\cup_{n}X_n$ endowed with the natural Borel $\sigma$-algebra, say $\mathcal{B}$, is a standard Borel space. Given a Borel subset $A\subset F_n$, we put
\begin{equation*}
[A]_n:=\{x\in X\colon x=(f_n,c_{n+1},c_{n+2},\dots)\in X_n\text{ and }f_n\in A\}
\end{equation*}
and call this set a \emph{$n$-cylinder}. It is clear that the $\sigma$-algebra $\mathcal{B}$ is generated by the family of all cylinders.

Now we are going to define a ``canonical'' measure on $(X,\mathcal{B})$. Let $\kappa_n$ stand for the equidistribution on $C_n$ and $\nu_n:=(\# C_1\dots \# C_n)^{-1}\lambda_G|_{F_n}$ on $F_n$. We define a product measure $\mu_n$ on $X_n$ by setting $\mu_n=\nu_n\times \kappa_{n+1}\times \kappa_{n+2}\times \dots$, $n\in\N$. Then the embeddings~\eqref{eq:8} are measure preserving. Hence a $\sigma$-finite measure $\mu$ on $X$ is well defined by the restrictions $\mu|_{X_n}=\mu_n$, $n\in\N$. Since
$$
\mu_{n+1}(X_{n+1})=\frac{\nu_{n+1}(F_{n+1})}{\nu_{n+1}(F_nC_{n+1})}\mu_n(X_n)=\frac{\lambda_G(F_{n+1})}{\lambda_G(F_n)\# C_{n+1}}\mu_n(X_n),
$$
it follows that $\mu$ is finite if and only if
\begin{equation}\label{eq:9}
\prod_{n=0}^{\infty}\frac{\lambda_G(F_{n+1})}{\lambda_G(F_n)\# C_{n+1}}<\infty,\text{ i.e. }\sum_{n=0}^{\infty}\frac{\lambda_G(F_{n+1}\setminus(F_nC_{n+1}))}{\lambda_G(F_n)\# C_{n+1}} <\infty.
\end{equation}
From now on we will assume that~\eqref{eq:9} is satisfied. Moreover, we normalize $\lambda_G$ in such a way that $\mu(X)=1$. 

To construct a $\mu$-preserving action of $G$ on $(X,\mu)$, we fix a filtration $K_1\subset K_2\subset \dots$ of $G=\cup_{m\geq 1}K_m$ by compact sets. Additionally assume that 
\begin{equation}\label{eq:16}
\text{for all }m\geq 1\text{ we have } K_mK_m\subset K_{m+1}.
\end{equation}
Given $n,m\in \N$ set
$$
D_m^{(n)}:=\left(\bigcap_{k\in K_m}(k^{-1}F_n)\cap F_n \right)\times \prod_{l>n}C_l\subset X_n.
$$
It is easy to verity that $D_{m+1}^{(n)}\subset D_m^{(n)}\subset D_{m}^{(n+1)}$. 
We define a Borel mapping
\begin{equation}\label{bypass}
K_m\times D_{m+1}^{(n)}\ni (g,x) \mapsto T_{m,g}^{(n)}x\in D_m^{(n)}
\end{equation}
by setting for $x=(f_n,c_{n+1},c_{n+2},\dots)$
$$
T_{m,g}^{(n)}(f_n,c_{n+1},c_{n+2},\dots):=(gf_n,c_{n+1},c_{n+2},\dots).
$$
Indeed, for $g\in K_m$ and $f_n\in \bigcap_{k\in K_{m+1}}(k^{-1}F_n)\cap F_n$ we obtain that $gf_n\in F_n$ and
$$
gf_n\in \bigcap_{k\in K_{m+1}}gk^{-1}F_n=\bigcap_{k\in K_{m+1}g^{-1}}k^{-1}F_n\subset \bigcap_{k\in K_m}k^{-1}F_n
$$
since by~\eqref{eq:16} we have $K_{m+1}g^{-1}\supset K_m$. Now let $D_m:=\bigcup_{n=1}^{\infty}D_m^{(n)}$. Then a Borel mapping
$$
T_{m,g}\colon K_m\times D_{m+1}\ni (g,x)\mapsto T_{m,g}x\in D_m
$$
is well defined by the restrictions $T_{m,g}|_{D_m^{(n)}}=T_{m,g}^{(n)}$ for $g\in K_m$ and $n\geq 1$. It is clear that $D_m\supset D_{m+1}$ and $T_{m,g}|_{D_{m+2}}=T_{m+1,g}$ for all $m$. If follows from~\eqref{w:fo} that $\mu_n(D_m^{(n)})\to 1$ as $n\to\infty$. Hence $\mu(D_m)=1$ for all $m\in \N$. Finally we set $\widehat{X}:=\bigcap_{m\geq 1}D_m$ and define a Borel mapping $T\colon G\times \widehat{X}\ni(g,x)\mapsto T_gx\in \widehat{X}$ by setting $T_gx:=T_{m,g}s$ for some (and hence any) $m$ such that $g\in K_m$. It is clear that $\mu(\widehat{X})=1$. Thus, we obtain that $\mathfrak{T}=(T_g)_{g\in G}$ is a free Borel measure preserving action of $G$ on a conull subset of the standard probability space $(X,\mathcal{B},\mu)$. It is easy to see that $\mathfrak{T}$ does not depend on the choice of filtration $(K_m)_{m\geq 1}$. 

\begin{df}
$\mathfrak{T}$ is called the (C, F)-action of $G$ associated with $(C_n,F_n)_{n}$.
\end{df}
We now list some basic properties of $\mathfrak{T}$. Given Borel subsets $A,B\subset F_n$ we have
\begin{align}
&[A\cap B]_n=[A]_n\cap [B]_n,\ [A\cup B]_n=[A]_n\cup [B]_n, \nonumber \\
&[A]_n=[AC_{n+1}]_{n+1}=\bigsqcup_{c\in C_{n+1}}[Ac]_{n+1},\nonumber  \\
&T_g[A]_n=[gA]_{n}\ \text{if }gA\subset F_n,\label{y3}  \\
&\mu([A]_n)=\# C_{n+1}\cdot \mu([Ac]_{n+1})\text{ for every }c\in C_{n+1},\label{y4} \\
&\mu([A]_n)=\frac{\lambda_G(A)}{\lambda_G(F_n)}\mu(X_n),\label{y5} 
\end{align}
where $\sqcup$ stands for the union of mutually disjoint sets.

\begin{uw}
(C,F)-actions are of so-called \emph{funny rank one} (see~\cite{MR798894} for $\Z$-actions and~\cite{Sokhet} for the general case). In particular, they are ergodic.
\end{uw}

\subsubsection{Construction of a counter-example}
In this section we will provide an example of a weakly mixing flow $\cT=(T_t)_{t\in\R}$ with $T_1$ which is $2$-fold simple such that $\cT$ is only 2-QS (i.e. not 2-fold simple). Moreover, we will show later that our flow has a factor with the same self-joining properties, which is additionally disjoint from all 2-fold simple flows.

Let $\text{SU}(2)$ stand for the special unitary group of order $2$, i.e. the group of matrices with complex entries which are of the form $\left(\begin{matrix}z& -\overline{\omega}\\ \omega& \overline{z}\end{matrix}\right)$ and whose determinant is equal to one. Let $\lambda_{\text{SU}(2)}$ denote the normalized Haar measure on $\text{SU}(2)$. Let
$
\varphi \colon \R \to \text{Aut}(\text{SU}(2),\lambda_{\text{SU}(2)})
$
be given by
$$
\varphi_t(N)=\left( \begin{matrix}e^{\pi i t/2} & 0 \\ 0 & e^{-\pi i t/2}\end{matrix} \right)N \left( \begin{matrix}e^{-\pi i t/2} & 0 \\ 0 & e^{\pi i t/2}\end{matrix}\right).
$$
We consider a semidirect product $G=\R\ltimes_\varphi \text{SU}(2)$. Recall that multiplication in $G$ is given by
$$
(t,M)(s,N)=(t+s,M\varphi_t(N)).
$$
Notice that
$$
C(G)=\{(2n,I),(2n,-I)\colon n\in\Z\}.
$$
Given $a> 0$ we let
\begin{align*}
&I^\Z[a]:=\{m\in \Z\colon |m|<a\},\\
&I^\Z_+[a]:=I^\Z[a]\cup\{a\},\\
&I^\R[a]:=(-a,a].
\end{align*}
Let $(r_n)_{n\geq 0}$ be an increasing sequence of positive integers such that
\begin{equation}\label{eq:gr}
\lim_{n\to\infty}\frac{n^4}{r_n}=0.
\end{equation}
We will impose later one more restriction on the growth of $(r_n)_{n\geq 0}$. We define recursively two other sequences $(a_n)_{n\geq 0}$ and $(\widetilde{a}_n)_{n\geq 0}$ by setting
\begin{align}
\begin{split}\label{jk1}
&a_0=\widetilde{a}_0=1,\\
&a_{n+1}:=\widetilde{a}_n(2r_n-1),\\
&\widetilde{a}_{n+1}:=a_{n+1}+(2n+1)\widetilde{a}_n
\end{split}
\end{align}
For each $n\in\N$ we let
\begin{align}
\begin{split}\label{jk2}
&H:=\Z,\ H_n:=I^\Z[r_n], \\
&F_n:=I^\R[a_n]\times \text{SU}(2),\ \widetilde{F}_n:=I^\R[\widetilde{a}_n]\times \text{SU}(2),\\
&S_n:=I^\R[(2n-1)\widetilde{a}_{n-1}]\times \text{SU}(2).
\end{split}
\end{align}
We also consider a homomorphism $\phi_n\colon H\to \Z\ltimes\text{SU}(2)\subset G$ given by 
$$
\phi_n(h)=(2h\widetilde{a}_n,I).
$$
Then we have
\begin{align}
&S_n\subset F_n,\ F_nS_n=S_nF_n\subset\widetilde{F}_n\subset G,\label{6-8}\\
&F_{n+1}=\bigsqcup_{h\in H_n}\widetilde{F}_n\phi_n(h)=\bigsqcup_{h\in H_n}\phi_n(h)\widetilde{F}_n,\label{6-9}\\
&S_{n+1}=\bigsqcup_{h\in I^\Z[n+1]}\widetilde{F}_n\phi_n(h)  =\bigsqcup_{h\in I^\Z[n+1]}\phi_n(h)\widetilde{F}_n.\label{6-10}
\end{align}
Fix a sequence $\vep_n\to 0$ as $n\to\infty$. 

For any two finite sets $A,B$ and a map $\phi\colon A\to B$ we define a probability measure on $B$:
$$
\text{dist}_{a\in A}\phi(a):=\frac{1}{\# A}\sum_{a\in A}\delta_{\phi(a)}.
$$
Given two measures $\kappa,\rho$ on a finite set $B$ let 
$$
\left\|\kappa-\rho\right\|_1:=\sum_{b\in B}|\kappa(b)-\rho(b)|.
$$
For $n\in\N$ let $\widehat{S}_n$ be as in Proposition~\ref{techniczny} and such that its assertion holds for sets $A',B'$  which are sums of translations of at most $\# C_1\cdot \ldots\cdot \# C_{n-1}$ elements form $\cA$ (see the remark after Proposition~\ref{techniczny}).

\begin{lm}[\cite{MR2346554}, see also \cite{MR1622315}]\label{lm:6.2}
	If $r_n$ is sufficiently large then there exists a map $s_n\colon H_n\to \widehat{S}_n$ such that for any $\delta\geq n^{-2} r_n$,
	\begin{equation}
		\left\|\emph{dist}_{t\in I^\Z[\delta]}(s_n(h+t),s_n(h'+t),s_n(h''+t))-\lambda_{\widehat{S}_n}\otimes \lambda_{\widehat{S}_n}\otimes \lambda_{\widehat{S}_n}\right\|_1<\vep_n 
	\end{equation}
	whenever $h,h',h''\in H_n$ are pairwise distinct and $\{h,h',h''\}+I^\Z[\delta]\subset H_n$.
\end{lm}
From now on we will assume that $r_n$ is large enough, so that we can use Lemma~\ref{lm:6.2}. For every $n\in\N$ we fix a map $s_n$
whose existence is asserted in the lemma. Without loss of generality we may assume that the following boundary condition holds:
$$
s_n(r_n-1)=s_n(-r_n+1)=0.
$$
Now we can define a map $c_{n+1}\colon H_n\to  F_{n+1}$ by setting $c_{n+1}(h):=s_n(h)\phi_n(h)$. We also put $C_{n+1}:=c_{n+1}(H_n)$.

It is easy to derive from~\eqref{6-8} and~\eqref{6-9} that~\eqref{w:fo},~\eqref{w:2},~\eqref{w:3} and~\eqref{w:4} are satisfied for the sequence $(F_n,C_{n+1})_{n\geq 0}$. Moreover,
\begin{equation}\label{eq:inpart}
\frac{\lambda(F_{n+1})}{\lambda(F_n)\# C_{n+1}}=\frac{2a_{n+1}}{2a_n(2r_n-1)}=\frac{\widetilde{a}_n}{a_n}=1+\frac{(2n-1)\widetilde{a}_{n-1}}{a_n}=1+\frac{2n-1}{2r_{n-1}-1}.
\end{equation}
From this and from~\eqref{eq:gr} we deduce that~\eqref{eq:9} holds. Hence we can consider the associated (C, F)-action $\mathfrak{T}$ of $G$ on
$(X,\cB,\mu)$. 

\paragraph{Auxiliary lemmas}
\begin{lm}\label{6.2}
Let $f=f'\phi_{n-1}(h),\hat{f}=\hat{f}'\phi_{n-1}(\hat{h})$ with $f',\hat{f}'\in \widetilde{F}_{n-1}$ and $h,\hat{h}\in H$. 
\begin{enumerate}[(i)]
\item\label{6.2.i}
Then we have $\widetilde{F}_{n-1}\phi_{n-1}(h+I^\Z[n-1])
\subset fS_n\cup \hat{f}S_n \subset \widetilde{F}_{n-1}\phi_{n-1}(h+I^\Z[n+1]),$
whence
\begin{align*}
\lambda(fS_n\setminus \widetilde{F}_{n-1}\phi_{n-1}(h+I^\Z[n-1]))\leq 4\lambda(\widetilde{F}_{n-1}),\\
\lambda(\hat{f}S_n\setminus \widetilde{F}_{n-1}\phi_{n-1}(h+I^\Z[n-1]))\leq 4\lambda(\widetilde{F}_{n-1}),
\end{align*}
and
$
\lambda(fS_n\triangle \hat{f}S_n)\leq 4\lambda(\widetilde{F}_{n-1}).
$
\item\label{6.2.ii}
If, in addition, $fS_n\subset F_n$ then, for any subset $A\subset F_{n-1}$,
$$
\frac{\lambda(AC_n\cap fS_n)}{\lambda(S_n)}=\lambda_{F_{n-1}}(A)+\overline{o}_n(1).
$$
\end{enumerate}
\end{lm}

\begin{proof}
By~\eqref{6-10} we have
$$
fS_n=f'\phi_{n-1}(h)\widetilde{F}_{n-1}\phi_{n-1}(I^\Z[n])
=f'\widetilde{F}_{n-1}\phi_{n-1}(h+I^\Z[n]).
$$
Since $\widetilde{F}_{n-1}\widetilde{F}_{n-1}\subset \bigsqcup_{|i|\leq 1}\widetilde{F}_{n-1}\phi_{n-1}(i)$, there exists a partition of $\widetilde{F}_{n-1}$ into subsets $A_i$, $|i|\leq 1$ such that $f'A_i\subset \widetilde{F}_{n-1}\phi_{n-1}(i)$ for any $i$. Therefore
\begin{multline}\label{h1}
fS_n=\bigsqcup_{|i\leq 1|}f' A_i\phi_{n-1}(-i)\phi_{n-1}(i+h+I^\Z[n])\\
\subset \bigsqcup_{|i|\leq 1}f'A_i\phi_{n-1}(-i)\phi_{n-1}(h+I^\Z[n+1])=\widetilde{F}_{n-1}\phi_{n-1}(h+I^\Z[n+1])
\end{multline}
since $\bigsqcup_{|i|\leq 1}f'A_i\phi_{n-1}(-i)=\widetilde{F}_{n-1}$. In a similar way we obtain
\begin{equation}\label{h2}
fS_n\supset \widetilde{F}_{n-1}\phi_{n-1}(h+I^\Z[n-1]).
\end{equation}
Clearly,~\eqref{h1} and~\eqref{h2} remain true if we replace $f$ with $\widehat{f}$. The remaining assertions of (i) are direct consequences of~\eqref{h1} and~\eqref{h2}.

Suppose that $fS_n\subset F_n$. Then $K:=h+I^\Z[n-1]\subset H_{n-1}$ and, using (i) we obtain
\begin{align*}
&\frac{\lambda(AC_n\cap fS_n)}{\lambda(S_n)}=\frac{1}{\lambda(S_n)}\left(\lambda(AC_n\cap \widetilde{F}_{n-1}\phi_{n-1}(K)) \pm 4\lambda(\widetilde{F}_{n-1})\right)\\
&=\frac{1}{\lambda(S_n)}\sum_{h'\in H_{n-1}}\left(\lambda(As_{n-1}(h')\phi_{n-1}(h')\cap\widetilde{F}_{n-1}\phi_{n-1}(K))\pm 4\lambda(\widetilde{F}_{n-1}) \right)\\
&=\frac{1}{\lambda(S_n)}\sum_{k\in K}\left(\lambda(As_{n-1}(h'))\pm 4\lambda(\widetilde{F}_{n-1}) \right)
=\frac{\# K}{\lambda(S_n)}\lambda(A)\pm 4\frac{\lambda(\widetilde{F}_{n-1})}{\lambda(S_n)}\\
&=\lambda_{F_{n-1}}(A)\cdot \frac{\lambda(F_{n-1})}{\lambda(\widetilde{F}_{n-1})}\cdot\frac{\lambda(\widetilde{F}_{n-1})}{\lambda(S_n)}\cdot \# K\pm 4\frac{\lambda(\widetilde{F}_{n-1})}{\lambda(S_n)}\\
&=\lambda_{F_{n-1}}(A)\cdot \frac{a_{n-1}}{\widetilde{a}_{n-1}}\cdot \frac{\widetilde{a}_{n-1}\cdot (2n-3)}{(2n-1)\cdot\widetilde{a}_{n-1}}\pm4\frac{\widetilde{a}_{n-1}}{(2n-1)\widetilde{a}_{n-1}}
=\lambda_{F_{n-1}}(A)+\overline{o}_n(1).
\end{align*}
\end{proof}

\begin{lm}\label{lm:fubini}
Let $(G,\lambda)$ be a locally compact unimodular group with Haar measure $\lambda$. For any measurable sets $A,B,F,S\subset G$ of finite measure $\lambda$ we have
$
\int_{S\times S} \lambda (Av\cap Bw\cap F)\ dv\ dw=\int_{A\times B}\lambda(aS\cap bS\cap F)\ da\ db.
$
\end{lm}
\begin{proof}
Elementary calculation yields that for any measurable sets $C,D\subset G$ of finite measure we have $\int_S \lambda(Cv\cap D)\ dv=\int_C \lambda(xS\cap D)\ dx$. Applying this recursively to $\int_{S\times S} \lambda (Av\cap Bw\cap F)\ dv\ dw$ (first for $C=A$, $D=Bw\cap F$ and then for $C=B$, $v=w$ and $D=aS\cap F$), we obtain the desired formula.
\end{proof}

\subsubsection{Weak mixing}
\begin{pr}\label{TisWM}
$T_{(2,I)}$ (and hence also $T_{(1,I)}$) is weakly mixing.
\end{pr}

\begin{proof}
Let $g_n=\phi_n(1)=(2\widetilde{a}_n,I)=(2,I)^{\widetilde{a}_n}$. We will show that $(g_n)_{n\in\N}$ is mixing for $T$, i.e.
\begin{equation}\label{gw1}
\lim_{n\to\infty}\mu(T_{g_n}D\cap D')=\mu(D)\mu(D')
\end{equation}
Since $(T_g)_{g\in G}$ is of funny rank-one, it suffices to prove~\eqref{gw1} for cylinders $D,D'\in\cA$. Let $A=A_{(m)},\ B=B_{(m)}\subset F_m$ be such that $A,B\in\cA$.
We will show now that 
$$
\mu\left(T_{g_n}\left[A\right]_m\cap \left[B\right]_m \right)= \mu\left(\left[A\right]_m\right)\mu\left(\left[B\right]_m\right)+\overline{o}_n(1),
$$
where $\overline{o}_n(1)$ means (here and below) a sequence that goes to $0$ as $n$ goes to infinity and that does not depend on $A,B$.
For $n\geq m$ we have $[A]_m=[A_{(n)}]_n,\ [B]_m=[B_{(n)}]_n$ for $A_{(n)}, B_{(n)}\subset F_n$ such that $A_{(n+1)}=A_{(n)}C_{n+1}$, $B_{(n+1)}=B_{(n)}C_{n+1}$. For $n\geq m$ let
\begin{align*}
F'_n&:=F_n\cap \bigcap_{(i_1,M_1),(i_2,M_2)\in S_n} F_n(i_1,M_1)(i_2,M_2)^{-1},\\
A'_{(n)}&:=A_{(n)}\cap F'_n\\
B'_{(n)}&:=B_{(n)}\cap F'_n,\\
H'_n&:=H_n\cap(H_n-1).
\end{align*}
Notice that 
$$
\mu\left(\left[ A_{(n)}\setminus A'_{(n)}\right]_n\right)\leq \mu\left(\left[F_n\setminus F'_n\right]_n\right) \text{ and }\mu\left(\left[ B_{(n)}\setminus B'_{(n)}\right]_n\right)\leq \mu\left(\left[F_n\setminus F'_n\right]_n\right).
$$
We claim that 
\begin{equation}\label{eq:fo}
\mu\left(\left[F_n\setminus F'_n\right]_n\right)=\overline{o}_n(1).
\end{equation}
Indeed, notice that 
$
F'_n=I^\R[a_n-2(2n-1)\widetilde{a}_{n-1}]\times \text{SU}(2),
$
whence
\begin{align}\begin{split}\label{y6}
	\lambda_{F_n}(F'_n)&=\frac{\lambda(F'_n)}{\lambda(F_n)}=\frac{a_n-2(2n-1)\widetilde{a}_{n-1}}{a_n}=1-2\frac{(2n-1)\widetilde{a}_{n-1}}{a_n}\\
	&=1-2\frac{(2n-1)\widetilde{a}_{n-1}}{ \widetilde{a}_{n-1}(2r_{n-1}-1)}=1-2\frac{2n-1}{2r_{n-1}-1}\to 1
\end{split}\end{align}
by~\eqref{eq:gr}, and~\eqref{eq:fo} follows. Therefore
\begin{align}
	\begin{split}\label{24-2}
		\mu&\left(T_{g_n}\left[A_{(n)}\right]_n\cap \left[B_{(n)}\right]_n \right)=\mu\left(T_{g_n}\left[A'_{(n)}\right]_n\cap \left[B'_{(n)}\right]_n \right)+\overline{o}_n(1)\\
		&=\sum_{h\in H_n}\mu\left(T_{g_n}\left[A'_{(n)}c_{n+1}(h)\right]_{n+1}\cap \left[B'_{(n)}\right]_n \right)+\overline{o}_n(1).
	\end{split}
\end{align}
Notice also that $g_n \in C(G)$ for all $n \in \N$. Hence
\begin{multline*}
g_nAc_{n+1}(h)=g_nAs_n(h)\phi_n(h)=As_n(h)g_n\phi_n(h)\\
=As_n(h)\phi_n(h+1)=As_n(h)s_n(h+1)^{-1}c_{n+1}(h+1)
\end{multline*}
provided that $h,h+1\in H_n$. Moreover, for all $h\in H_n$
$$
\mu\left(T_{g_n}\left[A'_{(n)}c_{n+1}(h)\right]_{n+1} \right)=\frac{\mu\left( \left[ A'_{(n)}\right]_n\right)}{\# C_{n+1}}\leq\frac{\mu\left(\left[ F_n \right]_n\right)}{\# C_{n+1}}=\frac{\mu(X_n)}{\# C_{n+1}}=\overline{o}_n(1).
$$
Therefore
\begin{align}
\begin{split}\label{24-1}
\sum_{h\in H_n}&\mu\left(T_{g_n}\left[A'_{(n)}c_{n+1}(h)\right]_{n+1}\cap \left[B'_{(n)}\right]_n \right)\\
&=\sum_{h\in H'_n}\mu\left(T_{g_n}\left[A'_{(n)}c_{n+1}(h)\right]_{n+1}\cap \left[B'_{(n)}\right]_n \right)\\
&=\sum_{h\in H'_n}\mu\left(\left[A'_{(n)}s_n(h)s_n(h+1)^{-1}c_{n+1}(h+1)\right]_{n+1}\cap \left[B'_{(n)}\right]_n \right).
\end{split}
\end{align}
We also have
$$
A'_{(n)}s_n(h)s_n(h+1)^{-1}c_{n+1}(h+1)\subset F_nc_{n+1}(h+1)\subset F_{n+1}.
$$
Moreover, 
$$
\left[B'_{(n)}\right]_n=\bigsqcup_{\widetilde{h}\in H_n}\left[ B'_{(n)}c_{n+1}(\widetilde{h})\right]_{n+1},
$$
where $B'_{(n)}c_{n+1}(\widetilde{h})\subset F_n c_{n+1}(\widetilde{h})$ and
$$
\left[A_{(n)}'s_n(h)s_n(h+1)^{-1}c_{n+1}(h+1)\right]_{n+1}\cap\left[B'_{(n)}c_{n+1}(\widetilde{h})\right]_{n+1}=\varnothing
$$
when $h+1\neq\widetilde{h}$. Hence
\begin{align}
\begin{split}\label{240}
\sum_{h\in H'_n}\mu &\left(\left[A'_{(n)}s_n(h)s_n(h+1)^{-1}c_{n+1}(h+1)\right]_{n+1}\cap [B'_{(n)}]_n \right)\\
&=\sum_{h\in H'_n}\left(\left[\left(A'_{(n)}s_n(h)s_n(h+1)^{-1}\cap B'_{(n)}\right)c_{n+1}(h+1) \right]_{n+1}\right)\\
&=\frac{1}{\# H'_n}\sum_{h\in H'_n}\mu\left(\left[A'_{(n)}s_n(h)s_n(h+1)^{-1}\cap B'_{(n)}\right]_n\right)\\
&=\frac{1}{\# H'_n}\sum_{h\in H'_n}\lambda_{F_n}\left(A'_{(n)}s_n(h)s_n(h+1)^{-1}\cap B'_{(n)}\right)\mu(X_n)\\
&=\frac{1}{\# H'_n}\sum_{h\in H'_n}\lambda_{F_n}\left(A'_{(n)}s_n(h)\cap B'_{(n)}s_n(h+1) \right)+\overline{o}_n(1)\\
&=\frac{1}{\# H'_n}\sum_{h\in H'_n}\lambda_{F_n}\left(A_{(n)}s_n(h)\cap B_{(n)}s_n(h+1) \right)+\overline{o}_n(1).
\end{split}
\end{align}
Recall that $A_{(n)}=A_{(n-1)}C_n$, $B_{(n)}=B_{(n-1)}C_n$ and $\left[A_{(n-1)}\right]_{n-1}$, $\left[B_{(n-1)}\right]_{n-1}$ are sums of at most $\#C_1\cdot \ldots\cdot \# C_{n-1}$ translations of some elements from $\cA$ (recall that $[A_{(m)}]_{m},\ [B_{(m)}]_m\in \cA$). Let $\xi_n:=\text{dist}_{h\in H'_n}(s_n(h),s_n(h+1))$. It follows from Lemma~\ref{lm:6.2} that
$$
\left\|\xi_n-\lambda_{\widehat{S}_n}\otimes \lambda_{\widehat{S}_n} \right\|<\vep_n=\overline{o}_n(1).
$$
Define $f\colon S_n\times S_n \to \R$ by
$$
f(v,w)=\lambda_{F_n}\left(A_{(n)}v\cap B_{(n)}w\right).
$$
Then
\begin{align}
\begin{split}\label{241}
&\frac{1}{\# H'_n}\sum_{h\in H'_n}\lambda_{F_n}\left(A_{(n)}s_n(h)\cap B_{(n)}s_n(h+1) \right)\\
&=\int_{S_n\times S_n}\lambda_{F_n}\left(A_{(n)}v\cap B_{(n)}w\right)\ d \xi_n(v,w)\\
&=\int f\ d\xi_n=\int f\ d \lambda_{\widehat{S}_n}\otimes \lambda_{\widehat{S}_n}\pm \vep_n\\
&=\int_{S_n\times S_n}\lambda_{F_n}(A_{(n-1)} C_nv\cap B_{(n-1)} C_nw)\ d\lambda_{\widehat{S}_n}\otimes \lambda_{\widehat{S}_n}(v,w)\pm\vep_n.
\end{split}
\end{align}
Hence by the choice of $\widehat{S}_n$ we obtain
\begin{align}\label{242}
\begin{split}
&\int_{S_n\times S_n}\lambda_{F_n}\left(A_{(n-1)} C_nv\cap B_{(n-1)} C_nw\right)\ d\lambda_{\widehat{S}_n}\otimes \lambda_{\widehat{S}_n}(v,w)\\
&=\int_{S_n\times S_n}\lambda_{F_n}\left(A_{(n-1)} C_nv\cap B_{(n-1)} C_nw\right)\ d\lambda_{{S}_n}\otimes \lambda_{{S}_n}(v,w)\pm\vep_n\\
&=\sum_{h,h'\in H_{n-1}}\int_{S_n\times S_n}\frac{\lambda(A_{(n-1)}  c_n(h)v \cap B_{(n-1)} c_n(h')w\cap F_n)}{\lambda(F_n)}\ d\lambda_{S_n}(v)\ d\lambda_{S_n}(w)\pm\vep_n\\
&=\sum_{h,h'\in H_{n-1}}\int_{A_{(n-1)} \times B_{(n-1)}}\frac{\lambda(ac_n(h)S_n\cap bc_n(h') S_n\cap F_n)}{(\lambda({S_n}))^2\lambda(F_n)}\ d\lambda(a)\ d\lambda(b)\pm\vep_n.
\end{split}
\end{align}
For $a\in A_{(n-1)}$, $b\in B_{(n-1)}$ we have
\begin{align}
ac_n(h)=as_{n-1}(h)\phi_{n-1}(h)&\in\widetilde{F}_{n-1}\phi_{n-1}(h),\label{jo:1}\\
bc_n(h')=bs_{n-1}(h')\phi_{n-1}(h')&\in \widetilde{F}_{n-1}\phi_{n-1}(h').\label{jo:2}
\end{align}
Moreover,
\begin{align*}
(0,I)c_n(h)=c_n(h)=s_{n-1}(h)\phi_{n-1}(h)&\in\widetilde{F}_{n-1}\phi_{n-1}(h),\\
(0,I)c_n(h')=c_n(h')=s_{n-1}(h')\phi_{n-1}(h')&\in \widetilde{F}_{n-1}\phi_{n-1}(h').
\end{align*}
Therefore by Lemma~\ref{6.2}~\eqref{6.2.i} we obtain
\begin{multline}\label{jo:3}
\lambda\left((ac_n(h)S_n\cap bc_n(h')S_n\cap F_n)\triangle(c_n(h)S_n\cap c_n(h')S_n\cap F_n) \right)\\
\leq \lambda(ac_n(h)S_n\triangle c_n(hS_n)+\lambda(bc_n(h')S_n\triangle c_n(h')S_n)
\leq 8\lambda(\widetilde{F}_{n-1}).
\end{multline}
By~\eqref{jo:1},~\eqref{jo:2} and, again, Lemma~\ref{6.2}~\eqref{6.2.i}
\begin{equation*}
ac_n(h)S_n\cap bc_n(h')S_n\subset \widetilde{F}_{n-1}\phi_{n-1}(h+I^\Z[n+1])\cap \widetilde{F}_{n-1}\phi_{n-1}(h'+I^\Z[n+1]).
\end{equation*}
Hence $ac_n(h)S_n\cap bc_n(h')S_n\neq \varnothing$ only if $h'-h\in I^\Z[2n+1]$. If the latter is satisfied, we say that $h$ and $h'$ are \emph{partners}. Denote by $P(h)$ the set of partners of $h$ from $H_{n-1}$. Clearly, $\# P(h)\leq 4n+1$. Therefore we deduce from~\eqref{24-2}, \eqref{24-1}, \eqref{240}, \eqref{241}, \eqref{242} 
that
\begin{align}
\begin{split}\label{eq:6}
&\mu(T_{g_n}[A]_{m}\cap [B]_{m})\\
&\phantom{===}=\sum_{h\in H_{n-1}}\sum_{h'\in P(h)}\frac{\lambda(A_{(n-1)})\lambda(B_{(n-1)})}{(\lambda(S_n))^2\lambda(F_n)}\\
&\phantom{======}\cdot\left(\lambda(c_n(h)S_n\cap c_n(h')S_n\cap F_n)\pm 8\lambda(\widetilde{F}_{n-1}) \right)\\
&\phantom{===}=\frac{\lambda(A_{(n-1)})\lambda(B_{(n-1)})}{(\lambda(S_n))^2\lambda(F_n)}\\
&\phantom{======}\cdot\sum_{h\in H_{n-1}}\sum_{h'\in P(h)}\lambda(c_n(h)S_n\cap c_n(h')S_n\cap F_n)\\
&\phantom{======}\pm \lambda(H_{n-1})(4n+1)\frac{(\lambda(F_{n-1}))^2}{(\lambda(S_n))^2\lambda(F_n)}\cdot 4\lambda(\widetilde{F}_{n-1})\\
&\phantom{===}=\frac{\lambda(A_{(n-1)})\lambda(B_{(n-1)})}{(\lambda(F_{n-1}))^2}\theta_n\\
&\phantom{======}\pm \lambda(H_{n-1})(4n+1)\frac{(\lambda(F_{n-1}))^2}{(\lambda(S_n))^2\lambda(F_n)}\cdot 4\lambda(\widetilde{F}_{n-1})\\
&\phantom{===}=\lambda_{F_{n-1}}(A_{(n-1)})\lambda_{F_{n-1}}(B_{(n-1)})\theta_n\\
&\phantom{======}\pm \lambda(H_{n-1})(4n+1)\frac{(\lambda(F_{n-1}))^2}{(\lambda(S_n))^2\lambda(F_n)}\cdot 4\lambda(\widetilde{F}_{n-1}),
\end{split}
\end{align}
where $\theta_n>0$.  Notice that~\eqref{eq:inpart} together with~\eqref{eq:gr} means that $a_n/\widetilde{a}_n\to 1$ as $n\to \infty$. Therefore, by~\eqref{jk1} and~\eqref{jk2} we have
\begin{align*}
\lambda(H_{n-1})(4n+1)&\frac{(\lambda(F_{n-1}))^2}{(\lambda(S_n))^2\lambda(F_n)}\cdot 4\lambda(\widetilde{F}_{n-1})\\
&=\frac{(2r_{n-1}-1)(4n+1)(2a_{n-1})^2\cdot 4\cdot 2\widetilde{a}_{n-1}}{((2n-1)\widetilde{a}_{n-1})^2\cdot 2a_n}\\
&=\frac{16(4n+1)}{2n-1}\cdot\left(\frac{a_{n-1}}{\widetilde{a}_{n-1}}\right)^2\cdot\frac{\widetilde{a}_{n-1}(2r_{n-1}-1)}{a_n(2n-1)}\\
&=\frac{16(4n+1)}{2n-1}\cdot\left(\frac{a_{n-1}}{\widetilde{a}_{n-1}}\right)^2\cdot\frac{1}{2n-1}=\overline{o}_n(1).
\end{align*}
Therefore substituting $A_{(n-1)}=B_{(n-1)}=F_{n-1}$ in~\eqref{eq:6}  (see Remark~\ref{tech}) and passing to the limit we obtain $\theta_n\to 1$ as $n\to\infty$. Therefore, using~\eqref{y5}, we obtain
\begin{equation}\label{eq:7}
\mu(T_{g_n}[A]_{m}\cap[B]_{m})=\mu([A]_{m})\mu([B]_{m})+\overline{o}_n(1)
\end{equation}
and the assertion follows.
\end{proof}

\subsubsection{Self-joining properties}
For $k=(t,M)\in G$ let $k^\ast=(1,I)(t,M)(1,I)^{-1}=(t,\varphi_{1}(M))$. Notice that $(k^{\ast})^\ast=k.$\label{fffgg}

\begin{pr}\label{2:1}
The transformation $T_{(1,I)}$ is ``2:1''. Moreover, 
$\cJ_2^e(T_{(1,I)})=\left\{ \frac{1}{2}\left(\mu_{T_k}+\mu_{T_{k^\ast}}\right)\colon k\in G\right\}.$
\end{pr}
\begin{proof}
Take any joining $\nu\in J^e_2(T_{(1,I)})$. Let $I_n:=I^\Z[n^{-2}a_n]$, $J_n=I^\Z[n^{-2}r_n]$ and $\Phi_n=I_n+2\widetilde{a}_nJ_n$. We first notice that $(\Phi_n)_{n=1}^{\infty}$ is a F\o lner sequence in $\Z$. Since
$$
\frac{a_n}{n^2}+\frac{2\widetilde{a}_nr_n}{n^2}<\frac{\widetilde{a}_n(2r_n+1)}{n^2}<\frac{2a_{n+1}}{(n+1)^2},
$$
it follows that $\Phi_n\subset I_{n+1}+I_{n+1}$, whence $\cup_{m=1}^{n}\Phi_m\subset I_{n+1}+I_{n+1}$. This implies
$$
\# (\Phi_{n+1}-\cup_{m=1}^{n}\Phi_m)\leq 3\# \Phi_{n+1}
$$
for every $n\in\N$, i.e. Shulman's condition~\cite{MR1865397} is satisfied for $(\Phi_n)_{n=1}^{\infty}$. By~\cite{MR1865397}, the pointwise ergodic theorem holds along $(\Phi_n)_{n=1}^{\infty}$ for any ergodic transformation. Since $T_{(1,I)}$ is ergodic by Proposition~\ref{TisWM}, we have
\begin{equation}\label{6-15}
\frac{1}{\# \Phi_n} \sum_{i\in \Phi_n} \raz_D(T_{(i,I)}x)\raz_{D'}(T_{(i,I)}x')\to \nu(D\times D')
\end{equation}
as $n\to\infty$ for $\nu$-a.a. $(x,x')\in X\times X$ and for all cylinders $D,D'\subset X$. We call such $(x,x')$ a generic point for $(T_{(1,I)}\times T_{(1,I)},\nu)$. Fix one of them. 
Then $x,x'\in X_n$ for all sufficiently large $n$ and we have the following expansions:
\begin{align*}
x&=(f_n, c_{n+1}(h_n),c_{n+2}(h_{n+1}),\dots),\\
x'&=(f'_n, c_{n+1}(h'_n),c_{n+2}(h'_{n+1}),\dots)
\end{align*}
with $f_n,f'_n\in F_n$ and $h_i,h'_i\in H_i$, $i>n$. We let $H_n^-:=I^\Z[(1-n^{-2})r_n]\subset H_n$. Then
$$
\frac{\# H_n^-}{\# H_n}\geq 1-n^{-2}.
$$
Since the marginals of $\nu$ are both equal to $\mu$, by the Borel-Cantelli lemma we may assume without loss of generality that $h_n,h_n'\in H_n^-$ for all sufficiently large $n$. This implies, in turn, that
$$
f_{n+1}=f_nc_{n+1}(h_n)\in \widetilde{F}_n\phi_n(H_n^-)\subset I_+^{\Z}[(2r_n(1-n^{-2})-1)\widetilde{a}_n]\times SU(2)
$$
and, similarly, 
$$
f'_{n+1}\in I_+^{\Z}[(2r_n(1-n^{-2})-1)\widetilde{a}_n] \times SU(2).
$$
Notice that given $g\in \Phi_n$ we have
$
(g,I)=(b,I)\phi_n(t)
$
for some uniquely determined $b\in I_n$ and $t\in J_n$. Moreover, $t+h_n\in H_n$. We also claim that
\begin{equation}\label{2-10}
(b,I)f_nS_nS_n=(b,I)f_nS_nS_n^{-1}\subset F_n.
\end{equation}
To verify this, it suffices to show that
$$
\frac{a_n}{n^2}+\left(2r_{n-1}\left(1-\frac{1}{(n-1)^2}\right)-1\right)\widetilde{a}_{n-1}+2(2n-1)\widetilde{a}_{n-1}<a_n.
$$
We will show that the following stronger inequality holds for $n$ is large enough:
$$
\frac{a_n}{n^2}+2r_{n-1}\left(1-\frac{1}{(n-1)^2}\right)\widetilde{a}_{n-1}+4n\widetilde{a}_{n-1}<a_n,
$$
i.e.
$$
2r_{n-1}\left( 1-\frac{1}{(n-1)^2}\right)+4n<(2r_{n-1}-1)\left(1-\frac{1}{n^2} \right).
$$
The latter inequality, in turn, is equivalent to 
$
1+4n-\frac{1}{n^2}<2r_{n-1}\left(\frac{1}{(n-1)^2}-\frac{1}{n^2} \right),
$
which follows from~\eqref{eq:gr} in a routine way. Hence 
\begin{align*}
&(g,I)f_n s_n(h_n)\phi_n(h_n)=dc_{n+1}(t+h_n),\\
&(g,I)f'_n s_n(h'_n)\phi_n(h'_n)=d'c_{n+1}(t+h'_n),
\end{align*}
where $d=(b,I)f_ns_n(h_n)s_n(t+h_n)^{-1},
d'=(b,I)f'_ns_n(h'_n)s_n(t+h'_n)^{-1} \in F_n$
by~\eqref{2-10}. 

We consider separately two cases:
\begin{enumerate}[(i)]
\item $h_n=h'_n$ for all $n\geq N$. Then $f_n {f_n'}^{-1}=f_N {f'_N}^{-1}=:k$ for all $n\geq N$. Then 
\begin{equation*}
(b,I)k(b,I)^{-1}=
	\begin{cases}
	k, &\text{ when $b$ is even}\\
	k^\ast, &\text{ when $b$ is odd}.
	\end{cases}
\end{equation*}
\item $h_n\neq h'_n$ for infinitely many $n$. 
\end{enumerate}
Our goal is to prove that 
\begin{equation*}
	\nu=
	\begin{cases}
		\frac{1}{2}(\mu_{T_k}+\mu_{T_{k^\ast}}),& \text{ in case (i)}\\
		\mu\otimes \mu,& \text{ in case (ii)}\\
	\end{cases}
\end{equation*}
Notice that it suffices to show the corresponding equalities of measures restricted to the classes of cylinder sets from $\cA$. We may also impose additional assumptions on the cylinders under consideration, provided that every set can still be approximated by finite sums of such cylinders. 

Take $[A_{(m)}]_m,\ [B_{(m)}]_m \in \cA$ and for $n\geq m$ set 
$$
A_{(n+1)}:=A_{(n)}C_{n+1},\ B_{(n+1)}:=B_{(n)}C_{n+1}.
$$
We have
\begin{align}
\begin{split}\label{7.03e}
&\frac{1}{\# \Phi_n}\# \left\{g\in\Phi_n\colon \left(T_{(g,I)}x,T_{(g,I)}x'\right)\in [A_{(m)}]_m\times [B_{(m)}]_m \right\}\\
&\phantom{==}=\frac{1}{\# \Phi_n}\# \left\{g\in\Phi_n\colon \left(T_{(g,I)}x,T_{(g,I)}x'\right)\in [A_{(n)}]_n\times [B_{(n)}]_n \right\}\\
&\phantom{==}=\frac{1}{\# I_n}\sum_{b\in I_n}\frac{\# \left\{t\in J_n\colon d\in A_{(n)},d'\in B_{(n)} \right\}}{\# J_n}\\
&\phantom{==}=\frac{1}{\# I_n}\sum_{b\in I_n}\xi_n\left(A_{(n)}^{-1}(b,I)f_ns_n(h_n)\times B_{(n)}^{-1}(b,I)f'_ns_n(h'_n) \right),
\end{split}
\end{align}
where $\xi_n=\text{dist}_{t\in J_n}\left(s_n(t+h_n),s_n(t+h'_n) \right)$.

We cover first the case (i). Take a cylinder $D'\in \cA$ of order $m$ which is a finite sum of cylinders from $\cA$. It is clear that $\widetilde{D}':=D'\cap [k^{-1}F_m]_m$ is also a finite sum of cylinders from $\cA$. Moreover,
\begin{multline*}
\mu(D'\triangle \widetilde{D}')=\mu([\sqcup_{j=1}^{j_0}A_j]_m\setminus [k^{-1}F_m]_m)\\
=\frac{\lambda(\sqcup_{j=1}^{j_0}A_j\setminus k^{-1}F_m)}{\lambda(F_m)}\mu(X_m)\leq\frac{\lambda(F_m\setminus k^{-1}F_m)}{\lambda(F_m)}\\
=\mu\left(\left[F_m\setminus k^{-1}F_m  \right]_m\right).
\end{multline*}
Given $\vep$, for $m$ large enough, $\mu\left(\left[F_m\setminus k^{-1}F_m  \right]_m\right)<\vep$ since $(I^\R[a_m])_{m\in\R}$ is a F\o lner sequence in $\R$. Therefore we may additionally assume that every cylinder $[B]_m\in\cA$ satisfies the condition $kB\subset F_m$ if $m$ is large enough. Since $kF_m=k^\ast F_m$, it follows immediately that also $k^\ast B\subset F_m$.

It is easy to deduce from Lemma~\ref{lm:6.2} that $\|\xi_n-\Delta\|_1<\vep_n$ for $n>N$, where $\Delta$ is the probability equidistributed on $\widehat{S}_n\times \widehat{S}_n$. This yields (having in mind how we have chosen $\widehat{S}_n$)
\begin{align}
\begin{split}\label{7.03d}
&\frac{1}{\# I_n}\sum_{b\in I_n}\xi_n\left(A_{(n)}^{-1}(b,I)f_ns_n(h_n)\times B_{(n)}^{-1}(b,I)f'_ns_n(h'_n)\right)\\
&\phantom{=}=\frac{1}{\# I_n}\sum_{b\in I_n}\lambda_{\widehat{S}_n}\left(A_{(n)}^{-1}(b,I)f_ns_n(h_n)\cap B_{(n)}^{-1}(b,I)f'_ns_n(h_n) \right)\pm\vep_n\\
&\phantom{=}=\frac{1}{\# I_n}\sum_{b\in I_n}\lambda_{{S}_n}\left(A_{(n)}^{-1}(b,I)f_ns_n(h_n)\cap B_{(n)}^{-1}(b,I)f'_ns_n(h_n) \right)\pm2\vep_n
\end{split}
\end{align}
We obtain further
\begin{align}
\begin{split}\label{7.03c}
&\frac{1}{\# I_n}\sum_{b\in I_n}\lambda_{{S}_n}\left(A_{(n)}^{-1}(b,I)f_ns_n(h_n)\cap B_{(n)}^{-1}(b,I)f'_ns_n(h_n) \right)\\
&\phantom{==}=\frac{1}{\# I_n}\sum_{b\in I_n}\frac{\lambda\left(A_{(n)}\cap (b,I)f_n{f'}_n^{-1}(b,I)^{-1}B_{(n)}\cap (b,I)f_ns_n(h_n)S_n \right)}{\lambda(S_n)}\\
&\phantom{==}=\frac{1}{\# I_n}\sum_{2b\in I_n}\frac{\lambda\left(A_{(n)}\cap kB_{(n)}\cap (2b,I)f_ns_n(h_n)S_n\right)}{\lambda(S_n)}\\
&\phantom{===}+\frac{1}{\# I_n}\sum_{2b+1\in I_n}\frac{\lambda\left(A_{(n)}\cap k^\ast B_{(n)}\cap (2b,I)f_ns_n(h_n)S_n\right)}{\lambda(S_n)}.
\end{split}
\end{align}
Since $kB_{(m)}\subset F_m$, we have $kB_{(n)}\subset F_n$.
Therefore
$
A_{(n-1)}C_n\cap kB_{(n-1)}C_n=\left(A_{(n-1)}\cap kB_{(n-1)}\right)C_n.
$
It follows from Lemma~\ref{6.2}~\eqref{6.2.ii} and from~\eqref{2-10} that
\begin{align}
\begin{split}\label{7.03a}
&\frac{\lambda\left(A_{(n)}\cap k B_{(n)} \cap (2b,I)f_ns_n(h_n)S_n \right)}{\lambda(S_n)}\\
&\phantom{======}=\frac{\lambda\left(\left(A_{(n-1)}\cap kB_{(n-1)}\right)C_n \cap (2b,I)f_ns_n(h_n)S_n \right)}{\lambda(S_n)}\\
&\phantom{======}=\lambda_{F_{n-1}}\left(A_{(n-1)}\cap kB_{(n-1)} \right)+\overline{o}_n(1)\\
&\phantom{======}=\frac{\mu\left(\left[A_{(n-1)}\cap k B_{(n-1)} \right]_{n-1} \right)}{\mu(X_{n-1})}+\overline{o}_n(1)\\
&\phantom{======}=\mu\left(\left[ A_{(m)}\right]_m \cap T_k \left[B_{(m)} \right]_m \right)+\overline{o}_n(1).
\end{split}
\end{align}
In a similar way we obtain
\begin{align}
\begin{split}\label{7.03b}
&\frac{\lambda\left(A_{(n)}\cap k^\ast B_{(n)} \cap (2b,I)f_ns_n(h_n)S_n \right)}{\lambda(S_n)}\\
&\phantom{============}=\mu\left(\left[ A_{(m)}\right]_m \cap T_{k^\ast} \left[B_{(m)} \right]_m \right)+\overline{o}_n(1).
\end{split}
\end{align}
If follows from~\eqref{6-15},~\eqref{7.03e},~\eqref{7.03d},~\eqref{7.03c},~\eqref{7.03a} and~\eqref{7.03b} that
$
\nu=\frac{1}{2}(\mu_{T_k}+\mu_{T_{k^\ast}}).
$

We consider now case (ii). It follows from Lemma~\ref{lm:6.2} that $\left\|\xi_n-\lambda_{\widehat{S}_n}\otimes \lambda_{\widehat{S}_n}\right\|_1<\vep_n$ for all such $n$. Therefore
\begin{align}
\begin{split}\label{73c}
&\frac{1}{\# I_n}\sum_{b\in I_n}\xi_n\left(A_{(n)}^{-1}(b,I)f_ns_n(h_n)\times B_{(n)}^{-1}(b,I)f'_ns_n(h'_n)\right)\\
&\phantom{=}=\frac{1}{\# I_n}\sum_{b\in I_n}\lambda_{\widehat{S}_n}\left(A_{(n)}^{-1}(b,I)f_ns_n(h_n)\right)\lambda_{\widehat{S}_n}\left(B_{(n)}^{-1}(b,I)f'_ns_n(h'_n)\right)\pm \vep_n\\
&\phantom{=}=\frac{1}{\# I_n}\sum_{b\in I_n}\lambda_{S_n}\left(A_{(n)}^{-1}(b,I)f_ns_n(h_n)\right)\lambda_{S_n}\left(B_{(n)}^{-1}(b,I)f'_ns_n(h'_n)\right)\pm 2\vep_{n},
\end{split}
\end{align}
where the last equality follows from our choice of $\widehat{S}_n$. Now we derive from Lemma~\ref{6.2}~\eqref{6.2.ii} that
\begin{multline}\label{73a}
\lambda_{S_n}\left(A_{(n)}^{-1}(b,I)f_ns_n(h_n)\right)=\frac{\lambda\left(A_{(n)}\cap (b,I)f_ns_n(h_n)S_n \right)}{\lambda(S_n)}\\
=\lambda_{F_{n-1}}\left(B_{(n-1)}\right)+\overline{o}_n(1)=\frac{\mu\left(\left[ A_{(n-1)}\right]_{n-1} \right)}{\mu(X_n)}=\mu\left([A_{(m)}]_m \right)+\overline{o}_n(1)
\end{multline}
and
\begin{equation}\label{73b}
\lambda_{S_n}\left(B_{(n)}^{-1}(b,I)f'_ns_n(h'_n) \right)=\mu\left([B_{(m)}]_m \right)+\overline{o}_n(1).
\end{equation}
It follows from~\eqref{7.03e},~\eqref{73c},~\eqref{73a} and~\eqref{73b} that
$
\nu=\mu\otimes \mu,
$
which ends the proof.
\end{proof}

\begin{pr}
$T_{(1,I)}$ is 3-fold PID.
\end{pr}
\begin{proof}[Sketch of the proof]
Let $\nu\in\cJ^e_3(T_{(1,I)})$ be pairwise independent. We will show that $\nu=\mu^{\otimes 3}$. To this end, we will follow the same path as in the proof of Proposition~\ref{2:1}. Fix a generic point $(x,x',x'')$ for $(T_{(1,I)}^{\times 3},\nu)$. Then $x,x',x''\in X_n$ for all sufficiently large $n$. Without loss of generality we may assume that
\begin{align*}
x&=(f_n,c_{n+1}(h_n),c_{n+1}(h_{n+1}),\dots),\\
x'&=(f_n,c_{n+1}(h'_n),c_{n+1}(h'_{n+1}),\dots),\\
x''&=(f_n,c_{n+1}(h''_n),c_{n+1}(h''_{n+1}),\dots),
\end{align*}
with $f_n,f'_n,f''_n\in F_n$ and $h_i,h'_i,h''_i\in H_i^-$ for $i>n$. Now take 
$$[A_{(m)}]_m,[A'_{(m)}]_m,[A''_{(m)}]_m\in \cA$$ 
and set $A_{(n+1)}=A_{(n)}C_{n+1},\ A'_{(n+1)}=A'_{(n)}C_{n+1},\ A''_{(n+1)}=A''_{(n)}C_{n+1}$ for $n\geq m$. Then as in the proof of previous proposition we have
\begin{multline}\label{6-17}
\frac{1}{\# \Phi_n}\# \left\{g\in \Psi_n\colon (T_{g,I}x,T_{g,I}x',T_{g,I}x'')\in [A_{(m)}]_m\times [A'_{(m)}]_m\times [A''_{(m)}]_m \right\}\\
=\frac{1}{\# I_n} \sum_{b\in I_n} \xi_n ( A_{(n)}^{-1} d \times {A'_{(n)}}^{-1}d\times {A''_{(n)}}^{-1}d),
\end{multline}
where $d=(b,I)f_ns_n(h_n)$, $d'=(b,I)f'_ns_n(h'_n)$, $d=(b,I)f''_ns_n(h''_n)$ and
$$
\xi_n=dist_{t\in J_n}\left(s_n(t+h_n),s_n(t+h'_n),s_n(t+h''_n)\right).
$$
Next, given $k>0$,
\begin{multline*}
\frac{\mu\otimes \mu\left(\left\{(y,y')\in X_k\times X_k\colon y_i\neq y_i' \text{ for all }i>k \right\} \right)}{\mu(X_k)\mu(X_k)}\\
=\prod_{i>k}\frac{\# C_i^2-\#C_i}{\# C_i}\prod_{i>k}\left(1-\frac{1}{\# H_i} \right)>0,
\end{multline*}
where $(y_i)_{i\geq k}$, $(y'_i)_{i\geq k}$ are the ``coordinates'' of $y,y'\in X_k=F_k\times C_{k+1}\times C_{k+2}\times \dots$ respectively. Since $\nu$-almost every point is generic for $(T_{(1,I)}^{\times 3},\nu)$, we can select $(x,x',x'')$ in such a way that $h_i,h'_i,h''_i$ are pairwise distinct for all $i\geq n$ whenever $n$ is large enough. Therefore, it follows from Lemma~\ref{lm:6.2} that $\|\xi_i-\lambda_{\widehat{S}_i}\otimes \lambda_{\widehat{S}_i} \otimes \lambda_{\widehat{S}_i}\|<\vep_i$ for all $i\geq n$. Now we derive from~\eqref{6-17} and Lemma~\ref{6.2} (having in mind how we have chosen the sets $\widehat{S}_i$) that
$$
\nu([A_{(m)}]_m\times [A'_{(m)}]_m\times [A''_{(m)}]_m)=\mu([A_{(m)}]_m)\mu([A'_{(m)}]_m)\mu([A''_{(m)}]_m)+\overline{o}_n(1).
$$
This implies $\nu=\mu^{\otimes 3}$.
\end{proof}

\paragraph{Remarks}
\begin{enumerate}
\item
In view of Proposition~\ref{2:1} and the definition of $k^\ast$ (see page~\pageref{fffgg}), one can expect that $T_{(2,I)}$ is 2-fold simple. This is indeed the case (the proof uses similar arguments as the proof of Proposition~\ref{2:1}). Moreover, since $T_{(1,I)}$ is 3-fold PID, also $T_{(2,I)}$ is 3-fold PID by~\cite{MR2346554}. It follows by~\cite{MR1194963} that $T_{(2,I)}$ is simple (and that $T_{(1,I)}$ is also PID).
\item
It follows from Proposition~\ref{2:1} and~\ref{QS-cpt} that the constructed example has as its sub-action the flow $(T_{(t,I)})_{t\in\R}$ which is only QS (it is not simple) with a simple time-2 map $T_{(2,I)}$. However, we cannot say at this stage yet, that the flow under consideration is disjoint from simple flows (cf. Corollary~\ref{273aa}) -- it is clear only that it has no factors which would be simple  (see the next remark).
\item
Recall that since $T_{(2,I)}$ is simple, any of its factors is determined by a compact subgroup of $C(T_{(2,I)})=G$. The largest compact subgroup of $G$ is $\{0\}\times SU(2)\simeq SU(2)$ and it is normal in $G$. It determines the smallest non-trivial factor of $T_1$, we will denote it by $\cA(SU(2))=\cA$.  Notice that $\cA$ is also a factor of $\cT$. Take $g\in G\setminus \{0\}\times SU(2)$. Then $\frac{1}{2}(\mu_{T_g}+\mu_{T_{g^\ast}})|_{\cA\otimes \cA}$ is a 2-fold ergodic joining of $\cT|_{\cA}$ and it is not difficult to see that it is ``2:1'' over its marginals (this argument is taken from~\cite{MR2346554}). Since $\cT|_{\cA}$ has no non-trivial factors, in view of Corollary~\ref{273aa} we obtain an example of a flow disjoint from simple flows, which is QS and its time-2 map is simple.
\item An easy linear time change yields now a flow $\cT$ which is QS, it is disjoint from simple flows, and such that $T_1$ is simple. In particular, $T_1$ cannot be embedded into a simple flow.
\end{enumerate}

\appendix

\section{Uniform distribution}\label{se:ud}
In this section we include some results on uniform distribution.
\subsection{Necessary tools}
All the information in this section is taken from~\cite{MR0419394}, unless stated otherwise. We will often use it freely, without precise reference.

\begin{df}
Let $X$ be a compact Hausdorff space with a (regular) probability Borel measure $\mu$.  A a sequence $(x_n)_{n\in\N}\subset X$ is said to be \emph{uniformly distributed} in $(X,\mu)$ if for every continuous function $f\colon X \to \R$ we have
$
\lim_{N\to \infty}\frac{1}{N}\sum_{n\leq N}f(x_n)=\int_X f\ d\mu.
$
\end{df}
\begin{df}
We say that $(x_n)_{n\in\N}\subset \R^s$ is \emph{uniformly distributed mod $1$} if $(\{x_n\})_{n\in\N}$ is uniformly distributed in $[0,1]^s$ with respect to the Lebesgue measure.
\end{df}

\begin{df}
Let $(x_n)_{n\in\N}\subset\R^s$. \emph{Discrepancy} is classically defined in the following two ways: for $N\in \N$ we let
\begin{align*}
&D_N^\ast=D_N^\ast(x_1,\dots, x_N):=\sup_{J^\ast}\left|\frac{1}{N}\sum_{n\leq N}\raz_{J^\ast}(\{x_n\})-\lambda_{\R^s}(J^\ast) \right|,\\
&D_N=D_N(x_1,\dots, x_N):=\sup_{J}\left|\frac{1}{N}\sum_{n\leq N}\raz_{J}(\{x_n\})-\lambda_{\R^s}(J) \right|.
\end{align*}
where $J^\ast$ runs over all subsets of $[0,1)^s$ of the form $[0,\beta_1)\times\dots\times [0,\beta_s)$ and $J$ runs over subsets of $[0,1)^s$ of the form $[\alpha_1,\beta_1)\times\dots\times [\alpha_s,\beta_s)$. 
\end{df}
\begin{pr}\label{dyskrepancja}
Let $(x_n)_{n\in\N}\subset \R^s$.
The following are equivalent:
\begin{itemize}
\item
the sequence $(x_n)_{n\in\N}$ is uniformly distributed mod $1$,
\item
$D_N\to 0$ as $N\to \infty$,
\item
$D_N^\ast\to 0$ as $N\to \infty$.
\end{itemize}
\end{pr}

\begin{df}
Let $\mathcal{S}=\{(x_{n,\sigma})\subset X\colon \sigma\in J\}$ be a family of sequences, indexed by $J$. $\mathcal{S}$ is said to be \emph{equi-$\mu$-uniformly distributed} in $X$ if for every $f\in C(X)$, we have
$
\lim_{N\to\infty}\sup_{\sigma\in J}\left|\sum_{n\leq N}f(x_{n,\sigma})-\int f\ d\mu\right|=0.
$
\end{df}
\begin{pr}[\cite{MR0419394}]\label{pr6.2}
Let $X$ be a compact Hausdorff uniform space. Suppose that $\{P_\sigma\colon \sigma \in J\}$ is a family of measure-preserving transformations on $X$ (with respect to given $\mu\in M_1(X)$) that is equicontinuous at every point $x\in X$, and suppose that $(x_n)_{n\in\N}\subset X$ is $\mu$-uniformly distributed. Then the family $\{(P_\sigma x_n)\subset X\colon \sigma \in J\}$ is equi-$\mu$-uniformly distributed.
\end{pr}

\begin{tw}[\cite{MR0277674}, see also~\cite{MR0419394}]\label{hlawka}
Let $(x_n)_{n\in\N}\subset \R^s$ be uniformly distributed mod $1$ and let $f\colon [0,1]^s \to \R$ be continuous. Let
$$
D^\ast_N(f,(x_n)_{1\leq n\leq N}):=\left|\frac{1}{N}\sum_{i=1}^{N}f(x_i)-\int_{[0,1]^s} f\ d\lambda_{\R^s} \right|.
$$
Then
$$
D^\ast_N(f,(x_n)_{1\leq n\leq N})\leq (1+2^{2s-1})M\left(\left[\left(D_N^*(f,(x_n)_{1\leq n\leq N})\right)^{-1/k} \right]^{-1} \right),
$$
where $M$ is the modulus of continuity of $f$.
\end{tw}
\begin{wn}\label{hlawka1}
Let $(x_n)_{n\in\N}\subset \R^s$  be uniformly distributed mod 1 and let $\cF$ be a family of equi-continuous functions on $[0,1]^s$. Then
$$
\sup_{f\in\cF}D^\ast_N(f,(x_n)_{n\in\N})\to 0 \text{ as }N\to\infty.
$$
\end{wn}

\begin{tw}[von Neumann, Oxtoby, Ulam, see~\cite{MR0005803}]\label{oxtoby}
Let $\mu_1,\mu_2$ be non-atomic probability Borel measures on $[0,1]^s$, with full support and such that $\mu_i(\partial([0,1]^s))=0$, $i=1,2$. Then there exists a homeomorphism $h\colon [0,1]^s \to [0,1]^s$ such that $\mu_1 \circ h^{-1}=\mu_2$. Moreover, $h|_{\partial([0,1]^s)}=id$.
\end{tw}

\subsection{Technical results}
Let $G=\R\ltimes_\varphi \text{SU}(2)$. By $\lambda_G$, $\lambda_{SU(2)}$ we will denote the Haar measures on $G$ and $SU(2)$ respectively, such that $\lambda_G((0,1)\times SU(2))=1$ and $\lambda_{SU(2)}(SU(2))=1$. Given a measurable set $S\subset G$, by $\lambda_S$ we denote the conditional probability measure on $S$. Finally, given a finite set $\widehat{S}\subset G$, by $\lambda_{\widehat{S}}$ we denote the normalized counting measure, i.e. $\lambda_{\widehat{S}}(A)=\frac{\# (\widehat{S}\cap A)}{\# \widehat{S}}$. Given $n\in\N$, let $\vep_n>0$. 
\begin{pr}\label{techniczny}
There exists a dense family $\cA$ of subsets of $G=\R\ltimes_\varphi \text{SU}(2)$, and for any $n\in\N$ there exists a finite subset $\widehat{S}_n\subset S_n$ such that the following hold:
\begin{enumerate}[(i)]
\item
For all $A,B\in\cA$ and all $a,b\in G$ we have
$$
\lambda_{{S}_n}(A^{-1}a\cap B^{-1}b)=\lambda_{\widehat{S}_n}(A^{-1}a\cap B^{-1}b)\pm\vep_n.
$$
\item
For all $A,B\in \cA$ and all $a,b\in G$ we have
$$
\int_{S_n\times S_n}f d\ \lambda_{{S}_n}\otimes \lambda_{{S}_n}=\int_{S_n\times S_n}f d\ \lambda_{\widehat{S}_n}\otimes \lambda_{\widehat{S}_n}\pm \vep_n
$$
for $f\colon G\times G\to \R$ given by $f(v,w)=\lambda_{F_n}(Aav\cap Bbw)$.
\end{enumerate}
\end{pr}
Before we begin the proof, we prepare tools and provide some necessary lemmas. First we introduce the notation.  Let $\cC_d$ stand for the family of cubes in $[0,1]^d$. By $\bm{d}\colon \R\times\R\to\R$ we denote the maximum metric in $\R^d$ and by $\bm{d}(x,C)$ we  denote the distance of the point $x$ from the set $C$. Given $C\in \cC_d$ and $\vep>0$ let
$$
C_\vep:=\{x\not\in C\colon \bm{d}(x,C)<\vep \}\text{ and } f_{C,\vep}:=\max\left(1-\frac{\bm{d}(x,C)}{\vep},0\right).
$$
Notice that for each $\vep>0$
\begin{equation}\label{19.3a}
\{f_{C,\vep}\colon C\in\cC_d\}\text{ is an equi-continuous family of functions}
\end{equation}
and
\begin{equation}\label{19.3b}
\lambda(C_\vep)\leq 2d\vep \text{ for all }C\in\cC_d\text{ and }\vep>0.
\end{equation}

\begin{lm}\label{lemma1.2}
Let $(X,\mu)$ be a compact metric space with a probability Borel measure and let $(y_n)_{n\in\N}\subset X$ be uniformly distributed with respect to $\mu$. Let $\cF$ be a family of measure-preserving equi-continuous homeomorphisms from $X$ to $[0,1]^d$. Then 
$$
\sup_{\Psi\in \cF}D_N^\ast\left((\Psi(y_n))_{1\leq n\leq N}\right)\to 0.
$$
\end{lm}
\begin{proof}
Let $\vep>0$ and choose $0<\delta<\vep$ small enough, so that for every cube $C_\delta$ whose sides have length at most $\delta$ we have $\Psi^{-1}(C_\delta)\subset C'_\vep$, where $C'_\vep$ is some cube whose sides have length at most $\vep$. Let $K:=d\cdot 2^{d-1}\cdot\frac{1}{\delta}$. 

We have
\begin{multline*}
\left|\frac{1}{N}\sum_{n\leq N}\raz_C\circ \Psi(x_n)-\lambda(C) \right|\leq\underbrace{\left|\frac{1}{N}\sum_{n\leq N}\left(\raz_C\circ \Psi(x_n)-f_{C,\delta}\circ \Psi(x_n)\right) \right|}_{I_1}\\
+\underbrace{\left|\frac{1}{N}\sum_{n\leq N}f_{C,\delta}\circ \Psi(x_n)-\int f_{C,\delta}\ d\lambda \right|}_{I_2}
+\underbrace{\left|\int f_{C,\delta}\ d\lambda-\lambda(C) \right|}_{I_3}.
\end{multline*}
Notice that
$
I_1\leq \frac{1}{N}\sum_{n\leq N}\raz_{C_\delta}\circ \Psi(x_n)=\frac{1}{N}\sum_{n\leq N}\raz_{\Psi^{-1}(C_\delta)}(x_n).
$
By the choice of $K$ the set $C_\delta$ is included in a union of (at most) $K$ cubes whose sides are equal at most $\delta$. Therefore, by the choice of $\delta$, $\Psi^{-1}(C_\delta)$ is included in a union of (at most) $K$ cubes whose sides are equal at most $\vep$. It follows immediately that there exists $N_0$ such that for $N\geq N_0$ we have $I_1\leq \vep$. 

By Theorem~\ref{hlawka} we have $I_2\to 0$. Moreover, additionally taking~\eqref{19.3a} into account, we see that the speed of convergence can be estimated using only $D_N^\ast ((x_n)_{1\leq n\leq N})$ (and $\vep$).

Finally, $I_3\leq \lambda(C_\delta)<2d\delta\leq 2d\vep$ and the claim follows.
\end{proof}
\begin{lm}\label{lemma1.3}
Let $(y_n)_{n\in \N}\subset [0,1]^d$ be uniformly distributed. Let $\cF$ be a family of equicontinuous homeomorphisms of $[0,1]^d$ preserving the Lebesgue measure. Then
\begin{equation*}
\sup_{\Psi_1,\Psi_2 \in \cF} \sup_{C,D,E\subset \cC_d} \left|\frac{1}{N}\sum_{n\leq N}\raz_{\Psi_1(C)\cap \Psi_2(D)\cap E}(y_n)-\lambda(\Psi_1(C)\cap \Psi_2(D)\cap E) \right| \to 0.
\end{equation*}
\end{lm}
\begin{proof}
Choose $\vep>0$ and let
\begin{align*}
I_1&:=\lambda(\Psi_1(C)\cap \Psi_2(D)\cap E)-\int f_{C,\vep}\circ \Psi_1^{-1}\cdot f_{D,\vep}\circ \Psi_2^{-1}\cdot f_{E,\vep} \ d\lambda,\\
I_2&:=\frac{1}{N}\sum_{n\leq N}\raz_{\Psi_1(C)\cap \Psi_2(D)\cap E}(x_n)-\frac{1}{N}\sum_{n\leq N}(f_{C,\vep}\circ \Psi_1^{-1}\cdot f_{D,\vep}\circ \Psi_2^{-1}\cdot f_{E,\vep})(x_n),\\
I_3&:= \frac{1}{N}\sum_{n\leq N}(f_{C,\vep}\circ \Psi_1^{-1}\cdot f_{D,\vep}\circ \Psi_2^{-1}\cdot f_{E,\vep})(x_n)\\
& - \int f_{C,\vep}\circ \Psi_1^{-1}\cdot f_{D,\vep}\circ \Psi_2^{-1}\cdot f_{E,\vep} \ d\lambda.
\end{align*}
It suffices to show that $I_1,I_2,I_3\to 0$ and the convergence is uniform with respect to $C,D,E\in \cC_d$ and $\Psi_1,\Psi_2\in \cF$. 

We have
\begingroup
\allowdisplaybreaks
\begin{align*}
|I_1|\leq& \left|\int \raz_C\circ \Psi_1^{-1}\cdot \raz_D\circ \Psi_2^{-1}\cdot \raz_E - \raz_C\circ \Psi_1^{-1}\cdot \raz_D\circ \Psi_2^{-1}\cdot f_{E,\vep}\ d\lambda \right|\\
&+\left|\int  \raz_C\circ \Psi_1^{-1}\cdot \raz_D\circ \Psi_2^{-1}\cdot f_{E,\vep} -  \raz_C\circ \Psi_1^{-1}\cdot f_{D,\vep}\circ \Psi_2^{-1}\cdot f_{E,\vep} \right|\\
&+\left|\int  \raz_C\circ \Psi_1^{-1}\cdot f_{D,\vep}\circ \Psi_2^{-1}\cdot f_{E,\vep} -  f_{C,\vep}\circ \Psi_1^{-1}\cdot f_{D,\vep}\circ \Psi_2^{-1}\cdot f_{E,\vep} \right|\\
\leq& \int \left| \raz_E-f_{E,\vep} \right| \ d\lambda+\int \left|\raz_D\circ \Psi_2^{-1}-f_{D,\vep}\circ \Psi_2^{-1} \right| \ d\lambda\\
&+\int \left|\raz_C\circ \Psi_1^{-1}-f_{C,\vep}\circ \Psi_1^{-1} \right| \ d\lambda\\
=&\int |\raz_E-f_{E,\vep}|\ d\lambda+\int |\raz_D-f_{D,\vep}|\ d\lambda+\int |\raz_C-f_{C,\vep}|\ d\lambda\\
\leq& \lambda(E_{\vep})+\lambda(D_\vep)+\lambda(C_\vep)\leq 6d\vep,
\end{align*}
\endgroup
where the last inequality follows by~\eqref{19.3b}.

We will now estimate $I_2$:
\begin{align*}
|I_2|\leq& \left|\frac{1}{N}\sum_{n\leq N}\left[(\raz_{\Psi_1(C)}\cdot \raz_{\Psi_2(D)}\cdot\raz_E) (x_n)-(\raz_{\Psi_1(C)}\cdot \raz_{\Psi_2(D)}\cdot f_{E,\vep})\right] (x_n) \right|\\
&+\left|\frac{1}{N}\sum_{n\leq N}\left[(\raz_{\Psi_1(C)}\cdot \raz_{\Psi_2(D)}\cdot f_{E,\vep}) (x_n)-(\raz_{\Psi_1(C)}\cdot (f_{D,\vep}\circ \Psi_2^{-1})\cdot f_{E,\vep}) (x_n) \right]\right|\\
&+\left|\frac{1}{N}\sum_{n\leq N}\left[(\raz_{\Psi_1(C)}\cdot (f_{D,\vep}\circ \Psi_2^{-1})\cdot f_{E,\vep}) (x_n)\right.\right.\\ 
&\left.\left. -((f_{C,\vep}\circ \Psi_1^{-1})\cdot (f_{D,\vep}\circ \Psi_2^{-1})\cdot f_{E,\vep}) (x_n) \right]\right|\\
\leq & \frac{1}{N}\sum_{n\leq N}\raz_{E_\vep}(x_n)+\frac{1}{N}\sum_{n\leq N}\raz_{\Psi_2(D_\vep)}(x_n)+\frac{1}{N}\sum_{n\leq N}\raz_{\Psi_1(C_\vep)}(x_n)\\
=&\frac{1}{N}\sum_{n\leq N}\raz_{E_\vep}(x_n)+\frac{1}{N}\sum_{n\leq N}\raz_{D_\vep}(\Psi_2^{-1}(x_n))+\frac{1}{N}\sum_{n\leq N}\raz_{C_\vep}(\Psi_1^{-1}(x_n))
\end{align*}
Since each of the sets $E_\vep,D_\vep,C_\vep$ is a difference of two cubes, the above expression converges to $\lambda(E_\vep)+\lambda(D_\vep)+\lambda(C_\vep)\leq 6d\vep$.
Moreover, by Lemma~\ref{lemma1.2}, the convergence is uniform with respect to $C,D,E\in\cC_d$ and $\Psi_1,\Psi_2\in\cF$.

$I_3$ also converges to zero and again, by Theorem~\ref{hlawka},  the convergence is uniform with respect to $C,D,E\in\cC_d$ and $\Psi_1,\Psi_2\in\cF$.
\end{proof}

Let $d$ be a right-invariant metric on $G$. Let
$\cU=\{U_i\colon 1\leq i\leq M\}$ be a finite cover of $SU(2)$ by open sets, such that $\lambda_{SU(2)}(\partial U_i)=0$ for $1\leq i\leq M$.
Then
\begin{equation}\label{pokrycie}
\cV:=\bigcup_{l\in\Z}\bigcup_{i=1}^{M}(l-1,l+2)\times U_i
\end{equation}
is an open cover of $G$, such that $\lambda_G(\partial ((l-1,l+2)\times U_i))=0$.
\begin{lm}\label{liczbaleb}
The cover~\eqref{pokrycie} has a positive Lebesgue number.
\end{lm}
\begin{proof}
Suppose that there is no $\delta>0$ such that for each $(t,\bm M)\in G$ we have $B((t,\bm M),\delta)\subset (l-1,l+2)\times U_i$ for some $l\in\Z$ and $1\leq i\leq M$. We claim that
\begin{equation}\label{kule}
B(x,\vep)\subset (l-1,l+2)\times U_i \iff B(x(-l,I),\vep)\subset (-1,2)\times U_i.
\end{equation}
Indeed, assume that $B((t, \bm M),\vep)\subset (l-1,l+2)\times U_i$ and take $y\in B((t,\bm M)(-l,I),\vep)$. We have $d(y(l,I),(t,\bm M))=d(y,(t,\bm M)(-l,I))<\vep$, whence
$$
y\in \left[(l-1,l+2)\times U_i \right](-l,I)=(-1,2)\times U_i
$$
and~\eqref{kule} follows. According to our assumption there exists $x_n=(t_n,\bm M_n)$ such that 
\begin{equation}\label{kontr}
\mbox{none of the balls $B\left(x_n,\frac{1}{n}\right)$ is a subset of an element of our cover.}
\end{equation}
Without loss of generality, in view of~\eqref{kule}, we may assume that
$x_n\in (-1,2)\times U_i \text{ for some }i$, and also that $x_n\to x$ for some $x=(t,\bm M)\in (l-1,l+2)\times U_j$, where $l\in \Z$, $1\leq j\leq M$. Then for $N$ large enough we have $B(x,\frac{1}{N})\subset (l-1,l+2)\times U_j$. Let $n\geq 2N$ be large enough, so that $x_n\in B(x,\frac{1}{2N})$. Then $B(x_n,\frac{1}{n})\subset B(x_n,\frac{1}{2N})\subset B\left(x,\frac{1}{N}\right)\subset (l-1,l+2)\times U_j$, which yields a contradiction with~\eqref{kontr}. This ends the proof.
\end{proof}
For $1\leq i\leq M$ let $\psi_i\colon \overline{U}_i\to [0,1]^3$ be a homeomorphism which carries $\lambda_{\overline{U}_i}$ to the Lebesgue measure $\lambda_{[0,1]^3}$ on $[0,1]^3$.\footnote{Such homeomorphism exists in view of Theorem~\ref{oxtoby}.} For $1\leq i\leq M$ and $l\in\Z$ let
$\psi_{l,i}\colon [l-1,l+2]\times \overline{U}_i \to [-1,2]\times [0,1]^3$
be given by 
$$
\psi_{l,i}(x,\bm M)=\left(x-l,\psi_i(\bm M) \right)
$$
and let  $\psi\colon [-1,2]\times [0,1]^3\to [0,1]^4$ be given by
$$
\psi(x_1,x_2,x_3,x_4)=\left((x_1+1)/3,x_2,x_3,x_4\right).
$$
For $a\in G$ let $g_a,g\colon G\to G$ be given by 
$$
g_a(b)=ba\text{ and }g(b)=b^{-1}\text{ for }b\in G.
$$ 
\begin{lm}\label{lemma1.3}
The family of functions
$$
\left\{ g_a\circ g\circ \psi_{l,i}^{-1}\circ \psi^{-1}\colon a\in G, l\in\Z, 1\leq i\leq M\right\}
$$
is uniformly bi-equicontinuous.
\end{lm}
\begin{proof}
It suffices to show that the following families of functions are uniformly bi-equicontinuous:
\begin{enumerate}[(i)]
\item
$\cF_1=\left\{\psi_{l,i}^{-1}\colon l\in\Z, 1\leq i\leq M \right\}$,
\item
$\cF_2=\left\{g|_{[l-1,l+2]\times SU(2)}\colon l\in\Z \right\}$,
\item
$\cF_3=\left\{g_a\colon a\in G \right\}$.
\end{enumerate}
Notice that $\cF_1$ is a finite family of homeomorphisms of compact spaces and elements of $\cF_3$ are isometries, so both $\cF_1, \cF_3$ are uniformly bi-equi continuous. We consider now $\cF_2$. For $l\in\Z$ let $g_l=g|_{[l-1,l+2]\times SU(2)}$. For 
$(t,M), (s,N)\in [l-1,l+2]\times SU(2)$
we have
$$
d((t-l,M),(s-l,N))=d((t,M)(-l,I),(s,N)(-l,I))=d((t,M),(s,N)).
$$
and
\begin{align*}
d(g_l(t,M),g_l(s,N))&=d((-t,\varphi_{-t}(M^{-1})),(-s,\varphi_{-s}(N^{-1})))\\
&=d((-t,\varphi_{-t}(M^{-1}))(l,I),(-s,\varphi_{-s}(N^{-1}))(l,I))\\
&=d((l-t,\varphi_{-t}(M^{-1})),(l-s,\varphi_{-s}(N^{-1})))\\
&=d((l-t,\varphi_{l-t}(M^{-1})),(l-s,\varphi_{l-s}(N^{-1})))\\
&=d(g_0(t-l,M),g_0(s-l,N)).
\end{align*}
Moreover, $g_l^{-1}=g_{-1-l}$. This ends the proof.
\end{proof}
Let $\{\rho_i\}_{1\leq i\leq M}$ and $\{\eta_l\}_{l\in\Z}$ be smooth partitions of unity, subordinate to cover $\cU$ and $\cup_{l\in\Z}(l-1,l+2)=\R$ respectively.
Then clearly $\{\eta_l\otimes \rho_i\}_{1\leq i\leq M,l\in\Z}$ is a smooth partition of unity subordinate to cover $\cV$.

Let $\widetilde{\mathcal{A}}$ be a family of ``cubes'' in $ [0,1]^4$ sufficiently small, so that for $a_1,a_2,a_3\in G$, $l_1,l_2,l_3\in\Z$, $1\leq i_1,i_2,i_3 \leq M$, $C_1,C_2,D_1,D_2,D_3\in\widetilde{\cA}$ whenever the sets
$$
\bigcup_{j=1}^{2} g_{a_j}\circ g\circ \psi_{l_j,i_j}^{-1}\circ \psi^{-1}(C_j)\text{ and }\bigcup_{j=1}^{3}g_{a_j}\circ \psi_{l_j,i_j}^{-1}\circ \psi^{-1}(D_j)\text{ are connected}
$$
then
$$
\bigcup_{j=1}^{2} g_{a_j}\circ g\circ \psi_{l_j,i_j}^{-1}\circ \psi^{-1}(C_j)\subset (l_C-1,l_C+2)\times U_{i_C}
$$
and
\begin{equation}\label{326b}
\bigcup_{j=1}^{3} g_{a_j}\circ \psi_{l_j,i_j}^{-1}\circ \psi^{-1}(D_j)\subset (l_D-1,l_D+2)\times U_{i_D}
\end{equation}
for some $l_C,_D\in\Z, 1\leq i_C,i_D\leq M$. This can be done in view of Lemma~\ref{liczbaleb} and Lemma~\ref{lemma1.3}. 
\begin{uw}
Without loss of generality, we may assume that in fact an even stronger condition than~\eqref{326b} holds: 
\begin{multline}\label{326b1}
\left(\bigcup_{j=1}^{3} g_{a_j}\circ \psi_{l_j,i_j}^{-1}\circ \psi^{-1}(D_j)\right)_{\delta_0}\cup \left(\bigcup_{j=1}^{3} g_{a_j}\circ \psi_{l_j,i_j}^{-1}\circ \psi^{-1}(D_j)\right)\\
\subset (l_D-1,l_D+2)\times U_{i_D},
\end{multline}
where $\delta_0>0$ is a sufficiently small number, independent of the choice of $D_j$, $a_j$, $1\leq j\leq 3$. This can be done for example by adding one more step before fixing the open cover of $SU(2)$. Namely, given $\{U_i\}_{1\leq i\leq M}$, we need to find $\delta_0>0$ such that $\{(U_i)_{\delta_0}\cup U_i\}_{1\leq i\leq M}$ still consists of sets homeomorphic to $(0,1)^4$ and work with both covers simultaneously.
\end{uw}

Let
$$
\cA=\left\{ \psi_{l,i}^{-1}\circ \psi^{-1}(C)\colon C\in\widetilde{\cA},l\in\Z,1\leq i\leq M\right\}.
$$
Let $(x_n)_{n\in\N}\subset [0,1]\times SU(2)$ be uniformly distributed in $[0,1]\times SU(2)$.\footnote{The existence of such a sequence can be shown e.g. in the following way. Choose a finite open cover of $\T\times SU(2)$ consisting of continuity sets (i.e. sets whose boundaries are of zero measure) which are simply connected. It yields a partition of $\T\times SU(2)$ into a finite number of open sets, up to a set of measure zero. Using Theorem~\ref{oxtoby}, for each set of this partition, we can find a homeomorphism to $(0,1)^4$ carrying the conditional Haar measure to the Lebesgue measure. Any sequence which is uniformly distributed in $[0,1)^4$ yields sequences uniformly distributed in the closure of the elements of the partition. The desired sequence can be now constructed by taking elements from these sequences with frequencies approximating the measures of the sets from the cover.} For $l\in\Z$ we denote by $(x_{n,l})_{n\in \N}$ the sequence in $[l,l+1]\times SU(2)$ given by $x_{n,l}=x_n+(l,0)$ (the addition is understood coordinate-wise). 
\begin{proof}[Proof of Proposition~\ref{techniczny} (i)]
Fix $n\in\N$. The set $\widehat{S}_n$ we are looking for will consists of elements of sequences $(x_{n,l})_{n\in\N}$, more precisely, we will have
$$
\widehat{S}_n=\left\{x_{n,l}\colon 1\leq n\leq n_0, l\in\Z \right\}\cap S_n,
$$
where $n_0\in\N$ will be some sufficiently large number (depending on $\vep_n$). Take $A,B\in\cA$ and all $a,b\in G$. Then
$$
A=\psi^{-1}_{l_A,i_A}\circ \psi^{-1}(C_A),\ B=\psi^{-1}_{l_B,i_B}\circ \psi^{-1}(C_B)
$$
for some $C_A,C_B\in\widetilde{\cA}$ and $l_A,l_B\in \Z$, $1\leq i_A,i_B\leq M$.

If $A^{-1}a\cap B^{-1}b\cap S_n=\varnothing$ then clearly
$$
\lambda_{S_n}(A^{-1}a\cap B^{-1}b)=\lambda_{\widehat{S}_n}(A^{-1}a\cap B^{-1}b)=0,
$$
no matter which finite subset $\widehat{S}_n\subset S_n$ we choose. Suppose now that $A^{-1}a\cap B^{-1}b\cap S_n\neq \varnothing$ and consider the following cases:
\begin{multicols}{2}
\begin{enumerate}[(a)]
\item $A^{-1}a\cup B^{-1}b\subset S_n$,
\item $A^{-1}a\cup B^{-1}b\not\subset S_n$.
\end{enumerate}
\end{multicols}
We will show how to proceed in case (b) (case (a) can be treated in a very similar way). Recall that $S_n=I^{\R}[(2n-1)\widetilde{a}_{n-1}]\times SU(2)$, so in view of the definition of $\cA$ one of the following holds:
\begin{enumerate}
\item[(b1)]
$A^{-1}a\cup B^{-1}b\subset (l-2,l+1)\times U_i$,
\item[(b2)]
$A^{-1}a\cup B^{-1}b\subset (l-1,l+2)\times U_i$,
\item[(b3)]
$A^{-1}a\cup B^{-1}b\subset (-l-2,-l+1)\times U_i$,
\item[(b4)]
$A^{-1}a\cup B^{-1}b\subset (-l-1,-l+2)\times U_i$,
\end{enumerate}
where $1\leq i\leq M$ and $l=-(2n-1)\widetilde{a}_{n-1}$. We will cover case (b1) (the other cases can be treated in the same way). We have
\begingroup
\allowdisplaybreaks
\begin{align}
\begin{split}\label{mar1}
\lambda_{S_n}(A^{-1}&a\cap B^{-1}b)=\frac{\lambda((l-1,l+2)\times SU(2))}{\lambda(S_n)}\\
&\cdot\lambda_{(l-1,l+2)\times SU(2)}(A^{-1}a\cap B^{-1}b\cap ((l,l+2)\times SU(2)))\\
=&\frac{3}{2|l|}\cdot\lambda_{(l-1,l+2)\times SU(2)}(A^{-1}a\cap B^{-1}b\cap ((l,l+2)\times SU(2)))\\
=&\frac{3}{2|l|}\cdot\lambda\left(\psi\circ \psi_{l,i}\circ g_a\circ g\circ \psi_{l_A,i_A}^{-1}\circ \psi^{-1}(C_A)\right.\\
&\left.\cap \psi\circ \psi_{l,i}\circ g_b\circ g\circ \psi_{l_B,i_B}^{-1}\circ \psi^{-1}(C_B)\cap \left(\left(1/3,1\right)\times (0,1)^3\right) \right)\\
=&\frac{3}{2|l|}\cdot\lambda (\Psi_A(C_A)\cap \Psi_B(C_B)\cap C),
\end{split}
\end{align}
\endgroup
where
\begin{align*}
&\Psi_A=\psi\circ \psi_{l,i}\circ g_a\circ g\circ \psi_{l_A,i_A}^{-1}\circ \psi^{-1},\\
&\Psi_B=\psi\circ \psi_{l,i}\circ g_b\circ g\circ \psi_{l_B,i_B}^{-1}\circ \psi^{-1},\\
&C=\left(1/3,1\right)\times (0,1)^3.
\end{align*}
Define $(y_n)_{n\in\N}\subset [0,1]^4$, by
\begin{equation*}
y_n=\begin{cases}
\psi\circ \psi_{l,i}(x_{(n+2)/3,l-1}),&\text{ if }n\equiv 0 \mod 3,\\
\psi\circ \psi_{l,i}(x_{(n+1)/3,l}),&\text{ if }n\equiv 1 \mod 3,\\
\psi\circ \psi_{l,i}(x_{n/3,l+1}),&\text{ if }n\equiv 2 \mod 3.
\end{cases}
\end{equation*}
Notice that $(y_n)_{n\in\N}$ is uniformly distributed in $[0,1]^4$ and the discrepancy of this sequence does not depend on $l\in\Z$. We have
\begingroup
\allowdisplaybreaks
\begin{align}
\begin{split}\label{mar2}
\frac{1}{N}\sum_{n\leq N}&\raz_{A^{-1}a\cap B^{-1}b\cap ((l,l+2)\times SU(2))}(x_{n,l-1})\\
&+\frac{1}{N}\sum_{n\leq N}\raz_{A^{-1}a\cap B^{-1}b\cap ((l,l+2)\times SU(2))}(x_{n,l})\\
&+\frac{1}{N}\sum_{n\leq N}\raz_{A^{-1}a\cap B^{-1}b\cap ((l,l+2)\times SU(2))}(x_{n,l+1})\\
=&3\cdot\frac{1}{3N}\sum_{n=1}^{3N}\raz_{\Psi_A(C_A)\cap \Psi_B(C_B)\cap C}(y_n)\\
\to& 3\lambda(\Psi_A(C_A)\cap \Psi_B(C_B)\cap C).
\end{split}
\end{align}
\endgroup
In view of Lemma~\ref{lemma1.3} the above convergence is uniform with respect to $A,B$ and $a,b$. It follows from~\eqref{mar1} and~\eqref{mar2} that $\widehat{S}_n$ defined in the following way:
$$
\widehat{S}_n=S_n\cap\left(\bigcup_{l\in Z}\bigcup_{n=1}^{n_0}x_{n,l}\right).
$$
satisfies (i) provided that $n_0$ is large enough.
\end{proof}
\begin{proof}[Proof of Proposition~\ref{techniczny} (ii)]
We claim that the family
\begin{multline}\label{263c}
\{f\colon G\times G\to \R\colon f(v,w)=\lambda_{F_n}(Aav\cap Bbw)\}\\
 \text{ is uniformly equi-continuous.}
\end{multline}
Fix $\vep>0$. Let $\delta\in(0,\delta_0)$ be sufficiently small, so that for $C\in\cA$ and $c\in G$ such that $\lambda(Cc)=\lambda(C)>\vep/2$ we can find a ball of radius $\delta$ being a subset of $Cc$. Such choice is possible in view of Lemma~\ref{lemma1.3}.
Take $A,B\in\cA$, $a,b\in G$, $n\in\N$ and let $f(v,w)=\lambda_{F_n}(Aav\cap Bbw)$. Take $v_1,v_2,w\in G$ with $d(v_1,v_2)<\delta$. We have
\begin{equation}\label{263a}
|f(v_1,w)-f(v_2,w)|\leq \lambda((Aav_1\triangle Aav_2)\cap Bbw).
\end{equation}
We will consider the following cases:
\begin{enumerate}
\item[(i)]
$(Aav_1\cup Aav_2)\cap Bbw =\varnothing$,
\item[(iia)]
$(Aav_1\cup Aav_2)\cap Bbw \neq\varnothing$, $\lambda(A)\leq\vep/2$,
\item[(iib)]
$(Aav_1\cup Aav_2)\cap Bbw \neq\varnothing$, $\lambda(A)>\vep/2$.
\end{enumerate}
Using~\eqref{263a} we obtain in case (i) $|f(v_1,w)-f(v_2,w)|=0$ and in case (iia) $|f(v_1,w)-f(v_2,w)|\leq \vep$. In case (iib), by the choice of $\delta$, it is clear that $Aav_1\cap Aav_2\neq \varnothing$. Therefore (see the definition of $\widetilde{\cA}$) there exist $l\in \Z$ and $1\leq i\leq M$ such that
$$
(Aav_1\cup Aav_2\cup Bbw)_{\delta_0}\cup (Aav_1\cup Aav_2\cup Bbw) \subset (l-1,l+2)\times U_i.
$$
Now we can use the homeomorphism $\psi_{l,i}^{-1}\circ \psi$ to ``transport'' the above set to $[0,1]^4$. Using Lemma~\ref{lemma1.3} and adjusting $\delta$ if necessary we arrive at~\eqref{263c}.

Let $K=(2n-1)\widetilde{a}_{n-1}$ and for $k\in [-K,K-1]$ let $y_{m,k}=x_m+(k,I)$. Moreover, for $1\leq i\leq M$ let $(y_{m,k}^i)_{m\in\N}$ be a subsequence of elements of $(y_{m,k})_{m\in\N}$ which are in $[k,k+1]\times U_i$ and let $N_i=\# \{ x_{m,k}\in U_i\colon 1\leq m\leq N\}$ (notice that this quantity is independent of $k$). Then for a fixed $w$ we have
\begin{align}\label{whole}
\begin{split}
\frac{1}{2KN}&\sum_{k=-K}^{K-1}\sum_{m=1}^{N}f(y_{m,k},w)\\
&=\frac{1}{2KN}\sum_{k}\sum_{m}\sum_{l}\sum_{i}\eta_l\times\rho_i(y_{m,k})f(y_{m,k},w)\\
&=\frac{1}{2KN}\sum_{i}\sum_{l}\sum_{k=l-1}^{l+1}\sum_{m}\eta_l\times \rho_i(y_{m,k})f(y_{m,k},w)\\
&=\frac{1}{2KN}\sum_{i}\frac{N_i}{N}\sum_{l}\frac{1}{N_i}\sum_{k=l-1}^{l+1}\sum_{m=1}^{N_i}\underbrace{\eta_l\otimes \rho_i(y_{m,k}^i)f(y_{m,k},w)}_{f_{l,i,m,k}}.
\end{split}
\end{align}
We have
$$
\frac{1}{3N_i}\sum_{k=l-1}^{l+1}\sum_{m=1}^{N_i}f_{l,i,m,k}(y_{m,k}^i)\to \int f_{l,i,m,k}\ d\lambda_{(l-1,l+2)\times SU(2)}
$$
and
$$
\frac{N_i}{N}\to \lambda_{SU(2)}(U_i),
$$
and the above convergences are uniform in view of Theorem~\ref{hlawka} and Lemma~\ref{lemma1.3}. Hence the expression in \eqref{whole} converges to
\begin{multline*}
\frac{1}{2K}\sum_{i=1}^{M}\sum_l\lambda_{SU(2)}(U_i)\cdot 3\cdot \int f_{l,i,m,k}\ d\lambda_{(l-1,l+2)\times U_i}\\
=\frac{1}{2K}\int \sum_l\sum_{i=1}^{M}3\lambda_{SU(2)}(U_i) \eta_l\otimes \rho_i\cdot f(\cdot, w)\ d\lambda_{(l-1,l+2)\times U_i}=\int f(\cdot, w) \ d\lambda_{S_n}.
\end{multline*}
To end the proof it suffices to use Fubini's theorem and take
$\widehat{S}_n:=S_n\cap \left(\bigcup_{k\in\Z}\bigcup_{n=1}^{n_0}x_{n,k} \right)$
for $n_0$ large enough. Notice that this is consistent with the final choice of $\widehat{S}_n$ in the proof of Proposition~\ref{techniczny} (i).
\end{proof}
\begin{uw}\label{tech}
Without loss of generality, in Proposition~\ref{techniczny} we may assume that $F_n\in\cA$ for $n\in\N$.
\end{uw}
\begin{uw}
The assertion of Proposition~\ref{techniczny} remains true if instead of sets $A,B\in\cA$ we consider sets $A',B'$ which are finite sums of translations of elements from $\cA$, provided that the number of elements in the sum is bounded for each $n$.
\end{uw}

\section{More proofs}\label{apB}
In this section we include other proofs of some of the result presented in the main part of the paper.
\subsection{Another proof of Proposition~\ref{pr:msj}}
We will need a classical result on group homomorphisms:
\begin{tw}[Mackey, see~\cite{MR776417}]\label{mackey}
Let $G$ be a locally compact, second countable group and let $J$ be a second countable group. Let $\pi\colon G\to J$ be a measurable group homomorphism. Then $\pi$ is continuous.
\end{tw}

\begin{proof}[Proof 2 of Proposition~\ref{pr:msj}]
Let $\cS$ be such that $T_1=S_1$. Then for any $t\in\R$ we have $S_t\in C(S_1)=C(T_1)=C(\cT)$ (the latter equality follows from Proposition~\ref{pr:simple}). By the assumption that $\cT$ has the MSJ property, it follows that there exists a function $a\colon \R\to \R$ such that $S_t=T_{a(t)}$. It is a group homomorphism. We claim that it is continuous. Indeed, $t\mapsto S_t$ is measurable, whence (by Theorem~\ref{mackey}) it is continuous.  Now, using the fact that $S_t=T_{a(t)}$ and the open mapping theorem we obtain that $t\mapsto a(t)$ is also continuous. Therefore we have $a(t)=at$ for some $a\in\R$. By taking $t=1$ we obtain that $a=1$, i.e. $T_t=S_t$ for all $t\in\R$.
\end{proof}

\subsection{Another proof Proposition~\ref{pr12a} and~\ref{pr12a1}}
In this section we will give a more direct proof (avoiding isometric extensions) of Proposition~\ref{pr12a} and~\ref{pr12a1} in case $\gr{\nu}=\nu$.
\subsubsection{Tools: separating sieves}
Let $T$ and $S$ be automorphisms of standard probability Borel spaces $\xbm$ and $\ycn$ respectively.  Suppose that $T$ is an ergodic extension of $S$. Let $\pi\colon X\to Y$ be the factoring map.
\begin{df}[\cite{MR0234443}]
Given $A\in\cB$, let 
$$
A\times_YA:=\{(x,x')\in A\times A\colon \pi(x)=\pi(x')\}.
$$
A sequence $A_1\supset A_2\supset \dots$ of sets in $\cB$ with $\mu(A_n)>0$ and $\mu(A_n)\to 0$, is called a \emph{separating sieve} over $Y$ if there exists a subset $X_0\subset X$ with $\mu(X_0)=1$ such that for every $x,x'\in X_0$ with $\pi(x)=\pi(x')$, the condition ``for every $n\in\N$ there exists $m\in \Z$ with $T^m x,T^m x'\in A_n$'' implies $x=x'$, i.e.
$$
\bigcap_{n=1}^{\infty}\left(\bigcup_{m\in\Z}(A_n\times_Y A_n) \right)\cap (X_0\times_Y X_0)\subset \{(x,x)\colon x\in X_0\}.
$$
We also say that the extension $\pi$ has a \emph{separating sieve}. 
\end{df}
\begin{uw}\label{zimm}
Zimmer~\cite{MR0409770} showed that the existence of a separating sieve is equivalent to distality of the considered extension.
\end{uw}

\begin{uw}\label{uw:zmniejszanie}
Notice that whenever $(A_n)_{n\in\N}$ forms a separating sieve for $\pi$, then also any subsequence $(A_{n_k})_{k\geq 1}$ is a separating sieve for $\pi$. Moreover, any decreasing sequence $(B_n)_{n\in\N}$ of sets of positive measure with $B_n\subset A_n$ forms a separating sieve for $\pi$.
\end{uw}

\begin{lm}\label{lm:sita}\footnote{An analogous result holds for other group actions -- the proof can be repeated word by word.}
Suppose that $T$ is an ergodic extension of $S$ with the factoring map $\pi$. The following conditions are equivalent:
\begin{itemize}
\item[(i)]
There exists a separating sieve for $\pi$.
\item[(ii)]
There exists $X_0\subset X$ with $\mu(X_0)=1$ and a sequence of sets $(B_n)_{n\in\N}$ with $\mu(B_n)>0$, $\mu(B_n)\to 0$ and such that for $x,x'\in X$ with $\pi(x)=\pi(x')$ whenever for every $n\in\N$ there exists $m_n\in\Z$ such that $T^{m_n}x$, $T^{m_n}x'\in B_n$ then $x=x'$.
\end{itemize}
Moreover, if additionally $\mu$ is non-atomic, the above conditions are equivalent to
\begin{itemize}
\item[(ii')]
There exists $X_0\subset X$ with $\mu(X_0)=1$ and a sequence of sets $(B_n)_{n\in\N}$, $\mu(B_n)>0$ and such that for $x,x'\in X$ with $\pi(x)=\pi(x')$ whenever for every $n\in\N$ there exists $m_n\in\Z$ such that $T^{m_n}x$, $T^{m_n}x'\in B_n$ then $x=x'$.
\end{itemize}
\end{lm}

\begin{proof}
Clearly (i) implies (ii) and (ii').

We will show now that (ii) implies (i). Let us recursively define a sequence of sets $(A_n)_{n\in\N}$:
$$
A_1=B_1,\ A_n=T^{k_n}B_n\cap A_{n-1},
$$
where $k_n\in\Z$ is such that $\mu(A_n)=\mu(T^{k_n}B_n\cap A_{n-1})>0$. Such $k_n$ exists by ergodicity of $T$. We claim that $(A_n)_{n\in\N}$ is a separating sieve for $\pi$. Clearly $\mu(A_n)\to 0$. Let $x,x'\in X_0$ be such that for each $n\in\N$ there exists $m_n$ such that $T^{m_n}x,T^{m_n}x'\in A_n$. By (ii) this implies that $T^{-k_n+m_n}x,T^{-k_n+m_n}x'\in B_n$ whence $x=x'$ and we have proved that (i) follows from (ii).

To see that (ii') implies (i), it suffices to adjust the sets $A_n$ from the proof of the previous part, so that $\mu(A_n)\to 0$.
\end{proof}

\subsubsection{Proof}\label{drugidowod2}
Let $G$ be a locally compact, second countable abelian group and let $H$ be its co-compact subgroup. Suppose $G$ acts ergodically on $\xbm$. Let $\cA\subset\cB$ be a factor of $G$ such that the subaction of $H$ restricted to $(X,\cA,\mu|_\cA)$ is ergodic. There are two possibilities:
\begin{itemize}
	\item[($A_G$)] the $H$-sub-action is ergodic with respect to $\mu$,
	\item[($B_G$)] the $H$-sub-action is not ergodic with respect to $\mu$.
\end{itemize}
We will be interested in $G\in\{\R,\Z\}$ and $H=l\Z$ for $l\in\N$. In order to give an alternative proof of Proposition~\ref{pr12a} and Proposition~\ref{pr12a1} in case $\gr{\nu}=\nu$ we take the following path:
$$
(A_\Z) \rightarrow \left[(B_\R) \text{ and } (B_\Z)\right] \rightarrow (A_\R).
$$
Notice that
\begin{itemize}
\item
case ($A_\Z$) corresponds to Proposition~\ref{pr12a1} in case $\widetilde{\gr{\nu}}=\widetilde{\nu}$,
\item
case ($B_\R$) corresponds to Proposition~\ref{pr12a} in case $\widetilde{\gr{\nu}}\neq \widetilde{\nu}$,
\item
case ($B_\Z$) corresponds to Proposition~\ref{pr12a1} in case $\widetilde{\gr{\nu}}\neq \widetilde{\nu}$,
\item
case ($A_\R$) corresponds to Proposition~\ref{pr12a} in case $\widetilde{\gr{\nu}}=\widetilde{\nu}$,
\end{itemize}
\paragraph{Case ($A_\Z$)} 

\begin{proof}[Proof of Proposition~\ref{pr12a1} in case $\widetilde{\gr{\nu}}=\widetilde{\nu}$.]
We will use Remark~\ref{zimm} on page~\pageref{zimm} and use the notion of a separating sieve instead of the notion of relative distality.  It is clear that every separating sieve for $(T,\widetilde{\nu})\to(T|_\cA,\nu)$ is a separating sieve for $(T^l,\widetilde{\nu})\to ({T^l}|_\cA,\nu)$.
Let $A_1\supset A_2\supset \dots$ be a separating sieve for $(T^l,\widetilde{\nu})\to ({T^l}|_\cA,\nu)$ and let $X_0\subset X$ be the set of full measure taken from the definition of this separating sieve. Let $\pi$ be the map yielding the factor $\sigma$-algebra $\cA$. Define $\widetilde{X}$ as the intersection $\cap_{k\in\Z}T^k {X_0}$. Take $\widetilde{x},\widetilde{y}\in\widetilde{X}$ with $\pi(\widetilde{x})=\pi(\widetilde{y})$ and suppose that for all $n\in\N$ there exists $k_n\in\Z$ such that $T^{k_n}\widetilde{x},T^{k_n}\widetilde{y}\in A_n$. Let $l_n\in \{0,1,\dots,l-1\}$ be such that $k_n-l_n\in l\Z$. Then
$$
T^{k_n-l_n}T^{l_n}\widetilde{x},T^{k_n-l_n}T^{l_n}\widetilde{y}\in A_n.
$$
Since all $l_n$ belong to the finite set $\{0,1,\dots,l-1\}$ therefore along some subsequence $(n_k)_{k\in\N}$ we have $l_{n_k}$ constant. Moreover $\pi(T^{l_{n_k}}\widetilde{x})=\pi(T^{l_{n_k}}\widetilde{y})$ ($\pi$ is the factoring map) and $T^{l_{n_k}}\widetilde{x},T^{l_{n_k}}\widetilde{y}\in \widetilde{X}$. Our claim follows by the assumption that $A_1\supset A_2 \supset \dots$ is a separating sieve for $(T^l,\widetilde{\nu})\to ({T^l}|_\cA,\nu)$.
\end{proof}
\paragraph{Cases ($B_\R$) and ($B_\Z$)}
\begin{proof}[Proof of Proposition~\ref{pr12a} in case $\widetilde{\gr{\nu}}\neq\widetilde{\nu}$ (and $\gr{\nu}=\nu$)]
By Lemma~\ref{lm:rozkladdokl} the ergodic decomposition of $\widetilde{\gr{\nu}}$ for $T_1$ is of the form
\begin{equation*}
\widetilde{\gr{\nu}}=k\int_0^{1/k}\widetilde{\nu}\circ T_t\ dt
\end{equation*}
for some $k\in\N$ (the measures $\widetilde{\nu}\circ T_t$ are pairwise orthogonal for $t\in[0,1)$ and $\widetilde{\nu}$ is $T_{1/k}$-invariant). Moreover, $\widetilde{\nu}$ is ergodic for $T_{1/k}$ as it is ergodic for $T_1$. By Proposition~\ref{pr12a1} in case ($A_\Z$) proved above, the extension $(T_1,\widetilde{\nu})\to (T_1|_\cA,\nu)$ is distal if and only if the extension $(T_{1/k},\widetilde{\nu})\to (T_{1}|_\cA,\nu)$ is distal. Therefore we may assume without loss of generality that $k=1$ and
\begin{equation}\label{eq:mutu}
\text{the measures }\widetilde{\nu}\circ T_t\text{ are mutually singular for }t\in[0,1).
\end{equation}
Moreover, using Remark~\ref{u:zawie} on page~\pageref{u:zawie}, we may also assume that
\begin{multline}\label{eee}
(X,\cB,{\widetilde{\gr\nu}})=(X,\cB,\widetilde{\nu})\times [0,1],\\
 \text{ with the identification of }(x,1)\text{ with }(T_1x,0)\\
\text{and } T_t(x,s)=(x,s+t) \text{ for }s\in [0,1].
\end{multline}

Suppose that the extension $(T_1,\widetilde{\nu})\to ({T_1}|_\cA,\nu)$ is distal. We will show that also the extension $(\cT,\widetilde{\gr{\nu}})\to ({\cT}|_\cA,\gr{\nu})$ is distal. Since every ergodic factor of a non-ergodic group action is a factor of its almost every ergodic component we have $\pi(x,t)=\pi(x,s)$ for $0\leq t,s<1$, where $\pi$ is the factoring map. Let $X\times \{0\}\supset A_1\supset A_2\supset\dots$ be a separating sieve for $(T_1,\widetilde{\nu})\to ({T_1}|_\cA,\nu)$. Let ${X_0}$ be the set of full measure taken from the definition of the separating sieve for this extension. Take $\widetilde{x}=(x,r),\widetilde{y}=(y,s)\in {X_0}\times [0,1]$. Suppose that for all $n\in\N$ there exists $t_n\in\R$ such that
$$
T_{t_n}\widetilde{x},T_{t_n}\widetilde{y}\in B_n:=A_n\times [0,1/n].
$$
Then
$$
T_{t_n-r_n}\widetilde{x}=T_{t_n-r_n+r}(x,0),T_{t_n-s_n}\widetilde{y}=T_{t_n-s_n+s}(y,0)\in A_n\times \{0\}
$$
for some $r_n,s_n\in [0,1/n]$. It follows by~\eqref{eq:mutu} that $t_n-r_n+r,t_n-s_n+s\in\Z$, whence $r_n-s_n+s-r\in\Z$. Since $r_n-s_n\in [-1/n,1/n]$ and $s-r\in(-1,1)$ therefore there exists $n_0\in\N$ such that for $n\geq n_0$ we have $r_n-s_n+s-r=0$, i.e. $r_n-s_n=r-s=0$. Moreover, since $X\times \{0\}\supset A_1\supset A_2\supset\dots$ is a separating sieve for $(T_1,\widetilde{\nu})\to (T_1,\nu)$, $x=y$. Hence $\widetilde{x}=\widetilde{y}$ and $B_1\supset B_2\supset \dots$ is a separating sieve for $(\cT,\widetilde{\gr{\nu}})\to (\cT,\gr{\nu})$, which ends the first part of the proof.

Suppose now that the extension $(\cT,\widetilde{\gr{\nu}})\to ({\cT}|_\cA,\gr{\nu})$ is distal. We will show that also the extension $(T_1,\widetilde{\nu})\to ({T_1}|_\cA,\nu)$ is distal. By the assumption
$$
\mbox{there exists a separating sieve for }(\cT,\widetilde{\gr{\nu}})\to ({\cT}|_\cA,\gr{\nu}), 
$$
i.e. there exists a set $\widetilde{X}_0\subset X$ with $\widetilde{\gr{\nu}}(\widetilde{X}_0)=1$ and a sequence of sets $(A_n)_{n\in\N}$ with $\widetilde{\gr{\nu}}(A_n)>0$ and $\widetilde{\gr{\nu}}(A_n)\to 0$ such that whenever $x,y\in \widetilde{X}_0$ are such that $\pi(x)=\pi(y)$ and for every $n\in\N$ there exists $t_n\in\R$ with $T_{t_n}x,T_{t_n}y\in A_n$ then $x=y$.

We will use $(A_n)_{n\in\N}$ to find a separating sieve for $(T_1,\widetilde{\nu})\to (T_1,\nu)$. Without loss of generality we may assume that $\widetilde{X}_0$ is invariant under the flow $(T_t)_{t\in\R}$, whence $\widetilde{\nu}(X_0)=1$. 
Let
$$
B_n:=T_{t_n}A_n,\text{ where }t_n\in[0,1)\text{ is such that }\widetilde{\nu}(B_n)>0.
$$
We claim that $\widetilde{X}_0$ and $(B_n)_{n\in\N}$ satisfy the assumptions of Lemma~\ref{lm:sita}. Indeed, let $x,y\in\widetilde{X}_0$ be such that $\pi(x)=\pi(y)$ and for all $n\in\N$ there exists $k_n\in\Z$ with $T_{k_n}x,T_{k_n}y\in B_n$. Then
$$
T_{k_n-t_n}x,T_{k_n-t_n}y\in A_n,
$$
which implies that $x=y$. To end the proof it suffices to notice that the measure $\widetilde{\nu}$ is non-atomic ($\widetilde{\gr{\nu}}=\int_0^1 \widetilde{\nu}\circ T_t \ dt$ and $\widetilde{\gr{\nu}}$ is finite) and use Lemma~\ref{lm:sita}.
\end{proof}

In case ($B_\Z$) the proofs are the same as in case ($B_\R$).

\paragraph{Case ($A_\R$)}

\begin{proof}[Proof of Proposition~\ref{pr12a} in case $\widetilde{\gr{\nu}}=\widetilde{\nu}$]

It is clear that every separating sieve for $(\cT,\widetilde{\nu})\to ({\cT}|_\cA,\nu)$ is a separating sieve for $(T_1,\widetilde{\nu})\to ({T_1}|_\cA,\nu)$. Using Remark~\ref{zimm} we obtain the first part of the claim.

Suppose now that the extension $(T_1,\widetilde{\nu})\to ({T_1}|_\cA,\nu)$ is isometric. By the proof of Proposition~\ref{Tzwarte} (see~\cite{MR1930996}) it follows that there exists $\kappa\in \mathcal{J}^e_\infty \left((T_1,\widetilde{\nu});\cA \right)$ such that 
$$
(T_1^{\times \infty},\kappa)\to ({T_1}|_\cA,\nu)\text{ is a compact group extension.}
$$
On the space $X^{\times \infty}$ we have a natural action of $\cT^{\times \infty}$. It is however not necessarily true that the measure $\kappa$ is $\cT^{\times \infty}$-invariant. Hence we consider two cases:
\begin{multicols}{2}
\begin{itemize}
\item[(i)]
$\kappa$ is $\cT^{\times \infty}$-invariant,
\item[(ii)]
$\kappa$ is not $\cT^{\times\infty}$-invariant.
\end{itemize}
\end{multicols}
In case (i) we have the following two possibilities:
\begin{itemize}
\item[(i1)]
$\mathcal{J}^e_2((\cT^{\times\infty},\kappa);\cA)\subset \mathcal{J}^e_2((T_1^{\times \infty},\kappa);\cA)$,
\item[(i2)]
$\mathcal{J}^e_2((\cT^{\times\infty},\kappa);\cA)\not\subset \mathcal{J}^e_2((T_1^{\times \infty},\kappa);\cA)$.
\end{itemize}
In case (i1) by Proposition~\ref{pr:veech1} and Proposition~\ref{pr:veech2} we obtain that $(\cT^{\times\infty},\kappa)\to ({\cT}|_\cA,\nu)$ is a compact group extension. Therefore $(\cT,\widetilde{\nu})\to ({\cT}|_\cA,\nu)$ is an isometric extension.
In case (i2) take 
$$
\gr{\xi} \in \cJ^e_2((\cT^{\times \infty},\kappa);\cA)\setminus \cJ^e_2((T_1^{\times \infty},\kappa);\cA).
$$
Then
$$
\gr{\xi}=\int_0^1 \xi \circ \left(T_t^{\times \infty}\right)^{\times 2}\ dt, \text{ where }\xi\in \cJ^e_2((T_1^{\times \infty},\kappa);\cA).
$$
By Proposition~\ref{pr:veech1} the measure $\xi$ is a graph self-joining, i.e.
$$
(T_1^{\times \infty}\times T_1^{\times \infty},\xi)\simeq (T_1^{\times \infty},\kappa).
$$
Therefore
$$
(T_1^{\times \infty}\times T_1^{\times \infty},\xi) \to (T_1,\nu) \text{ is a compact group extension,}
$$
i.e. it is distal. It follows from Proposition~\ref{pr12a} in case ($B_\R$) proved above, that 
$$
(\cT^{\times \infty}\times \cT^{\times \infty},\gr{\xi}) \to ({\cT}|_\cA,\nu)\text{ is also distal.}
$$
Hence $(\cT,\widetilde{\nu})\to ({\cT}|_\cA,\nu)$ is also distal.

In case (ii) consider
$$
\gr{\kappa}:=\int_0^1 \kappa \circ T_t^{\times \infty}\ dt.
$$
By Proposition~\ref{pr12a} in case ($B_\R$) it follows that $(\cT^{\times \infty},\gr{\kappa})\to (\cT|_\cA,\nu)$ is distal, whence $(\cT,\widetilde{\nu})\to (\cT|_\cA,\nu)$ is also a distal extension.

If the extension $(T_1,\widetilde{\nu})\to ({T_1}|_\cA,\nu)$ is distal (and not isometric), it suffices to consider the maximal isometric extensions contributing to $(T_1,\widetilde{\nu})\to ({T_1}|_\cA,\nu)$ and use a similar ``maximizing'' procedure as in the proof of Proposition~\ref{pr12a} (page~\pageref{pr12a}).
\end{proof}

\subsection{Embeddability and 2-QS (special case) -- another proof}
We will give now an independent proof of the fact that 2-fold simplicity of an ergodic time automorphism of a flow implies that the whole flow is 2-fold quasi-simple.
\begin{lm}\label{lm:zawieszenie-skosny}
Let $\cT=(T_t)_{t\in\R}$ be an ergodic flow on $(X,\cB,\mu)$. Consider 
$$
S_t\colon (X,\cB,\mu)\otimes ([0,1),\lambda_{[0,1)})\to (X,\cB,\mu)\otimes ([0,1),\lambda_{[0,1)})
$$
which is a suspension over $T_1$. Then $\mathcal{S}$ is isomorphic to a skew product of $\cT$ with a rotation on the circle.
\end{lm}
\begin{proof}
Let the flow $\mathcal{S}=(S_t)_{t\in\R}\colon X\times[0,1)\to X\times[0,1)$ be given by$
S_t(x,r)=(T_{t}y,\{r+t\})$.
Then the map $\Psi\colon X\times[0,1)\to X\times[0,1)$ given by
$\Psi\colon (x,r)\mapsto (T_rx,r)$
yields an isomorphism between $\cT$ and $\cS$.
\end{proof}

\begin{pr}\label{pr2qs}
Let $\cT=(T_t)_{t\in\R}$ be a flow on $\xbm$ with $T_1$ ergodic. If $T_1$ is 2-fold simple then $\cT$ is 2-QS. In fact, for every $\gr{\xi}\in \cJ_2^e(\cT)\setminus \cJ_2^e(T_1)$ the flow $(\cT\times \cT,\gr{\xi})$ is isomorphic to a suspension flow over $(T_1,\mu)$.
\end{pr}

\begin{proof}
Consider $\gr{\xi}\in \cJ_2^e(\cT)\setminus\{\mu\otimes \mu\}$. There exists $\xi\in \cJ_2^e(T_1)$ such that $\gr{\xi}=\int_0^1 \xi\circ (T_t\times T_t)\ dt$. Without loss of generality we may assume that this is the ergodic decomposition of $\gr{\xi}$ for $T_1\times T_1$. Then $\gr{\xi}$-almost every point of the space $X\times X$ is of the form $(T_rx,T_rSx)$, where $r\in [0,1)$ and $S\in C(T_1)$. Define $\Phi\colon X\times X \to X\times \T$ for such points by
$$
\Phi(T_rx,T_{r}Sx):=(x,r).
$$
This map is 1-1 $\gr{\xi}$-almost everywhere. Moreover, for Borel sets $A\subset X$, $I\subset \T$ we have
$$
\mu\otimes \lambda_\T(A\times I)=\mu(A)\cdot \lambda_\T(I).
$$
On the other hand
$$
\Phi^{-1}(A\times I)=\{ (T_rx,T_rSx)\colon x\in A,r\in I\},
$$
whence by the definition of $\gr{\xi}$ we obtain $\gr{\xi}(\Phi^{-1}(A\times I))=\mu(A)\cdot \lambda_\T(I)$. Let $S_t\colon X\times \T\to X\times \T$ be given by $S_t(x,r)=(T_{[t+r]}x,\{r+t\})$. Then
$$
\Phi \circ (T_t\times T_t) = S_t\circ \Phi \text{ for all }t\in\R,
$$
i.e. $\cT\times\cT\simeq\mathcal{S}$. The claim follows by Lemma~\ref{lm:zawieszenie-skosny}.
\end{proof}

\subsection*{Acknowledgements}
Research supported by Narodowe Centrum Nauki grant \linebreak DEC-2011/03/B/ST1/00407.

\footnotesize
\bibliography{cala}

\noindent
Joanna Ku\l aga-Przymus\\
\textsc{Institute of Mathematics, Polish Acadamy of Sciences,\\
\'{S}niadeckich 8, 00-956 Warszawa, Poland}\\
and\\
\textsc{Faculty of Mathematics and Computer Science, \\
Nicolaus Copernicus University,\\
Chopina 12/18, 87-100 Toru\'{n}, Poland}\par\nopagebreak
\noindent
\textit{E-mail address:} \texttt{joanna.kulaga@gmail.com}

\end{document}